\documentclass[12pt, reqno]{amsart}
\usepackage[margin=0.78in]{geometry}

\usepackage[english]{babel}
\usepackage{amsmath,amssymb,amsfonts,amsthm,enumerate}
\usepackage{url}
\usepackage{graphicx,epstopdf,color}

\numberwithin{equation}{section}

\newtheorem{theorem}{Theorem}[section]
\newtheorem{lemma}[theorem]{Lemma}

\newtheorem{remark}[theorem]{Remark}
\newtheorem{proposition}[theorem]{Proposition}
\newtheorem{corollary}[theorem]{Corollary}

\newcommand{\R}{\mathbb{R}}
\newcommand{\Z}{\mathbb{Z}}
\newcommand{\N}{\mathbb{N}}

\newcommand{\QQQ}{\mathcal{Q}}

\renewcommand{\epsilon}{\varepsilon}
\newcommand{\eps}{\varepsilon}

\renewcommand{\le}{\leqslant}
\renewcommand{\geq}{\geqslant}
\renewcommand{\ge}{\geqslant}

\renewcommand{\emptyset}{\varnothing}

\title{Improvement of flatness for nonlocal phase transitions}
\thanks{This work has been supported by Alexander von Humboldt Foundation and ERC grant 277749 ``EPSILON Elliptic PDE's and Symmetry of Interfaces and Layers for Odd Nonlinearities". 
It is a pleasure to thank Alberto Farina
for his comments on a previous version
of this manuscript.}

\author{Serena Dipierro}
\address[Serena Dipierro]{
Dipartimento di Matematica, 
Universit\`a degli studi di Milano,
Via Saldini 50, 20133 Milan, Italy,
and
School of Mathematics and Statistics,
University of Melbourne,
Peter Hall Building,
Parkville VIC 3010,
Australia, and
School of Mathematics and Statistics,
University of Western Australia,
35 Stirling Highway
Crawley, Perth
WA 6009,
Australia.}
\email{sdipierro@unimelb.edu.au}
\author{Joaquim Serra}
\address[Joaquim Serra]{Weierstra{\ss} Institut f\"ur Angewandte Analysis
und Stochastik, Mohrenstra{\ss}e 39, 10117 Berlin, Germany, and
Eidgen\"ossische Technische Hochschule Z\"urich,
R\"amistrasse 101,
8092 Zurich,
Switzerland}
\email{joaquim.serra@upc.edu}

\author{Enrico Valdinoci}
\address[Enrico Valdinoci]{School of Mathematics and Statistics,
University of Melbourne,
Peter Hall Building,
Parkville VIC 3010,
Australia, and
School of Mathematics and Statistics,
University of Western Australia,
35 Stirling Highway
Crawley, Perth
WA 6009,
Australia, and
Weierstra{\ss} Institut f\"ur Angewandte Analysis
und Stochastik, Mohrenstra{\ss}e 39, 10117 Berlin, Germany,
and Dipartimento di Matematica, 
Universit\`a degli studi di Milano,
Via Saldini 50, 20133 Milan, Italy, and
Istituto di Matematica Applicata e Tecnologie Informatiche,
Consiglio Nazionale delle Ricerche,
Via Ferrata 1, 27100 Pavia, Italy}
\email{enrico@mat.uniroma3.it}

\begin{document}

\begin{abstract}
We establish an improvement of flatness result for critical points of Ginzburg-Landau energies with long-range interactions. It applies in particular to solutions of  $(-\Delta)^{s/2} u = u-u^3$ in $\R^n$ with~$s\in(0,1)$. 
As a corollary, we establish that  solutions with asymptotically 
flat level sets are $1$D  and prove the analogue of the De Giorgi conjecture (in the setting of minimizers) in dimension $n=3$ for all $s\in(0,1)$ and  in dimensions  $4\le n\le 8$ for  $s\in(0,1)$ sufficiently close to $1$.

The robustness of the proofs, which do not rely on the extension of Caffarelli and Silvestre,  allows us to include anisotropic functionals in our analysis.

Our improvement of flatness result holds for all solutions, and not only minimizers. This cannot be achieved in the classical case $-\Delta u = u-u^3$ 
(in view of the solutions bifurcating from catenoids constructed in~\cite{CAT1}).


\end{abstract}


\maketitle

\tableofcontents

\section{Introduction}

\subsection{Ginzburg-Landau energy with long range interactions}

The paper is concerned with critical points of the following Ginzburg-Landau energy with long range interactions
\[
J(v) := \frac{1}{4}\iint_{\R^n\times\R^n}  \frac{| v(x)-v(z)|^2}{\|x-z\|_{K}^{n+s}}\,dx\,dz 
+ \int_{\R^n} W(v)\,dx.
\]
Here, $K$ is an even\footnote{For~$K$ being even, as customary,
we mean that~$\{-x: x\in K\} =K$.} and $C^{1,1}$ convex body, $\|\cdot\|_K$ denotes the norm of $\R^n$ with unit ball $K$, and $W$ is a double-well potential.

Such energies naturally arise in
several contexts, such as phase transitions, atom dislocations in crystals,
mathematical biology, etc. (see e.g. 
Section~2 in~\cite{giampi}, the Appendix in~\cite{DDDV},
the Introduction in~\cite{BIO}, and also~\cite{bucur} and the references therein
for a series of motivations
under different perspectives).\medskip

We establish a improvement of flatness result for the level sets of critical points $J$. That is, for solutions  $u$ of
\begin{equation}\label{EQ}
L u = f(u)  \quad \mbox{in }\R^n,
\end{equation}
where~$L$ is an elliptic scaling invariant operator  of order~$s\in(0,1)$, of the form
\begin{equation}\label{formL1}
 Lu(x):= \int_{\R^n} \frac{u(x)-u(x+y)}{\|y\|_{K}^{n+s}}\,dy.
\end{equation}
Throughout the paper $f = -W'$ will be the derivative of the double-well potential.
Note that the case of the fractional Laplacian $L=(-\Delta)^{s/2}$ corresponds to $K$ being a ball ---and thus $\|\cdot\|_{K} = |\,\cdot\,|$.

\subsection{Large scale behavior and De Giorgi conjecture}

If $v$ is a  minimizer of $J$ in $\R^n$ (locally) then $v_{\eps}(x)= v(x/\eps)$ is a minimizer of 
\[
J_\eps(v) := \frac{1}{4}\iint_{\R^n\times\R^n}  \frac{| v(x)-v(z)|^2}{\|x-z\|_{K}^{n+s}}\,dx\,dz + \eps^{-s}\int_{\R^n} W(v)\,dx
\]
and solves the equation 
\[L v_{\eps} = \eps^{-s}f(v_{\eps}).\]  

By the methods introduced in \cite{CSV} for the analysis of nonlocal phase transitions, one can prove that $\| \nabla v_{\eps}\|_{L^1(B_R)} \le CR^{n-1}$ for all $R\ge 1$, with $C$ depending only on $n,s, K$. In particular, one can show that
(up to subsequence)  
\begin{equation}\label{limitE}
v_{\eps} \rightarrow \chi_{E}-\chi_{E^c} \quad \mbox{ in  } L^1_{\rm loc}(\R^n)
\end{equation} where $E$ is a 
minimizer of a fractional perimeter. More precisely, when $K$ is a ball one obtains the isotropic fractional perimeter (Caffarelli, Roquejoffre, and Savin  \cite{CRS}) while for other $K$ one obtains the anisotropic fractional perimeters (Ludwig in \cite{Ludwig}).

In the range $s\in[1,2]$, in the isotropic setting one still has \eqref{limitE} but then $E$ is known to be a minimizer of the classical perimeter. 
This was proven in~\cite{MM, SaVgamma} through $\Gamma-$convergence results. 

We thus see that there is a striking difference between the two regimes $s\in(0,1)$ and $s\in[1,2]$ 
as far as asymptotic behavior at large scales is concernd.
We use the wording {\em genuinely nonlocal regime} to refer to the case $s\in(0,1)$ because the long-range interactions survive in the asymptotic limit.

For $s=2$, the link between minimizers of the Ginzburg-Landau energy and minimizers 
of the classical perimeter motivates a famous conjecture of  Ennio De~Giorgi ~\cite{DG}. 
This conjecture states that 
``every bounded solution of $-\Delta u = u-u^3$ in $\R^n$ that is monotone in one variable, 
say $\partial_{x_n}u>0$, is $1$D in dimension $n\le 8$. Namely, its level sets are parallel hyperplanes''.

The threshold $n=8$ is related the classical results on entire minimal graphs:  affine functions (hyperplanes) are the only entire minimal graphs up to dimension $n=8$ (see~\cite{SIMO}) while  non-affine examples can be found in dimensions $n=9$ or higher 
(see~\cite{BOMB}).

Positive answers to De Giorgi conjecture have been established for $n=2$ in  \cite{GG}, $n=3$ in \cite{AC, AAC} and, in the setting of minimizers, for $4\le n\le 8$ in \cite{Savin}. A non- $1$D example was constructed in \cite{DKW} for $n= 9$. See also the
excellent survey~\cite{Savin-Survey} for the history of the conjecture and the known results.
\medskip

To prove the conjecture (in the setting of minimizers) for $4\le n\le 8$, Savin established in the celebrated paper \cite{Savin} that
\begin{equation}\label{asymp}
\left.
\begin{array}{c}
\mbox{ $u$ is a minimizing
 solution of $-\Delta u = u-u^3$,} 
\\
\mbox{and the level sets of $u$  are asymptotically flat}
\end{array}
\right\} 
\quad \Rightarrow \quad 
\mbox{$u$ is $1$D. }
\end{equation}
Here, the word ``minimizing'' refers to the associated energy
$$ \int_{\R^n}\frac{1}{2}|\nabla u| +\frac{1}{4}(1-u^2)^2\,dx$$ while  ``the level sets of $u$  are asymptotically flat''  means that $\{u_\epsilon=\theta\}$ converges uniformly on compact sets to a hyperplane for all $\theta\in(-1,1)$ as $\eps\downarrow0$.

Since, as explained above, for $s\in[1,2]$ the Ginzburg-Landau energy behaves asymptotically like the classical perimeter, it is natural to conjecture that the  statement of De Giorgi holds also true for  $(-\Delta)^{s/2} u = u-u^3$ whenever $s\in[1,2]$.

In this direction the cases $s\in[1,2]$ are currently as well understood as the case $s=2$ ---only the construction of a counterexample for $n=9$ is missing to have a full parallelism of results. These results have been obtain in ~\cite{CSM, SirV, CCi1, CCi2, MONOO, Onew2}. In particular, 
the analogue of \eqref{asymp} for  $s\in(1,2)$ has been recently established by Savin in \cite{Onew2} where also the important case of the half Laplacian $s=1$ has been announced.


For $s\in(0,1)$ it is natural to expect a analogue of the De Giorgi conjecture in sufficiently low dimensions. The heuristic giving $n=8$ as a critical dimension is only valid in the case $s\in[1,2]$ and the $1$D symmetry up to dimension $n=8$ is not expected for all $s\in(0,1)$ but just for $s$ sufficiently close to $1$.  
Despite of several works in that direction, up to now a positive result to the 
conjecture for $s\in(0,1)$ was only known in dimension $n=2$ ---as established in \cite{SirV, cab}.

In this paper we prove  ---see the forthcoming Theorem \ref{C:1}---
\begin{equation}\label{asymp2}
\left.
\begin{array}{c}
\mbox{ $u$ is a solution of  $L u = u-u^3$ with $s\in(0,1)$,} 
\\
\mbox{and the level sets of $u$  are asymptotically flat}
\end{array}
\right\} 
\quad \Rightarrow \quad 
\mbox{$u$ is $1$D. }
\end{equation}

As a consequence, we establish the De Giorgi conjecture (in the minimizer setting)  in the following cases
\begin{itemize}
\item in dimension $n=3$,  for all $s\in(0,1)$
\item in dimension $4\le n\le 8$ for $s\in (0,1)$ sufficiently close to $1$.
\end{itemize}

Let us stress a fundamental difference between the classical case in~\eqref{asymp} ---or similarly the cases $s\in[1,2)$---
and the nonlocal case in~\eqref{asymp2}. Namely, the results for $s\in[1,2]$ are for minimizers while our result for $s\in(0,1)$ holds in the more general setting of solution (critical points). This is a feature of the  genuinely nonlocal regime $s\in(0,1)$ that is not expected to be true in the case $s\in[1,2]$ (in view of the solutions bifurcating from catenoids constructed in~\cite{CAT1} for $s=2$).

In the cases $s\in [1,2]$, the implication \eqref{asymp} follows as a direct consequence of an important improvement of flatness result for level sets of solutions to $(-\Delta)^{s/2} u = u-u^3$. This result is in the same spirit of the one of De Giorgi for classical minimal surfaces. 
Similarly,  for $s \in(0,1)$ the implication \eqref{asymp2} follows
from an improvement of flatness result for solutions of $Lu =f(u)$, stated next in Subsection \ref{SUB13}.
As we will see, however, in the case $s \in(0,1)$ the improvement of flatness does not yield as a direct consequence \eqref{asymp2} as for  $s\in [1,2]$.
\medskip

Before stating our main results let us make quantitative versions of our assumptions.
\subsection{Quantitative assumptions on $L$ and $f$}\label{SUB12}

We assume that the convex set $K$ defining the operator $L$ satisfies
\begin{equation}\tag{H1}\label{assumpK}
\mbox{$K\subset B_1$ and each point of $\partial K$ can be touched by a ball of radius $r_K >0$ contained in $K$.}
\end{equation}
This is a quantitative version of $K$ being $C^{1,1}$.

We assume that $f$ belongs to $C^1\big([-1,1]\big)$ 
and satisfies, for some $\kappa>0$ and $c_\kappa>0$,  
\begin{equation}\tag{H2}\label{assumpf}
f(-1)=f(1)=0 \quad \mbox{and}\quad f'(t)<-c_\kappa 
\quad \mbox{for } t\in [-1,-1+\kappa]\cup[1-\kappa, 1].
\end{equation}

Moreover, we assume that
\begin{equation}\tag{H3}\label{existslayer}
\mbox{there exists $\phi_0$ satisfying}\quad  
\begin{cases}
\mathcal L\phi_0= f(\phi_0) \  &\mbox{in }\R,
\\
\phi_0'>0 &\mbox{in }\R,
\\
\phi_0(0) =0,
\\
\displaystyle\lim_{x\to \pm\infty} \phi_0 = \pm 1,
\end{cases}
\end{equation}
where $\mathcal L$ denotes (here and throughout the paper)
the fractional Laplacian in dimension one (without normalization constant)--- see \eqref{1dfraclap}.\medskip

We remark
that assumption  \eqref{assumpf} and \eqref{existslayer} are satisfied when $f=-W'$, 
with $W$ being a $C^2$ double-well potential with wells (i.e. minima) at $\pm 1$ and satisfying that $W''>0$ near $\pm 1$. 
Indeed, the existence a one-dimensional heteroclinic solution is proven  in \cite{PSV, cab} (see also \cite{cozzi}
for the case of general kernels) and thus  \eqref{existslayer} is satisfied. \medskip

The constants in the estimates will also depend on 
\begin{equation}\label{lkappa}
l_\kappa := \inf \big\{\, l>0\ : \ \phi_0\big([-l,l]\big).
\supset [-1+\kappa, 1-\kappa]\big\}.\end{equation}
Note that $l_\kappa$ is  (half of) the length of the symmetric interval where the transition of $\phi_0$ essentially occurs.

\subsection{Improvement of flatness result}\label{SUB13}
In the framework that we have just introduced, we are now in the position
of stating our main result as follows. 

Throughout the paper, we call a constant {\em universal} if it depends only on~$n$, 
$s$, $r_K$ , $\kappa$, $c_\kappa$ and $l_\kappa$, see Subsections~\ref{SUB11} and~\ref{SUB12}.  
In particular, universal constants depend only on $n$, $L$, and $f$.

In the statement of the next theorem, for fixed $\alpha_0>0$, given $a\in(0,1)$ we define
\begin{equation}\label{ja}
j_a := \left\lfloor  \frac{\log a}{\log(2^{-\alpha_0})} \right\rfloor .
\end{equation}
Note that $j_a$ is a nonnegative integer and that $2^{\alpha_0 j_a}$ is comparable to $1/a$. 
 
\begin{theorem}\label{thm:improvement}
Let $s\in(0,1)$.
Assume that $L$ satisfies $\eqref{assumpK}$ and that $f$ satisfies \eqref{assumpf} and \eqref{existslayer}.
Then there exist universal constants~$\alpha_0\in (0,s/2)$, $
p_0\in(2,\infty)$ and~$a_0\in(0,1/4)$ such that the following statement holds.
\smallskip

Let $a\in(0,a_0]$ and $\eps \in (0, a^{p_0}]$. 
Let $u: \R^n \rightarrow (-1,1)$ be a solution 
of $Lu=\varepsilon^{-s}f(u)$ in $B_{2^{j_a}}$ such 
that $0\in \{-1+\kappa \le u \le  1-\kappa\}$
and
\[ 
\{\omega_j \cdot x\le -a 2^{j(1+\alpha_0)}\}\, \subset\, \{u\le -1+\kappa\} \,\subset \,\{u\le 1-\kappa\}  \,\subset \,\{\omega_j\cdot x\le a 2^{j(1+\alpha_0)}\} \quad \mbox{in }B_{2^j}, 
\]
for $0\le j\le j_a$, where  $\omega_j\in S^{n-1}$.

Then, 
\[ 
\left\{\omega\cdot x \le - 
\frac{a}{2^{1+\alpha_0}} \right\} \,\subset \,\{u\le -1+\kappa\} \,\subset\, \{u\le 1-\kappa\}  \,\subset\, \left\{\omega\cdot x\le  \frac{a}{2^{1+\alpha_0}} \right\} \quad \mbox{in }B_{1 /2},
\]
for some $\omega\in S^{n-1}$.
\end{theorem}

\medskip

In order to explain more intuitively
of Theorem~\ref{thm:improvement}, let us introduce 
some (informal) terminology.
We call {\em transition level sets} (of~$u$) 
all the level sets~$\{u=\theta\}$ 
for~$\theta \in(-1+\kappa, 1-\kappa)$.
We say that the transition level sets are flat at a scale~$R$ 
if they are trapped, after some rotation,
in a cylinder~$B'_R\times (-aR, aR)$. 
We call {\em flatness} the adimensional quantity $a$.

With this terminology, Theorem~\ref{thm:improvement} 
says that if the transition level sets are flat enough at 
a very large scale, then its flatness improves geometrically 
at smaller scales.
However, as we will see in more detail in Section~\ref{SECT7}, 
the geometric improvement of the flatness does not hold up to scale 1 but only up to some (still very large) mesoscale. 
This is because we need to assume $\eps\le  a^{p_0}$ with $p_0$ large and not just  $\eps\le  ca$. 
This is related to the fact that the $1$D solution $\phi_0$ from \eqref{existslayer} decays to $\pm 1$ when $x\to \pm \infty$ only at a slow algebraic rate comparable to
$|x|^{-s}$.
We will comment more on this important difference with respect to~\cite{Savin} later on.

Theorem~\ref{thm:improvement} can be also understood as an approximate $C^{1,\alpha}$ regularity result
for level sets. 
Namely, if the transition level sets of the solution of~$L u = \eps^{-s}f(u)$ in $B_1$ 
are trapped between two parallel planes close enough to the origin, 
and~$\epsilon$ is small enough,  then the transition occurs 
essentially on a $C^{1,\alpha}$ graph in $B_{1/2}$ up to errors 
that decay algebraically (in $\eps$) as $\eps\downarrow 0$. 
The limit case 
as $\eps\downarrow 0$ of this result 
plays a crucial role in the regularity theory of nonlocal minimal surfaces; 
see  Theorem~6.8 in~\cite{CRS}.

An analogue of Theorem~\ref{thm:improvement} for  $s\in(1,2)$ has been obtained very recently by Savin in \cite[Theorem 6.1]{Onew2} by using a robust version  \cite{Onew1} of the original proof in \cite{Savin}.  The important case of the half Laplacian ($s=1$) turns out to be a borderline case for the method in \cite{Onew1, Onew2}, and a similar improvement of flatness result for $s=1$ has been announced also in \cite{Onew2}. 
Despite of the analogy in the statements,  there exist fundamental difference between Theorem~\ref{thm:improvement}
and Theorem 6.1 of \cite{Onew2}. Indeed, 
\begin{enumerate}
\item Theorem~\ref{thm:improvement} is for solutions (not necessarily minimizers)
\item  $\epsilon \le a^{p_0}$  is assumed Theorem~\ref{thm:improvement} which is stronger than $\epsilon \le ca$  in  \cite{Onew2} (the result for $s\in(0,1)$ is probably not true under the assumption $\epsilon \le ca$ due to the very slow decay of $u$ to $\pm1$)
\item Theorem~\ref{thm:improvement}  includes anisotropic functionals (in particular our proofs do not make us of the Caffarelli Silvestre extension).
\end{enumerate}

\subsection{$1$D symmetry of asymptotically 
flat solutions}\label{SE:14}

An important application Theorem~\ref{thm:improvement}
is the following rigidity result for solutions  in the whole space with asymptotically flat level sets.

We say that a function~$u:\R^n\to\R$ is $1$D  if there exist~$\bar u:\R\to\R$
and~$
\bar\omega\in S^{n-1}$ such 
that~$u(x)=\bar u(\bar\omega\cdot x)$ for any~$x\in\R^n$.


\begin{theorem}[One-dimensional symmetry 
for asymptotically flat solutions]\label{C:1}
Let $s\in(0,1)$. Assume that $L$ satisfies $\eqref{assumpK}$ and that $f$ 
satisfies \eqref{assumpf} and \eqref{existslayer} and let $u$ be a solution of~$
L u = f(u)$ in~$\R^n$. 

Assume that there exist $R_0\ge 1$ 
and $a:(R_0,+\infty) \rightarrow (0,1]$ such 
that~$a(R)\downarrow 0$ as $R\uparrow+\infty$ and 
such that, for all $R>R_0$, we have
\begin{equation}\label{ASS-R}
\{ \omega\cdot x\le -a(R)R\} \subset \{u\le -1+\kappa\}\subset  \{u\le 1-\kappa\} \subset \{ \omega\cdot x\le a(R)R\} \quad \mbox{in }B_{R},
\end{equation}
for some $\omega\in S^{n-1}$,
which may depend on $R$.

Then, $u$ is $1$D.
\end{theorem}

A similar result for~$s\in(1,2)$ is given in \cite[Theorem 1.1]{Onew2}.
In \cite{Onew2}, this asymptotic result is a direct application of the improvement of flatness result \cite[Theorem 6.1]{Onew2} and rescaling. 
In our case,  Theorem~\ref{C:1} 
will still be a consequence of Theorem~\ref{thm:improvement},
but not an immediate one.  Indeed:
\begin{itemize}
\item In~\cite{Savin, Onew2}  the improvement of flatness only requires $\eps\le ca$, and can be iterated from a large scale $R$ all the way up to scale~$1$ (thus giving the rigidity when letting $R\to \infty$). 
\item In contrast, Theorem~\ref{thm:improvement} requires $\eps\le a^{p_0}$ and thus 
we can only improve the flatness geometrically from a large ball~$B_R$ up to a still large mesoscopic ball~$B_r$ with~$r=R^{1-\delta}$.
\end{itemize}

Due to this, the proof of Theorem~\ref{C:1} becomes more interesting, and
requires a suitable multiscale iteration 
of Theorem~\ref{thm:improvement}, 
combined with the use of the sliding method of 
Berestycki, Caffarelli and Nirenberg~\cite{BCN, BCN2} in its full strength.
See Subsection~\ref{sec:proofs} 
for further details on the proofs.
\medskip





\subsection{Application to the De Giorgi conjecture for $s\in(0,1)$}

Let us now
consider the concrete case of minimizing solutions
of the nonlocal Allen-Cahn equation~$(-\Delta)^{s/2} u=u-u^3$,
with~$s\in(0,1)$. We remark that the problem is variational, with associate energy functional given by
\begin{equation}\label{energyfunctional}
 {\mathcal{E}}(u,\Omega):={\mathcal{E}^{\rm Dir}}(u,\Omega)
+\int_\Omega (1-u^2(x))^2\,dx,
\end{equation}
where, for some appropriate constant $C_{n,s}>0$,
\begin{equation}\label{KIN}
{\mathcal{E}^{\rm Dir}}(u,\Omega):= C_{n,s}\iint_{\R^{2n}\setminus(\R^n\setminus\Omega)^2}
\frac{|u(x)-u(y)|^2}{|x-y|^{n+s}}\,dx\,dy.\end{equation}

We say that a solution $u$ of~$(-\Delta)^{s/2} u=u-u^3$ is a {\em minimizer} of $\mathcal E$
in~$\R^n$ if
$$ {\mathcal{E}}(u,B)\le {\mathcal{E}}(u+\varphi,B),$$
for any ball~$B\subset\R^n$ and any~$\varphi\in C^\infty_0(B)$
(notice that, for simplicity, we are dropping the normalization constant in the
fractional Laplace framework).

In this setting, we have:

\begin{theorem}[One-dimensional symmetry in the plane]\label{C:2}
Let~$u$ be a minimizer of $\mathcal E$ 
in~$\R^2$.

Then, $u$ is  $1$D.
\end{theorem}

Theorem~\ref{C:2} has been also proved, by different methods,
in~\cite{cab, SirV}. On the other hand,
the following results are, as far as we know,
completely new, since they deal with higher-dimensional spaces
(indeed, the only symmetry results known for the fractional
Allen-Cahn equation are the ones
in~\cite{CCi1, CCi2}, which
hold in dimension~$n=3$ with~$s\in[1,2)$,
while we will consider now the case~$n\ge3$ and~$s\in(0,1)$,
under different assumptions).

\begin{theorem}[One-dimensional symmetry for monotone solutions
in $\R^3$]\label{C:3bis}
Let~$n\le3$ and~$u$ be a solution of~$(-\Delta)^{s/2} u=u-u^3$
in~$\R^n$.

Suppose that
$$ \frac{\partial u}{\partial x_n}(x)>0\quad{\mbox{ for any }}x\in\R^n$$
and
$$ \lim_{x_n\to\pm \infty} u(x',x_n)=\pm1.$$
Then, $u$ is $1$D.
\end{theorem}

Theorem~\ref{C:3bis} has also been recently exploited in~\cite{ESSO}
in order to obtain additional results of De Giorgi type.

The next two results deal with the case in which the fractional parameter~$s$
is sufficiently close to~$1$ (that is, roughly speaking, when the nonlocal
diffusive operator is sufficiently close to~$\sqrt{-\Delta}$).
In this case, it is known that the minimizers of the corresponding geometric
problem of fractional perimeters are close to the classical minimal surfaces (see~\cite{CafVal}). This fact provides an additional rigidity of
the interfaces that we can exploit in order to obtain symmetry results.

\begin{theorem}[One-dimensional symmetry when $s$ is close to 1]\label{C:3}
Let~$n\le 7$. Then, there exists $\eta_n\in(0,1)$ such that
for any~$s\in [1-\eta_n,\,1)$ the following statement holds true.

Let~$u$ be a minimizer of $\mathcal E$ 
in~$\R^n$. Then, $u$ is $1$D.
\end{theorem}

\begin{theorem}[One-dimensional symmetry for monotone solutions
in $\R^8$ when $s$ is close to~$1$]\label{C:3tris}
Let $n\le8$. Then, there exists $\eta_n\in(0,1)$ such that
for any~$s\in [1-\eta_n,\,1)$ the following statement holds true.

Let~$u$ be a solution of~$(-\Delta)^{s/2} u=u-u^3$ in~$\R^n$.

Suppose that             
\begin{equation}
\label{LIM1}
\frac{\partial u}{\partial x_n}(x)>0\quad{\mbox{ for any }}x\in\R^n\end{equation}
and
\begin{equation}
\label{LIM2} \lim_{x_n\to\pm \infty} u(x',x_n)=\pm1.\end{equation}
Then, $u$ is $1$D.
\end{theorem}

\subsection{Overview of the proofs and 
organization of the paper}\label{sec:proofs}

The proof of Theorem \ref{thm:improvement} 
follows the classical\footnote{It goes back to De~Giorgi; see e.g.
the retrospective in~\cite{MR3393314}} ``improvement of flatness
strategy''  that
was pioneered by Savin in \cite{Savin} 
for the case of level sets of
classical phase transitions.
The same general approach was suitably modified in~\cite{CRS} in the 
context of nonlocal minimal surfaces.
Let us give next the ``big picture'' of it in order to explain the structure of the paper ---we will assume here for simplicity $L=(-\Delta)^{s/2}$.

Very roughly, we take a sequence~$u_a$ of solutions 
of $(-\Delta)^{s/2}u_a =\eps^{-s}f(u_a)$ such that the transition level sets of $u_a$ are trapped in a very flat cylinder\footnote{An additional geometric trapping condition in dyadic balls  up to a certain larger scale is also required but this is omitted in this rough exposition } $\{|x'| \le 1, |x_n|\le a\}$. We assume that $\eps<{a^{p_0}}$ for $p_0$ large and
we show that $u_a\approx \pm1$ outside of essentially a $n-1$ dimensional surface (that is very flat but possibly very irregular).
We then consider ``vertical rescalings''
$$ (x',x_n)\mapsto \left(x', \frac{x_n}a\right)$$ of these ``transition surfaces''.
 
A main step in the proof then consists  in proving that the
vertical rescalings of the ``transition surfaces'' are compact as $a\downarrow 0$ and converge to a continuous graph  $g:\R^{n-1}\rightarrow \R$.. 
To achieve this compactness  we need a ``H\"older type'' estimate, or improvement of oscillation, for 
vertical rescalings of level sets.
The proof of this improvement of oscillation estimate is given in Section 4. It requires to build fine barriers for the semilinear equation and several auxiliary result that are given in Section 2 and 3.

A second step in the proof is to show that the limit graph $g$ is a viscosity
solution of the
linear translation invariant elliptic equation  $(-\Delta)^{\frac{1+s}{2}} g =0$ in $\R^{n-1}$. This is done in Section 5.

Finally we obtain  the improvement by compactness, inheriting  it from the $C^{1,\alpha}$ regularity of  $(-\Delta)^{\frac{1+s}{2}} g =0$. This is done in Section 6.

The rest of the paper, namely
Sections~\ref{SECT7} and~\ref{SECT8},
is devoted to the proof of Theorem~\ref{C:1} and its consequences. 
As explained before, Theorem~\ref{C:1} follows from Theorem \ref{thm:improvement} but not in a straightforward way.
Let us summarize next the
main steps of its proof.

We use two different iterations of
Theorem~\ref{thm:improvement}. 
The first iteration, that we informally call  ``preservation 
of flatness'', is given in Corollary~\ref{0NL:VV}. 
The second  iteration, really a geometric  ``improvement
of flatness'', is given in Corollary~\ref{NL:VV}. 
Corollary~\ref{NL:VV} is stronger in the sense that the flatness is improved geometrically in a sequence of dyadic balls, but only up to a large mesoscale. In Corollary~\ref{0NL:VV} the flatness does not improve but is just preserved across scales but, as a counterpart, it gives information up to scale 1.

To prove Theorem~\ref{C:1} we need to combine Corollary~\ref{0NL:VV} 
with a multi-scale application of 
Corollary~\ref{NL:VV}. Doing so, we prove that
the transition level sets are trapped, in all of $\R^n$, 
between a Lipschitz graph and a finite vertical 
translation of it. 
Then, we need to use the 
sliding method (in its full strength) to 
conclude that the level sets of the solution are indeed flat.  

\medskip
%
%
%
%
%
%
%
%
%
%

\subsection*{Notation}
For the convenience of the reader we gather here the 
notation that we will follow throughout all the paper. 
The following  list of notations is just for quick 
reference and all the notations are introduced (again) within the text at their first appearance. 
\begin{itemize} 
\item $L$, $f$  are the nonlocal elliptic operator and the nonlinearity,
respectively, see \eqref{EQ}.
\item $n\ge 2$, $s\in(0,1)$, $K$ are, respectively, the dimension, the order of the operator, and
the even, $C^{1,1}$ convex set defining the norm $\|\cdot \|_K$ in the definition of $L$.
\item ${\mathcal{L}}$ denotes the one-dimensional fractional Laplacian as in \eqref{1dfraclap}.
\item  $r_K$ is the inner curvature radius of $\partial K$; see \eqref{assumpK}.
\item  $\kappa$, $c_\kappa$ and $l_\kappa$ are the 
constants in the quantitative assumptions of $f$, see Subsection~\ref{SUB13}.
\item  We will call a constant {\em universal} if it depends only 
on $n$, 
$s$, $r_K$ , $\kappa$, $c_\kappa$ and $l_\kappa$.
In particular, universal constants depend only on $n$, $L$, and $f$.
\item $\lambda,\Lambda$ are the ellipticity constants of $L$, $\mathcal C = \mathcal C_L$ is the convex body with support 
function $h_L$, and $\rho'>\rho>0$ are the two constants in its 
curvature bounds, see Subsection~\ref{SUB12}.
\item  We write   
\begin{equation*}\label{setinclusion}
X\subset Y \quad \mbox{in }B\qquad \mbox{if} \qquad X\cap B \,\subset\, Y \cap B.
\end{equation*}
\item We denote by $\|\, \cdot\,\|_{\mathcal C}$ the norm with unit ball $\mathcal C$.
We also denote by $\mathcal C_r(y)$ the ball of radius $r$ and center~$y$ 
with respect to this norm, namely
\[\mathcal C_r(y) := y+r\mathcal C.\]
Notice that when $L$ is the fractional Laplacian $\mathcal C_r(y)$ is simply $B_r(y)$.
\item Points in $\R^{n-1}$ will be denoted by $x'$ and  $x=(x',x_n)$ 
denotes a point in $\R^{n}$ with $n$-th coordinate $x_n$.
{F}rom now on, we also denote by $B_r'$ the $(n-1)$-dimensional
ball of radius $r>0$.
\item $\xi$ denotes the function $\xi:\R^{n-1}\rightarrow \R$ which is defined by
\begin{equation}\label{xi}
\xi(x') =\xi(|x'|) :=\bigl(1+|x'|^2 \bigr)^{\frac{1+\alpha}{2}}-1.
\end{equation}
\item Given $b>0$, we denote by $d_b$
the signed distance function to the set $\{x_n\ge b\,\xi(x')\}$ with respect to the norm $\|\, \cdot\,\|_{\mathcal C}$, that is, 
\begin{equation*}
d_b(x) = 
\begin{cases}
+\inf \big\{ \| z-x\|_{\mathcal C}\ :\  z_n = b\,\xi(z')\big\},
\quad &\mbox{for  }  x_n \ge b\,\xi(x'),
\\
-\inf \big\{ \| z-x\|_{\mathcal C}\ :\    
z_n = b\,\xi(z') \big\},		 
&\mbox{for  }  x_n \le b\,\xi(x').
\end{cases}
\end{equation*}
\item  Given $\phi: \R\rightarrow (-1,1)$,
for any $x\in\R$, we set
\[ \phi^{b}(x) := \phi\bigl(d_{b}(x)\bigr).\]
Notice that $\phi^b: \R\rightarrow (-1,1)$,
and it may be seen as a 
``rearrangement'' of the layer solution $\phi$ with respect to the signed distance function.
\end{itemize}

In addition to the previous notations we use also the following very standard ones.
\begin{itemize}
\item Given~$r\in\R$, we denote by~$r_+:=\max\{r,0\}$ and~$r_-:=\max\{-r,0\}$.
\item Given a measurable function~$f:X_1\times \dots\times X_m\to\R$, we use the repeated integral notation
$$ \int_{X_1}\,dx_1 \,\dots\,\int_{X_m}\,dx_m\, f(x_1,\dots,x_m):=
\int_{X_1} \left[ \dots \int_{X_m} f(x_1,\dots,x_m)\,dx_m\dots\right]
\,dx_1.$$
\end{itemize}

\section{Approximate solutions via deformation
of level sets}\label{SECTION2}

In this section we construct approximate solutions in $B_1$ by deforming (slightly curving) the flat level sets of a one-dimensional solution.

\subsection{A layer cake formula}

The main results of this paper are valid for an operator $L$ of the form
\begin{equation}\label{formL}
 Lu(x):= \int_{\R^n} \big( u(x)-u(x+y)\big) \frac{\mu\big(y/|y|\big)}{|y|^{n+s}}\,dy.
\end{equation}
The measure $\mu$ in~\eqref{formL} is often called in jargon the
``spectral measure''. By assumption ---see \eqref{assumpK} on page~\pageref{assumpK}--- we have 
that~$\mu$ satisfies 
\begin{equation}
\label{muL0}
\mu(z)=\mu(-z)\quad\mbox{and}\quad 
0< \lambda \le \mu(z) \le \Lambda< +\infty  
\quad\mbox{for all }z\in  S^{n-1}.
\end{equation}
where $\lambda$, $\Lambda$ are positive constants depending only on $n$, $s$ and $r_K$ and are called the ellipticity constants.
\medskip

Now we give a simple layer cake
representation for the integro-differential operators.

We use the notation
\begin{equation}\label{CHInot}
\chi_{[a_1,a_2]}(\theta) :=
\left\{
\begin{matrix}
1 & {\mbox{ if }} a_1\le a_2 {\mbox{ and }} \theta\in [a_1,a_2],\\
\, \\
0 & \begin{matrix} {\mbox{ if either }} a_1>a_2,
\\ {\mbox{ or }}
a_1\le a_2 {\mbox{ and }} \theta\not\in [a_1,a_2].\end{matrix}
\end{matrix}
\right.
\end{equation}
Using this,
we have the following simple layer cake type representation for nonlocal operators:

\begin{lemma}\label{layer cake}
It holds that
\begin{equation}\label{RP:1}
L v(x) = \int_{\R^n}\,dy\,
\int_\R \,d\theta
\Big(\chi_{[v(x+y),v(x)]}(\theta)-\chi_{[v(x),v(x+y)]}(\theta)\Big)
\,\frac{\mu(y/|y|)}{|y|^{n+s}}.
\end{equation}
Furthermore, if~$x\in\R^n$ is such that~$v(x)=w(x)$, then
\begin{equation}\label{RP:2}
L v(x)-L w(x) = \int_{\R^n}\,dy\,
\int_\R \,d\theta
\Big(\chi_{[v(x+y),w(x+y)]}(\theta)-\chi_{[w(x+y),v(x+y)]}(\theta)\Big)
\,\frac{\mu(y/|y|)}{|y|^{n+s}}.
\end{equation}
\end{lemma}

\begin{proof} By~\eqref{CHInot},
$$ (a_1-a_2)_- = (a_2-a_1)_+ =\int_\R \chi_{[a_1,a_2]}
(\theta)\,d\theta$$
and therefore
\begin{eqnarray*}
v(x)-v(x+y) &=& \big( v(x)-v(x+y)\big)_+ - \big( v(x)-v(x+y)\big)_-
\\ &=& \int_\R \chi_{[v(x+y),v(x)]}(\theta)\,d\theta
- \int_\R \chi_{[v(x),v(x+y)]}(\theta)\,d\theta
.\end{eqnarray*}
So, we integrate and we find~\eqref{RP:1}.

Similarly, we write
$$ w(x+y)-v(x+y)
= \int_\R \chi_{[v(x+y),w(x+y)]}(\theta)\,d\theta
- \int_\R \chi_{[w(x+y),v(x+y)]}(\theta)\,d\theta,$$
which gives~\eqref{RP:2} after integration.
\end{proof}

\subsection{The operator $L$ and the convex set $\mathcal C_L$}\label{SUB11} 

Throughout the paper  the fractional Laplacian in dimension~$1$ 
(without normalization constant) will be denoted $\mathcal L$.  Namely, given a bounded $\psi\in C^2(\R)$, we define
\begin{equation}\label{1dfraclap}
\mathcal L \psi(z) := \int_{-\infty}^\infty 
\frac{\psi(z)-\psi( z+\zeta)}{|\zeta|^{1+s}} \,d\zeta, \quad z\in \R.
\end{equation}

For $\psi$ as above,  $\omega\in S^{n-1}$ and $h>0$, 
we define, for any~$x\in\R^n$,  
\begin{equation}\label{NOoh}
\bar\psi_{\omega,h}\,(x)  := \psi\left(\omega\cdot \frac{x}{h}\,\right).\end{equation}
Then, for each operator $L$ of the form \eqref{formL}, let $h_L : 
S^{n-1} \rightarrow (0,\infty)$ be defined as follows. We set $h_L(\omega) :=h$, where $h>0$ satisfies 
\begin{equation}\label{oKAJ:11} 
L \bar\psi_{\omega,h} (x)= \mathcal L \psi\left(\omega\cdot \frac{x}{h}\right) \ 
\mbox{for all }\psi\in C^2(\R)\cap L^\infty(\R).\end{equation}
Using the function $h_L$, we define the closed convex set
\begin{equation}\label{CAL C} \mathcal C= \mathcal C_L := 
\bigcap_{\omega\in S^{n-1}} \big\{ x\in \R^n\,:\, 
x\cdot \omega \le h_L(\omega) \big\}.\end{equation}
We notice also that, since $L$ is even, both $h_L$ and $\mathcal C_L$ are 
even, i.e. symmetric with respect to the origin. 
In addition, we remark that,
when $L=(-\Delta)^{s/2}$,  $\mathcal C_L$ is a ball (centered at $0$).

Our assumption \eqref{assumpK} on $K$ is made in order to guarantee that
\begin{equation}\label{C C1}
{\mbox{$\partial\,\mathcal C_L$ is $C^{1,1}$ and strictly convex.}}\end{equation}
More quantitatively, that there are constants ~$\rho'>\rho>0$ depending only on $n$, $s$, $K$ such that
\begin{equation} \tag{H1'}\label{assumpL}
{\mbox{the curvatures of  $\partial\,\mathcal C_L$ are bounded above 
by  $\displaystyle \frac 1\rho$ and below by $\displaystyle\frac{1}{\rho'}$.}}
\end{equation}
%

We remark that the definition of~$h_L$ in \eqref{oKAJ:11} is well posed,
and indeed an explicit expression of~$h_L(\omega)$ 
is obtained through the formula
\begin{equation}\label{formula1hL} 
h_L(\omega) = 
\left(\frac12
\int_{S^{n-1}}  |\omega\cdot \theta|^s \, 
\mu( \theta)\,d\theta\right)^{1/s} .
\end{equation}
To prove \eqref{formula1hL}, we proceed as follows
\begin{eqnarray*}
L\bar\psi_{\omega,h}\,(x) &=&
\int_{\R^n} \left( \psi\left(\omega\cdot \frac{x}{h}\,\right)-
\psi\left(\omega\cdot \frac{x+y}{h}\,\right)
\right)\,
\frac{\mu\big(y/|y|\big)}{|y|^{n+s}}\,dy\\
&=&
\int_{0}^{+\infty} \,d\varrho\,\int_{S^{n-1}}\,d\theta\,
\left( \psi\left(\omega\cdot \frac{x}{h}\,\right)-
\psi\left(\omega\cdot \frac{x}{h}+\omega\cdot \frac{\varrho\theta}{h}
\,\right)
\right)\,
\frac{\mu(\theta)}{\varrho^{1+s}}
\\
&=&\frac12\,
\int_{-\infty}^{+\infty} \,d\varrho\,\int_{S^{n-1}}\,d\theta\,
\left( \psi\left(\omega\cdot \frac{x}{h}\,\right)-
\psi\left(\omega\cdot \frac{x}{h}+\omega\cdot \frac{\varrho\theta}{h}
\,\right)
\right)\,
\frac{\mu(\theta)}{|\varrho|^{1+s}} \\
&=&\frac12\,
\int_{-\infty}^{+\infty} \,d\zeta\,\int_{S^{n-1}}\,d\theta\,
\left( \psi\left(\omega\cdot \frac{x}{h}\,\right)-
\psi\left(\omega\cdot \frac{x}{h}+\zeta
\,\right)
\right)\,
\frac{|\omega\cdot\theta|^s\,\mu(\theta)}{h^s \,|\zeta|^{1+s}},\end{eqnarray*}
where we used the change of variables $\zeta = \frac{\varrho\omega\cdot\theta}{h}$.
Hence, if $h=h_L(\omega)$ is given by \eqref{formula1hL},
\begin{eqnarray*}
L\bar\psi_{\omega,h}\,(x) = 
\int_{-\infty}^{+\infty} 
\left( \psi\left(\omega\cdot \frac{x}{h}\,\right)-
\psi\left(\omega\cdot \frac{x}{h}+\zeta
\,\right)
\right)\,
\frac{d\zeta
}{|\zeta|^{1+s}}
=
\mathcal L \psi\left(\omega\cdot\frac{ x}{h}\right),
\end{eqnarray*}
that is \eqref{oKAJ:11}.

A special case of \eqref{formula1hL} occurs when 
the spectral measure is induced by
a convex set, namely when
\[ \frac{\mu\big(y/|y|\big)}{|y|^{n+s}} = \frac{1}{\|y\|_K^{n+s}}  \]
for some convex set $K$, where $\| \cdot \|_K$ 
is the norm with unit ball~$K$, that is,
for any~$p\in\R^n$,
\begin{equation}\label{pKnorma}
\| p\|_K := \inf\{ t>0 {\mbox{ s.t. }} p/t\notin K\}.\end{equation}
Then, in this case, an integration in polar coordinates yields 
\begin{eqnarray*} 
&&
h_L(\omega) = \left( \frac12\,\int_{S^{n-1}}  d\theta  \frac{|\omega\cdot \theta|^s}{\|\theta\|_K^{n+s}} \right)^{1/s}  =
\left( \frac{n+s}2\,\int_{S^{n-1}}\,d\theta
\int_0^{1/\| \theta\|_K}\,d\varrho\,
|\omega\cdot \theta|^s\,\varrho^{n+s-1} \right)^{1/s} 
\\&&\qquad\qquad\qquad
=
\left( \frac{n+s}{2}\,\int_{K}   |\omega\cdot x|^s dx  \right)^{1/s}
.\end{eqnarray*}
As pointed out to us by M. Ludwig, 
to whom we are indebted for this comment and 
the interesting references provided, 
the convex body associated to this support function is the so 
called ``$L_p$-intersection body'' of $K$. 
These convex bodies are well studied in convex geometry, 
in relation to the important Busemann-Petty problem,
see~\cite{Berck} and references therein for more information on this subject. 

It is proved in \cite{Berck} that, for any given convex 
set $K$ (bounded and with nonempty interior) 
which is symmetric with respect to the origin, the 
function $h_L$ is strictly convex in all the nonradial directions. 
Also, from \eqref{formula1hL} it follows 
that $h_L$ is $C^{1,1}$ in $\R^n\setminus \{0\}$ 
when $\mu$ is $C^{1,1}$. Actually $\mu\in C^{2-s+\epsilon}$ 
suffices since the ``kernel''  $|\omega\cdot\theta|^s$ 
is $C^s$ and this yields a regularity improvement.

When $K$ is any $C^{1,1}$ convex set, the previous observations imply 
that the set  
\[\mathcal C^*_L := \{ h_L =1\}\] is a $C^{1,1}$, even with respect to $0$, strictly convex set.
Noting that $\mathcal C_L$ and $\mathcal C^*_L$ are one the polar body 
of the other, one can show that $\mathcal C_L$ 
is also a $C^{1,1}$, even, strictly convex set. 
Indeed, since $\mathcal C^*_L$ 
is a $C^{1,1}$, even, strictly convex set, 
any point of its boundary can be touched by two even ellipsoids, 
one contained in, and the other one containing,  $\mathcal C^*_L$.  
Considering the polar transformations of these ellipsoids 
we show the same property for $\mathcal C_L$.

The previous discussion can be summarized in the following 
\begin{lemma} [\cite{Berck}]
If $L$ is of the form \eqref{formL1} with $K$ even and satisfying \eqref{assumpK} 
then $\mathcal C_L$ satisfies \eqref{assumpL}
for some ~$\rho'>\rho>0$ depending only on $n$, $s$, $K$.
\end{lemma}

\begin{remark}\label{REMAA}
Theorems~\ref{thm:improvement}
and~\ref{C:1} are valid (and proved here) for general operators of the form~\eqref{formL}
and under the more general assumption in~\eqref{assumpL} on page~\pageref{assumpL}
replacing~\eqref{assumpK} on page~\pageref{assumpK}.
\end{remark}

In view of Remark~\ref{REMAA}, throughout the proofs, the general setting
in~\eqref{formL} will be assumed, together with~\eqref{assumpL}.

\subsection{Touching the level sets of the distance function by
concentric spheres}

This section discuss some geometric features
related to the signed anisotropic distance function to a convex set.
To this aim, we recall some basic properties of
the support function $h_L$ defined in \eqref{formula1hL}.
First of all, for any $x$, $y\in\R^n$, 
the following
inequality of
Cauchy-Schwarz type holds true
\begin{equation}\label{CSani}
x\cdot y\le h_L(y)\,\|x\|_{\mathcal{C}}.\end{equation}
See e.g. Lemma~1
in \cite{torino} for an elementary proof.
Note that here $\|\cdot\|_{\mathcal{C}}$ denotes the norm with unit ball $\mathcal C = \mathcal C_L$, that is a convex set different from (although related to) $K$.

As a counterpart of \eqref{CSani}, we have that
equality holds when one of the two vectors is normal
to the sphere to which the other vector belongs. More precisely,
we have that
if $z_0\in\R^n$, $R>0$,
$z\in\partial {\mathcal{C}}_R(z_0)$
and $\omega_0\in S^{n-1}$ is the inner normal
of $\partial {\mathcal{C}}_R(z_0)$
at the point $z$, then
\begin{equation}\label{LEM:JJ}
\omega_0\cdot (z_0-z)= R\,h_L(\omega_0),\end{equation}
see for example Lemma~3
in \cite{torino}.

Moreover, it is useful
to recall that $h_L$ is the
``support function'' of the convex body~${\mathcal C}$, namely
for any~$\omega\in S^{n-1}$ we have that
\begin{equation}\label{LE:DESE}
h_L(\omega)=\sup_{x\in {\mathcal C}}
x\cdot\omega,\end{equation}
see for instance Lemma~2
in \cite{torino}.

We recall also here that both $h_L$ and $\mathcal C$ are even.

Given a nonempty, closed and convex set~$K\subset\R^n$,
we define the anisotropic signed
distance function from $K$ as
\begin{equation}\label{dk defin}
d_K(x):=\inf\big\{ \ell(x) \,:\, \ell(x) =
\omega\cdot x + c ,\quad h_L(\omega)= 1,  \quad c\in\R \quad 
\mbox{and} \quad \ell\ge 0 \  \mbox{ in all of }K\big\}.  
\end{equation}
Notice that $d_K$ is a concave function,
since it is the infimum of affine functions.
Moreover, as shown 
for instance in Proposition~1
of \cite{torino}, 
it holds that
\begin{equation}\label{902e3iowy222hkj}
d_K(x) =
\begin{cases}
+\inf \big\{ \| z-x\|_{\mathcal C}\ :\  z\in\partial K\big\}
\quad &\mbox{for  }  x \in K,
\\
-\inf \big\{ \| z-x\|_{\mathcal C}\ :\    z\in\partial K\big\}
&\mbox{for  }  x\in\R^n\setminus K.
\end{cases}
\end{equation}
We have that $d_K$
is a Lipschitz function,
with Lipschitz constant~$1$ with respect to the anisotropic norm, namely,
for any~$p$, $q\in\R^n$,
\begin{equation}\label{lip lemma app}
|d_K(p)-d_K(q)|\le \| p-q\|_{\mathcal C},\end{equation}
see e.g. Lemma~4
in \cite{torino}.
\medskip

With this setting, we can now prove
that the level sets of~$d_K$ are touched by 
appropriate concentric anisotropic
spheres:

\begin{lemma}\label{TPQ}
Let $z_0\in K =\{ d_K>0\}$.
Assume that~$\mathcal C_R(z_0)\subset\{d_K>0\}$  touches $\partial K = \{d_K=0\}$ at some
point~$\bar z\in\{d_K=0\}$.

Then, for any $t\in(-\infty,R)$, 
\begin{equation}\label{LK}\begin{split}
&
{\mbox{the set $\mathcal C_{R-t}(z_0)$ 
is contained in $\{d_K>t\}$}}\\ &{\mbox{and
touches $\{d_K=t\}$ at the point }}\\ & 
z:= z_0+\frac{R-t}{R} (z-z_0)\in \big( \partial\mathcal C_{R-t}(z_0)\big)
\cap\{d_K=t\}.\end{split}\end{equation}

Furthermore, if we denote by~$\omega_0\in S^{n-1}$
the inner normal 
of $\partial {\mathcal{C}}_R(z_0)$
at the point $\bar z$, it holds that
\begin{equation}\label{123456dfghjk9210903:BIS}\begin{split}
&{\mbox{$R-\| x-z_0\|_{\mathcal{C}}\le d_K(x)\le$}}
\frac{\omega_0}{h_L(\omega_0)}\cdot (x-\bar z)
{\mbox{ for any $x\in\R^n$,}}\\
&{\mbox{and equalities hold when }} \ x= z_0+\frac{R-t}{R} (z-z_0) ,
{\mbox{ for some }}t\in(-\infty,R).\end{split}\end{equation}
In particular,
\begin{equation}\label{123456dfghjk9210903}
d_K \left( z_0+\frac{R-t}{R} (\bar z-z_0) \right)
=t .\end{equation}

In addition, if~$\tau\in(-\infty,R)$
and~$z_\tau:=z_0+\frac{R-\tau}{R} (z-z_0)$, then
\begin{equation}\label{123456dfghjk9210903:TRIS}\begin{split}&
{\mbox{ ${\mathcal{C}}_{|t-\tau|}(z_\tau)$
is tangent from the outside to the set}}\\ &\left\{x\in\R^n{\mbox{ s.t. }}
d_K(x)\le t\le\frac{\omega_0}{h_L(\omega_0)}\cdot (x-z)\right\}
\\&{\mbox{at the point~$z$.
}}\end{split}\end{equation}
\end{lemma}

\begin{proof} 
The geometric setting of Lemma \ref{TPQ}
is depicted in Figure~\ref{GG1}.

\begin{figure}[ht]
\includegraphics[width=14.5cm]{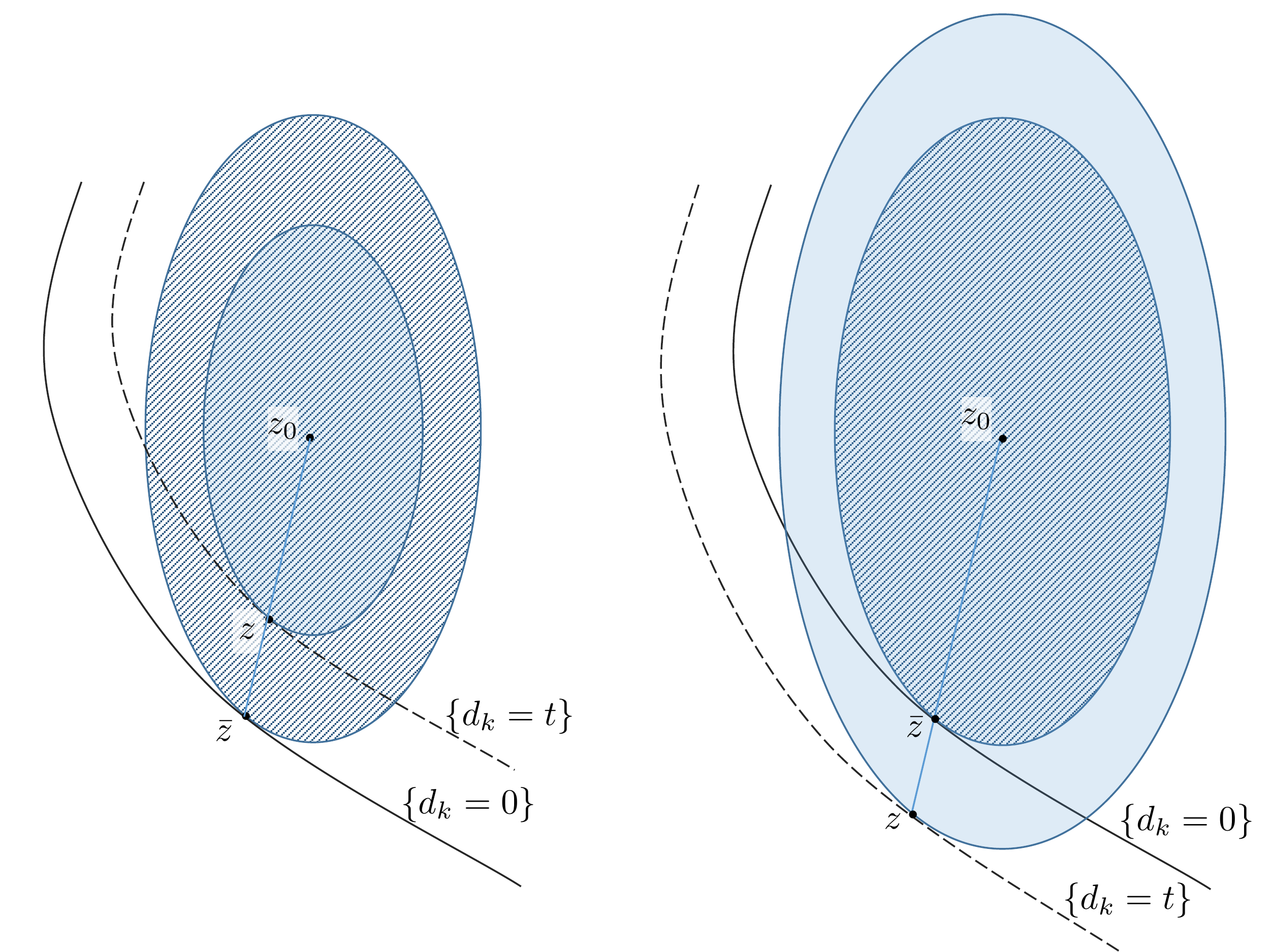}
\caption{\sl The geometry of Lemma \ref{TPQ}
when $t\in(0,R)$
and when $t\in(-\infty,0)$.}
\label{GG1}
\end{figure}

The proof goes like this.
For every~$t\in(-\infty,R)$, we have that
\begin{equation} \label{51022}
\| z-z_0\|_{\mathcal C}
=\frac{R-t}{R} \|\bar z-z_0\|_{\mathcal C}=R-t,\end{equation}
and therefore
\begin{equation}\label{TPQ:1}
z\in\partial \mathcal C_{R-t}(z_0).\end{equation}
In addition, we point out that, for every~$t\in(-\infty,R)$,
\begin{equation}\label{TPQ:2}
\mathcal C_{R-t}(z_0)\subset\{d_K\ge t\}.\end{equation}
To check this, we distinguish two cases: either $t\ge0$ (i.e. $t\in[0,R)$)
or~$t<0$.
If~$t\ge0$, we argue as follows. Let $p\in \mathcal C_{R-t}(z_0)$.
Then, for any $q$ with $\|q\|_{\mathcal C}\le t$ we have that $p+q\in
\mathcal C_{R}(z_0)\subset\{d_K\ge 0\}$.

Consequently, in light of \eqref{902e3iowy222hkj},
for any affine function $\ell(x) =
\omega\cdot x + c$, with $h_L(\omega)= 1$,  
$c\in\R$, and such that $\ell\ge 0$ in $\{d_K>0\}$, 
it holds that
\begin{equation}\label{iol5789oihg123f}
\ell(p+q)\ge 0.\end{equation} 
Therefore, 
we slide the halfspace with inner normal~$\frac{\omega}{|\omega|}$
till it touches~$\partial{\mathcal C}$ 
and we take this touching point~$q$.
Namely, we have~$q\in\partial{\mathcal C}_t$,
with~$\frac{\omega}{|\omega|}$ as inner normal of~$\partial{\mathcal C}_t$
at~$q$. Hence, by~\eqref{LEM:JJ},
$$ -\frac{\omega}{|\omega|}\cdot q= t\,h_L\left( \frac{\omega}{|\omega|} \right)
=\frac{t\,h_L(\omega)}{|\omega|}=\frac{t}{|\omega|}.$$
This and~\eqref{iol5789oihg123f} give that
$$ 0 \le \ell(p+q)=\omega\cdot p + c+\omega\cdot q
=
\omega\cdot p+c - t.$$
This shows that $\ell(p)\ge t$ and so,
in view of \eqref{defd}, that $d_K(p)\ge t$,
that establishes \eqref{TPQ:2} in this case.

So, we now check \eqref{TPQ:2} in the case in which $t<0$.
For this, let $p\in\mathcal C_{R-t}(z_0)$.
If $p\in \mathcal C_{R}(z_0)$, then $d_K(p)\ge 0\ge t$,
and we are done, so we can suppose that $p\in
\mathcal C_{R-t}(z_0)\setminus \mathcal C_{R}(z_0)$, hence
$$ \|p-z_0\|_{\mathcal{C}} \in \big[ R,\,R-t\big].$$
We take
$$ q:= z_0 + \frac{R\,(p-z_0)}{ \|p-z_0\|_{\mathcal{C}} }.$$
Notice that
$\|q-z_0\|_{\mathcal{C}}=R$,
hence $q\in
\mathcal C_{R}(z_0)\subset\{d_K\ge 0\}$.
This and \eqref{lip lemma app} imply that
$$ -d_K(p)\le d_K(q)-d_K(p) \le 
\|q-p\|_{\mathcal{C}} =
\big| R-\|p-z_0\|_{\mathcal{C}} \big|
=\|p-z_0\|_{\mathcal{C}}-R\le (R-t)-R,$$
that gives $d_K(p)\ge t$, as desired.
This completes the proof of \eqref{TPQ:2}.

Now we check that
\begin{equation}\label{TPQ:3}
d_K(z)=t.\end{equation}
To this aim, we observe that
\begin{equation*}
z\in{\mathcal C_{R-t}(z_0)}\subset
\{d_K\geq t\},\end{equation*}
thanks to~\eqref{TPQ:1} and~\eqref{TPQ:2}.
Consequently, to establish~\eqref{TPQ:3}, we only need to prove that
\begin{equation}\label{TPQ:4}
d_K( z)\le t.\end{equation}
To this goal, if $t\ge0$
we use \eqref{lip lemma app} and we see that
$$ d_K(z)=d_K(z)-d_K(\bar z)\le
\|z-\bar z\|_{\mathcal{C}}=
\frac{t}{R}\,\|\bar z-z _0\|_{\mathcal{C}}\le t,$$
which is \eqref{TPQ:3} in this case.

If instead $t<0$, we denote by~$\omega_0\in S^{n-1}$
the inner normal 
of $\partial {\mathcal{C}}_R(z_0)$
at the point $z$, and we exploit~\eqref{LEM:JJ} 
(recall also \eqref{dconc}) to see that
\begin{equation}\label{XdfX67}
\begin{split}
& d_K(z)\le \frac{\omega_0}{h_L(\omega_0)}\cdot(z-\bar z)
=\frac{\omega_0}{h_L(\omega_0)}\cdot\left(z_0-\bar z+\frac{R-t}{R}(\bar z-z_0)
\right)\\
&\qquad=\frac{t}{R}\frac{\omega_0}{h_L(\omega_0)}\cdot(z_0-\bar z)=t.
\end{split}\end{equation}
This finishes the proof of~\eqref{TPQ:3}.

Then, \eqref{LK} follows from~\eqref{TPQ:1},
\eqref{TPQ:2} and~\eqref{TPQ:3}. In turn, \eqref{LK} also implies~\eqref{123456dfghjk9210903}.

We also observe that, from the previous considerations,
\eqref{123456dfghjk9210903:BIS} follows in a
straightforward way using~\eqref{902e3iowy222hkj}.

Now we prove~\eqref{123456dfghjk9210903:TRIS}.
First of all, we notice that $\|z_\tau-\bar z \|_{\mathcal{C}}
=|t-\tau|$,
due to~\eqref{51022}, so~$\bar z$ lies on~$\partial
{\mathcal{C}}_{|t-\tau|}(z_\tau)$. Thus, to prove the result
in~\eqref{123456dfghjk9210903:TRIS}, we need to show that
\begin{equation}\label{PA901562734812}
\left\{x\in\R^n{\mbox{ s.t. }}
\|x-z_\tau\|_{\mathcal{C}}<|t-\tau|
{\mbox{ and }}
d_K(x)\le t \le\frac{\omega_0}{h_L(\omega_0)}\cdot (x-z) \right\}=
\varnothing.
\end{equation}

\begin{figure}[ht]
\includegraphics[width=14.5cm]{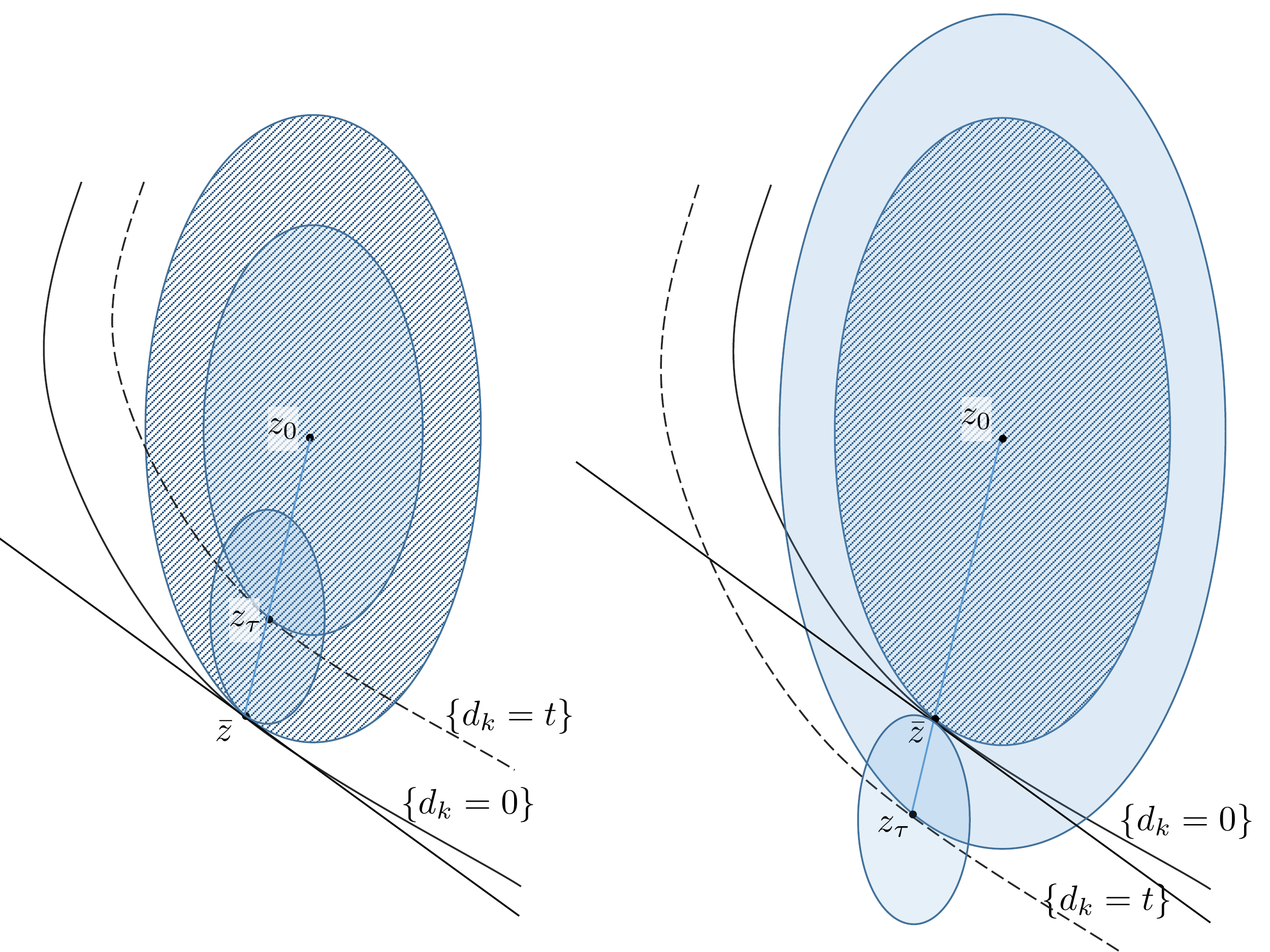}
\caption{\sl Proof of \eqref{PA901562734812}
when $\tau\ge t$
and when $\tau<t$.}
\label{dFIG789}
\end{figure}

For this, we refer to Figure~\ref{dFIG789},
we argue by contradiction
and we suppose that there exists $x$ in the set on the left hand side of
\eqref{PA901562734812}. Then, we distinguish two cases,
either $\tau\ge t$ or $\tau<t$. If $\tau\ge t$,
we use \eqref{123456dfghjk9210903} to see that
$d_K(z_\tau) =\tau$ and so, exploiting \eqref{lip lemma app},
\begin{eqnarray*}
&& 0\le t-d_K(x)=t-\tau +d_K(z_\tau)-d_K(x)
\le -|t-\tau|+\|x-z_\tau\|_{\mathcal{C}}<0,
\end{eqnarray*}
which is a contradiction. If instead $\tau<t$,
using \eqref{CSani} and \eqref{LEM:JJ} we find that
\begin{eqnarray*}
&& t\le\frac{\omega_0}{h_L(\omega_0)}\cdot (x-\bar z)
=\frac{\omega_0}{h_L(\omega_0)}\cdot (z_\tau-\bar z)
+\frac{\omega_0}{h_L(\omega_0)}\cdot (x-z_\tau)\\
&&\qquad\le 
\tau\frac{\omega_0}{h_L(\omega_0)}\cdot\frac{z_0-\bar z}{R}
 +\|x-z_\tau\|_{\mathcal{C}}
=\tau +\|x-z_\tau\|_{\mathcal{C}}<
\tau +|t-\tau|=t,\end{eqnarray*}
which is a contradiction.
This proves \eqref{PA901562734812}, which in turn gives~\eqref{123456dfghjk9210903:TRIS}.
\end{proof}

\subsection{Distance function from a convex graph}
Here, we look at the special case of the distance function
from a sufficiently flat graph with an appropriate growth.
For this, let $\alpha\in (0,s)$ be a fixed constant. 
Let us introduce the function $\xi:\R^{n-1}\rightarrow \R$ defined by
\[
\xi(x') = \bigl(1+|x'|^2 \bigr)^{\frac{1+\alpha}{2}}-1.
\]
Note that $\xi(0)=0$ and that $\xi$ is convex with 
\[D^2\xi ={\rm diag}\left(\frac{1+\alpha r^2}{1+ r^2}, 1,1,\dots, 1\right) (1+\alpha)(1+r^2)^{\frac{\alpha-1}{2}} \ge 0\]
in a coordinate system with the first axis pointing in the radial direction.

Given some orthonormal coordinates $x= (x',x_n)$ in $\R^n$ and $b>0$, 
let us define 
\[
\Gamma_b := \big\{\,x_n \ge b \,\xi(x')\,\big\}.
\]

From the convex set $\Gamma_{b}$ we define the following anisotropic signed distance function 
\begin{equation}\label{defd}
 d_{b}(x) := 
\inf\big\{ \ell(x) \,:\, \ell(x) =\omega\cdot x + c ,\quad h_L(\omega)= 1,  \quad c\in\R, \  \mbox{and} \quad \ell\ge 0 \  \mbox{ in all of }\Gamma_{b} \big\}.  
\end{equation}
By comparing with \eqref{dk defin}, we have that $d_b$ coincides
with $d_K$ with the particular choice $K:=\Gamma_b$.
Hence, in view of \eqref{902e3iowy222hkj}, it holds that
\begin{equation}\label{formulad}
 d_b(x) = 
\begin{cases}
+\inf \big\{ \| x-z\|_{\mathcal C}\ :\  z\in 
\partial \Gamma_b \big\} \quad &\mbox{for  }  
x\in \Gamma_b,
\\
-\inf \big\{ \| x-z\|_{\mathcal C}\ :\   z\in \partial \Gamma_b \big\} 		   &\mbox{for  }  x\in \R^n\setminus \Gamma_b,
\end{cases}
\end{equation}
where $\|\,\cdot\,\|_{\mathcal C}$ denotes the norm with unit ball $\mathcal C$;
for this, we use the notations in~\eqref{CAL C} 
and~\eqref{pKnorma} and we
recall that,
throughout the paper,  ${\mathcal{C}}= {\mathcal{C}}_L$ is the convex body associated 
to $L$ and, for any $r>0$ and any~$y\in\R^n$, we set
\begin{equation}\label{crp}
\mathcal C_r(y): = y+r\mathcal C.\end{equation}

The following result states that under the 
hypothesis~\eqref{assumpL}, and for~$b$ small enough,
all the level sets of~$d_b$ passing close enough 
to the origin are~$C^{1,1}$ graphs, with their 
second derivatives bounded by~$Cb$ near the 
origin and with growth at infinity controlled by~$Cb|x|^{1+\alpha}$.

\begin{lemma}\label{lemcurvbounds}
There exist~$b_0>0$ and $C_0>0$, depending only
on $\alpha$, $\rho$ and $\rho'$, such that for any~$b\in(0,b_0)$ and 
any~$t\in\R$, with~$\{d_b=t\}\cap \mathcal C_{4/\rho}\neq\varnothing$, we have that
\[ \{d_b=t\} \,= \, \{y_n = G(y')\} \]
where $G:\R^{n-1}\rightarrow \R$ is a suitable 
convex function satisfying
\begin{equation}\label{secondder}
\big|D^2 G\big| \le C_0 b  \quad\mbox{in }B'_{4\rho'/\rho}
\end{equation}
and
\begin{equation}\label{differenceG}
\big|G(y')-G(0)\big| \le C_0b (1+ |y'|)^{1+\alpha}   \quad\mbox{for all  }y'\in\R^{n-1}.
\end{equation}
\end{lemma}

To prove Lemma~\ref{lemcurvbounds} 
we need the following simple preliminary result:
 
\begin{lemma}\label{lemma:claim}
We have the following inequalities between the anisotropic and the Euclidean norm
\begin{equation}\label{2.4bis}
 \frac{1}{\rho'} |\,\cdot\,|  \le \| \,\cdot\,\|_{\mathcal C} 
\le \frac 1 \rho |\,\cdot\,|.
\end{equation}
\end{lemma}

\begin{proof}
By \eqref{assumpL}, we have
\[
 B_{\rho}\subset \mathcal C \subset B_{\rho'}.
\]
Therefore, recalling~\eqref{pKnorma},
\[
 \|x\|_{\mathcal C} = \sup \{t>0 {\mbox{ s.t. }}
x/t \notin \mathcal C\} \le  \sup \{t>0 {\mbox{ s.t. }}
x/t \notin B_\rho\} =  \frac 1 \rho |x|,
\]
which proves the second inequality in~\eqref{2.4bis}.
The second inequality is proven likewise.
\end{proof}

\begin{proof}[Proof of Lemma \ref{lemcurvbounds}] We have
\begin{equation}\label{2ndderxi}
 \big|D^2 \xi(x')\big| \le C\, (1+|x'|^2)^{\frac{\alpha-1}{2}},
\end{equation}
for some $C>0$ depending only on $\alpha$.

Using that $0\in\partial \Gamma_b =\{d_b=0\}$ 
and that, by assumption, there exists~$p\in 
\mathcal C_{4/\rho}$ such that $p\in \{d_b =t\}$, we have that
\begin{equation}\label{rhoaso}
|t|  \le \|p-0\|_{\mathcal C} \le \frac{4}{\rho}.
\end{equation}

Choose $y\in\{d_b =t\}$. Recalling Lemma~\ref{lemma:claim},
let $\bar y$ be a point on $\partial\Gamma_b$ for which
\[
\frac{1}{\rho'}|\bar y-y|  \le \|\bar y-y\|_{\mathcal C}=
 |d_b(y)| =|t|.
\]

By \eqref{2ndderxi} there exists a ball of radius $R\ge c/b$ 
contained in $\Gamma_b$ and touching $\partial\Gamma_b$ at the point $y$, where $c>0$ depends only on $\alpha$. Since $\mathcal C_r\subset B_{r\rho'}$ there exists $z_0$ in $\Gamma_b$ such that 
\begin{equation}\label{thetouching}
 \mathcal C_{R/\rho'}(y_0) \subset \Gamma_b\quad\mbox{ and touches }\partial\Gamma_b \mbox{ at }\bar y.
\end{equation}

Then, by Lemma \ref{TPQ} we have that
\begin{equation}\label{thetouching2}
\mathcal C_{R/\rho'-t}(y_0) \subset \{d_b>t\}\quad\mbox{ and touches } \{d_b=t\} \mbox{ at } y.
\end{equation}
Since $\mathcal C$ is assumed to be $C^{1,1}$, this shows that the boundary of the convex set $\{d_b>t\}$ is $C^{1,1}$.

Let us prove that, indeed,
the boundary of~$\{d_b>t\}$ is a graph and control 
the gradient and the second derivatives of this graph. 
We assume that $b_0$ is small enough so that 
\[  R/\rho' -t \ge \frac{c}{b\rho'} - \frac{4}{\rho}  \ge \frac{c}{b} \] 
where~$c$ denotes a positive universal constant (that may change each time).

Now, denoting $y= (y',y_n)$ and $\bar y= (\bar y',\bar y_n)$, we have
\[ |\bar y'| \le |y'| + |y-\bar y| \le   |y'| + \rho' |t| \le   |y'| + \frac{4}{\rho}  \le |y'| +C.\]

The tangent plane to $\mathcal C_{R/\rho'-t}(y_0)$ 
at~$\bar y$ is parallel to the tangent plane 
to~$\mathcal C_{R/\rho'}(y_0)$ at~$y$ and, 
by~\eqref{thetouching}, this slope is given by 
\[  b(1+\alpha) |\bar y'|(1+|\bar y'|^2)^{\frac{\alpha-1}{2}}  \le 2(1+|\bar y'|^2)^{\frac{\alpha}{2}} \le  C_0(1+|y'|^2)^{\frac{\alpha}{2}},\]
where $C_0$ is a universal constant and where we have used that 
\[ (b\xi(r))' = b(1+\alpha) r(1+r^2)^{\frac{\alpha-1}{2}}.\] 
Since the point $y$ can be chosen arbitrarily on 
the surface~$\{d_b=t\}$, this proves that this surface is an entire graph. Namely, that
\[ \{d_b=t\} \ =\  \{y_n =G(y')\} 
\qquad\mbox{where }|DG(y')| \le C_0(1+
|\bar y|^2)^{\frac{\alpha}{2}}.\]

Finally, the estimate for the second derivative 
in~\eqref{secondder} follows from \eqref{thetouching2} 
recalling that $R\ge cb$. On the other hand, 
\eqref{differenceG} follows from the fact that
\[ \big|G(y')-G(0)\big| \le\sup_{|z'|\le |y'|} |DG(z')| |y'| \le  
C_0|y'|b(1+|y'|^2)^{\frac{\alpha}{2}} \le C_0b (1+|y'|)^{
1+\alpha} \quad\mbox{for all  }y'\in\R^{n-1}. \]  
This completes the proof of Lemma~\ref{lemcurvbounds}.
\end{proof}


\subsection{Modeling solutions with the distance function}

We now construct useful barriers by using the level sets
of the distance function as a profile and controlling
the error produced in the equation by such procedure. For this, we
let  $\phi : \R \rightarrow (-1,1)$ be a $C^2$ and increasing function with $$
\lim_{z\to \pm \infty} \phi(z)= \pm 1.$$ Note that any such $\phi$ solves 
an equation of the type
\[ \mathcal L\phi  = f_\phi (\phi) \quad \mbox{in } \R,\]
where $f_\phi:(-1,1)\rightarrow \R$ is defined by
\begin{equation} \label{f phi}
f_\phi :=  (\mathcal L\phi)\circ\phi^{-1}.\end{equation}

Now we define a suitable rearrangement procedure
that produces a function $\phi_{b}: \R^n\rightarrow (-1,1)$ from any 
given $\phi$ as above and modeled along the level sets of the distance function~$d_b$,
as introduced in~\eqref{defd}. Namely, we set
\begin{equation}\label{phi b}
\phi^{b}(x) := \phi\bigl(d_{b}(x)\bigr).\end{equation}
Then, we have that~$\phi^{b}$ is ``almost'' a solution
of the equation with nonlinearity~$f_\phi$, as given by the following result:

\begin{lemma}\label{lemdeform}
Let $L$ satisfy \eqref{assumpL}. Then, there 
exist positive quantities~$b_0$ and $C_0$ 
depending only on $n$, $s$, $\lambda$, $\Lambda$, 
$\rho$ and $\rho'$ (and thus independent of $\phi$), 
such that the 
following holds.

Assume that 
\begin{equation}\label{imageassumption}
[-1+\delta, 1-\delta]\,\subset 
\,\phi\left( \left[-  \frac{1}{\rho'},\frac{1}{\rho'} \right]\right).
\end{equation}

Then, for all $\omega\in S^{n-1}$ and $b\in(0,b_0)$ we have
\begin{equation}\label{FAH111}
0\le L \phi^{b}  -f_\phi\big( \phi^{b} \big) \le
C_0 (b+ \delta)\quad \mbox{in }B_1.
\end{equation}
\end{lemma}

\begin{proof}
Let us fix $z\in B_1$.  Let $\theta_0= \phi^b(z)$ be the level of $\phi^b$ at $z$. 
By~\eqref{phi b}, we know that~$d_b(z)=
\phi^{-1}(\phi^b(z))=\phi^{-1}(\theta_0)=:t_0$.

We also recall that~$h_L$ was introduced 
in~\eqref{oKAJ:11} (or, equivalently,
in~\eqref{formula1hL}) and we let $\omega$
be the unit vector normal to~$\{d_b=t_0\}$ at~$z$
and pointing towards~$\{d_b>t_0\}$.
Then, we define
\begin{equation}\label{d tilde}
\tilde d(x) := \frac{\omega}{h_L(\omega)}\cdot(x-z) +t_0.\end{equation}
We also set~$\tilde\phi:=\phi\circ\tilde d$.
Using the notation in~\eqref{NOoh}, we have that
$$ \tilde\phi(x)=\phi\left(\frac{\omega}{h_L(\omega)}\cdot(x-z) +t_0\right)
=\phi\left(\frac{\omega}{h_L(\omega)}\cdot
\big(x-z+\omega\,h_L(\omega)\,t_0\big)\right)
=\bar\phi_{\omega,h}(x-z+\omega\,h\,t_0),$$
with~$h:=h_L(\omega)$.

Consequently, by \eqref{oKAJ:11} and~\eqref{f phi},
for any~$x\in\R^n$,
\begin{equation}\label{2.20}
\begin{split}
& L\tilde\phi(x)=L\bar\phi_{\omega,h}(x-z-\omega\,h\,t_0)=
A\phi\left(
\frac\omega{h}\cdot(x-z+\omega\,h\,t_0)
\right)\\
&\qquad\qquad=f_\phi\left(
\phi\left(
\frac\omega{h}\cdot(x-z+\omega\,h\,t_0)
\right)\right)
=
f_\phi\left(
\phi\left(
\frac\omega{h}\cdot(x-z)+t_0)
\right)\right)=f_\phi(\tilde\phi(x)).
\end{split}
\end{equation}

Now, by \eqref{123456dfghjk9210903:BIS} in Lemma~\ref{TPQ}  we have
\begin{equation}\label{dconc}
d_b\le \tilde d \quad \mbox{in }\R^n,
\end{equation}
see also Lemma~6
in~\cite{torino} for the elementary
proof of this and related facts. 
Moreover, 
\begin{equation}\label{dconc:X2X}
d_b= \tilde d \quad \mbox{along the ray}  \quad 
\mathcal R:= \{z_0+t'(z-z_0), \ t'\ge 0\}.\end{equation}
{F}rom the observations in~\eqref{dconc} and~\eqref{dconc:X2X}
it follows that  
\begin{equation}\label{uok78o8er83yfusdigfsjhv1}
\{d_b= t\} \quad \mbox{is tangent to} \quad \{\tilde d= t\} \quad \mbox{at some point on }\mathcal R.
\end{equation}
Notice that, by construction, 
\begin{equation}\label{0ipoks678900101}
\phi^b(z)=\phi(t_0)= \tilde\phi(z)\end{equation}
and, by~\eqref{dconc} and the monotonicity of $\phi$, it holds that~$
\phi^b\le \tilde\phi$. Accordingly, $L \phi^b(z)-L \tilde\phi(z)\ge0$.
Thus, we apply
the layer cake formula in~\eqref{RP:2} of Lemma~\ref{layer cake} and use that
the image of $\phi$ is contained in~$[-1,1]$ to conclude that
\begin{equation}\label{7:67987611}
\begin{split}
0\le L \phi^b(z)-L \tilde\phi(z)
&=\int_{\R^n}\,dy\,\int_\R \,d\theta
\chi_{[\phi^b(z+y),\,\tilde \phi(z+y)]}(\theta)
\,\frac{\mu(y/|y|)}{|y|^{n+s}}
\\&=\int_{-1}^1 d\theta \int_{\R^n } dy\, \frac{\mu(y/|y|)}{|y|^{n+s}}
\,{\chi_{S_\theta}}(z+y) = \int_{-1}^1 d\theta \,I_z(\theta)
\end{split}
\end{equation}
where
\[ 
S_\theta:=\big\{ x\in\R^n \ :\  \phi^b(x)\le\theta\le\tilde \phi(x)\big \} =
\big\{ x\in\R^n \ :\  
d_b(x)\le  \phi^{-1}(\theta)\le \tilde d(x)\big\}
\]
and 
\[
I_z(\theta) := 
\int_{\R^n }  \frac{\mu(y/|y|)}{|y|^{n+s}}
\,{\chi_{S_\theta}}(z+y) \,dy.\] 
Now we recall \eqref{2.20} and \eqref{0ipoks678900101} 
to see that
$L\tilde \phi(z)=f_\phi(\tilde \phi(z))=f_\phi(\phi^b(z))$
and so we can rewrite~\eqref{7:67987611} as
\begin{equation}\label{ST:FOND}
0\le L \phi^b(z)- f_\phi(\phi^b(z))=
\int_{-1}^1 d\theta \,I_z(\theta)
.\end{equation}
Now,
given $\theta\in(-1,1)$, let us define 
\[  t_\theta  := \phi^{-1}(\theta).\]

In the next steps of the proof we will establish different estimates for~$I_z(\theta)$
by distinguishing the two
cases $\{d_b = t_\theta\}\cap {\mathcal{C}}_{3/\rho}(z)
= \emptyset$  and $\{d_b = t_\theta\}\cap {\mathcal{C}}_{3/\rho}(z) \neq \emptyset$. \medskip

{\em Case 1}. Let $\{d_b = t_\theta\}\cap \mathcal C_{3/\rho}(z) = \emptyset$.
We take~$b\in(0,b_0)$ 
with $b_0$ small enough, depending only on~$\rho$ and~$\rho'$, 
and we claim that
we have that 
\begin{equation}\label{B32}
S_\theta\cap B_{2}= \emptyset.\end{equation} 
Indeed, by \eqref{123456dfghjk9210903:TRIS},
$\{d_b = t_\theta\}\cap \mathcal C_{3/\rho}(z) = 
\emptyset$ implies that~$S_\theta \cap \mathcal 
C_{3/\rho}(z) = \emptyset$.
Hence, recalling that~$z\in B_1$, we have that~$B_2\subset B_3(z)\subset \mathcal C_{3/\rho}(z)$ and hence \eqref{B32} follows.

Thus, since $z\in B_1$, using~\eqref{B32} we conclude that 
\begin{equation}\label{77:0129837asdf4}
I_z(\theta) = 
\int_{\R^n\setminus B_{1} }  \frac{\mu(y/|y|)}{|y|^{n+s}}
\,{\chi_{S_\theta}}(z+y) \,dy
\le  \int_{\R^n \setminus B_{1}}
\frac{\Lambda}{|y|^{n+s}}\,dy\le C,
\end{equation} for some $C>0$.

Now we claim that in this case we have
\begin{equation}\label{346dfiodbh10938y4t}
\theta\in [-1,\,-1+\delta)\cup(1-\delta,\,1].
\end{equation}
Indeed, if not, by \eqref{imageassumption},
$$ \theta\in [-1+\delta, 1-\delta]\subset
\phi\left( \left[-  \frac{1}{\rho'},\frac{1}{\rho'} \right]\right)$$
and so
$$ t_\theta =\phi^{-1}(\theta)\in \left[-  \frac{1}{\rho'},\frac{1}{\rho'} \right].$$
Then, using that $0\in\{d_b =0\}$ we find that
\[
 \inf \left\{ \frac{1}{\rho'} |y-0|\ :\ y\in \{d_b = t_\theta\}\right\}  \le \inf \big\{ \|y-0\|_{\mathcal C}\ :\ y\in \{d_b = t_\theta\}\big\} =|t_\theta| \le \frac{1}{\rho'}
 \]
and thus $\{d_b = t_\theta\}$ intersects $B_2$, which is 
a contradiction. This proves~\eqref{346dfiodbh10938y4t}.

{\em Case 2}.  Now we deal with the case
$\{d_b = t_\theta\}\cap \mathcal C_{3/\rho}(z) \neq \emptyset$ and $b\in(0,b_0)$, 
with $b_0$ small enough. 
Note that we have $\{d_b = t_\theta\}\cap \mathcal 
C_{4/\rho}\neq\emptyset$ since $z\in B_1\subset \mathcal C_{1/\rho}$.

In this case,
we recall \eqref{dconc:X2X} and \eqref{uok78o8er83yfusdigfsjhv1} and
we take
$\bar z = (\bar z', \bar z_n )$ to be
the triple intersection point described there, that is
\begin{equation}\label{ih7w8e639128:0}
\bar z \in \{d_b= t_\theta\}\cap
\{\tilde d= t_\theta\} \cap \mathcal R.
\end{equation}
With this notation, we can write the set $S_\theta$
as a suitable portion of space trapped between a linear function
and a convex one with small detachment one 
from the other.
For this, we exploit
Lemma~\ref{lemcurvbounds} to see that
\begin{equation}\label{ih7w8e639128}
\{d_b=t_\theta\}\ \,= \, \{y_n = G(y')\} \end{equation}
with $G$ convex and satisfying 
\begin{equation}\label{boundsG}
|DG(y')|\le C_0 b \quad\mbox{in }|y'|<\frac{4\rho'}{\rho}\quad \mbox{and}\quad \big|G(y')-G(0)\big| \le C_0 b (1+|y'|)^{1+ \alpha}\quad \mbox{for all }y'.
\end{equation}
Therefore,
the condition $d_b(x)\le t_\theta$ is equivalent to
the fact that the point $x$ lies below the graph of $G$, namely
that $x_n\le G(x')$. Similarly, from \eqref{ih7w8e639128:0},
we have that $\omega$
is normal to both $\{\tilde d=t_\theta\}$ and $\{d_b=t_\theta\}$
at $\bar z$ and so,
by \eqref{ih7w8e639128},
the condition that $t_\theta\le\tilde d(x)$ is equivalent to
$$ x_n \ge G(\bar z')+\nabla G(\bar z')\cdot (x'-\bar z').$$
In consequence of these observations, we have that
\begin{equation}\label{whoisStheta}
S_\theta \ = \ \big\{ G(\bar z') + \nabla G(\bar z')\cdot(x'-\bar z') 
\le x_n \le G(x') \big\}.
\end{equation}

Next we observe that, as a consequence 
of~\eqref{123456dfghjk9210903:TRIS},
for~$r= \|z-\bar z\|_{\mathcal C}$, we have
\begin{equation}\label{EL01}
\mathcal C_r(z) \subset \R^n\setminus S_{\theta}.
\end{equation}
Therefore, for all $y$ in $S_\theta$, recalling Lemma~\ref{lemma:claim},
$$ |y-\bar z|\le \rho' \|y-\bar z\|_{\mathcal C}   \le C \big( \|y- z\|_{\mathcal C}+r\big)\le
C|y-z|.$$
Accordingly, 
if~$z+y\in S_\theta$, then~$|z+y-\bar z|\le C|y|$.
As a consequence of this and~\eqref{whoisStheta}, we have that,
for any fixed $y'\in\R^{n-1}$,
\begin{eqnarray*}
&&
\int_{\R}  \frac{{\chi_{S_\theta}}(z+y)}{|y|^{n+s}} \,dy_n\le
C\int_{\R}  \frac{{\chi_{S_\theta}}(z+y)}{|z+y-\bar z|^{n+s}} \,dy_n
\le
C\int_{\R}  \frac{{\chi_{S_\theta}}(z+y)}{|z'+y'-\bar z'|^{n+s}} \,dy_n
\\ &&\qquad=C\,
\int_{\{ G(\bar z')+\nabla G(\bar z')\cdot (z'+y'-\bar z')
\le z_n+y_n\le
G(z'+y')\}}  \frac{dy_n}{|z'+y'-\bar z'|^{n+s}} \\&&\qquad=
C\,\frac{G(z'+y')-
G(\bar z')-\nabla G(\bar z')\cdot (z'+y'-\bar z')
}{|z'+y'-\bar z'|^{n+s}} 
.\end{eqnarray*}
Hence, if we integrate in $y'\in\R^{n-1}$ and use the change
of variable~$Y':=z'+y'-\bar z'$, up to renaming~$C>0$ we have that
\begin{equation}\label{cas2}
\begin{split}
I_z(\theta)\,&\le C
\int_{\R^n}  \frac{{\chi_{S_\theta}}(z+y)}{|y|^{n+s}} \,dy
\le C\int_{\R^{n-1}}
\frac{G(z'+y')-
G(\bar z')-\nabla G(\bar z')\cdot (z'+y'-\bar z')
}{|z'+y'-\bar z'|^{n+s}}\,dy'\\
&= C\int_{\R^{n-1}}
\frac{G(Y'+\bar z')-
G(\bar z')-\nabla G(\bar z')\cdot Y'
}{|Y'|^{n+s}}\,dY'
\le Cb
,\end{split}\end{equation}
where \eqref{boundsG} has been used 
in the last estimate ---note that $\bar z\in \mathcal C_{3/\rho}(z)$ and thus 
\[
 |\bar z'| \le |\bar z| \le \rho' \big( \|z-\bar z\|_{\mathcal C}   + \|z\|_{\mathcal C}  \big) \le \rho' \big( 3/\rho  + 1/\rho  \big) \le 4\rho'/\rho. 
\] 

{\em Final estimate.}  We recall that, from~\eqref{ST:FOND},
\[
0\le L \phi^b(z)- f_\phi(\phi^b(z))=
\int_{-1}^1 d\theta \,I_z(\theta) = \int_{{\mathcal{A}}} 
d\theta \,I_z(\theta) + \int_{{\mathcal{B}}} 
d\theta  \,I_z(\theta),
\]
where $\mathcal A$ is the set of levels~$\theta$ 
as in {\em Case 1} and~$\mathcal B$ is the set of 
levels~$\theta$ as in {\em Case 2}.
Then, on the one hand, \eqref{346dfiodbh10938y4t} 
implies that $|\mathcal A|\le 2\delta$, 
and, for each~$\theta \in {\mathcal{A}}$, we have 
that~$I_z(\theta)\le C$. On the other hand, 
\eqref{cas2} yields that, for each $\theta \in {\mathcal{B}}$, 
we have that~$I_z(\theta)\le Cb$.
Therefore, 
\[
0\le L \phi^b(z)- f_\phi(\phi^b(z))=
 \int_{\mathcal A} d\theta \,I_z(\theta) + 
\int_{\mathcal B} d\theta  \,I_z(\theta) \le C\delta +Cb,
\]
which proves~\eqref{FAH111}, as desired.
\end{proof}

\section{Decay estimates for solutions}\label{SECT3}

The goal of this section is to provide suitable decay estimates for our solutions.
For this, we start with
a preliminary result:

\begin{lemma}\label{lemgamma0}
Let $w$ be such that $Lw \le -k w$ in $B_R$, 
where $R\in[2,\infty)$ and~$k\in[1,\infty)$. Suppose that $0\le w\le 2$ in all of $\R^n$, then
\[ 0\le w \le \frac{C}{ (k^{1/s}R) ^{\gamma_0}} \quad \mbox{in }B_{1},\]
where $C$, $\gamma_0>0$ depend only on $n$,  $s$, and on the ellipticity constants.
\end{lemma}
\begin{proof} The idea of the proof is
to use a barrier argument at the different scales.
For the reader's convenience, we split the proof into three steps.\smallskip

\noindent{\em Step 1.} We prove the following statement. {\em Assume that $L\bar w  \le - \bar w$ in $B_1$ and
\begin{equation}\label{impro}
0\le \bar w \le 2^{\gamma_0j}\quad \mbox{in } B_{2^j}
\end{equation}
for all $j \ge 0$. Then, \eqref{impro} holds also for $j=-1$.}\smallskip

For this, we take~$\eta\in C^\infty_0(B_{3/4})$ 
radially nonincreasing, with $\eta= 1$ in $B_{1/2}$.
Let also~$\gamma_0\in (0,1)$, to be taken appropriately small, and set~$h_0 := 1-2^{-\gamma_0}>0$.
We define the function
\[\phi: = (1-h_0\eta) \chi_{B_1} + \sum_{j=1}^{\infty} 2^{\gamma_0j} \chi_{B_{2^j}\setminus B_{2^{j-1}}}.\]
We observe that $\phi=1-h_0\eta$ in~$B_1$
and~$\phi=2^{\gamma_0j}$ in~${B_{2^j}\setminus B_{2^{j-1}}}$ for any~$j\ge1$.
As a consequence, for any~$x\in B_{3/4}$,
\begin{eqnarray*}
-L\phi(x) &=& \int_{B_1} \frac{(1-h_0\eta)(z)-(1-h_0\eta)(x)}{|z-x|^{n+s}}\,\mu\left( \frac{z-x}{|z-x|} \right)\,dz\\&&\qquad
+\sum_{j=1}^{+\infty}
\int_{{B_{2^j}\setminus B_{2^{j-1}}}} \frac{2^{\gamma_0j}-(1-h_0\eta)(x)}{|z-x|^{n+s}}\,\mu\left( \frac{z-x}{|z-x|} \right)\,dz
\\ &\le& h_0\left[
\left| 
\int_{B_1} \frac{\eta(x)-\eta(z)}{|z-x|^{n+s}}\,\mu\left( \frac{z-x}{|z-x|} \right)\,dz\right|+\sum_{j=1}^{+\infty}\left|
\int_{{B_{2^j}\setminus B_{2^{j-1}}}} \frac{2^{\gamma_0j}-1-h_0}{|z-x|^{n+s}}\,\mu\left( \frac{z-x}{|z-x|} \right)\,dz\right| \right]
\\ 
&\le& Ch_0+C \sum_{1\le j\le\gamma_0^{-1/3}}(2^{\gamma_0j}-1)+C 
\sum_{ j\ge\gamma_0^{-1/3}}\frac{2^{\gamma_0j}}{2^{j(1+s)}}\\
&\le& Ch_0+C\frac{ 2^{\gamma_0^{2/3}}-1}{\gamma_0^{1/3}}+
\frac{C}{2^{\frac{1+s}{2\gamma_0^{1/3}}}},
\end{eqnarray*}
with~$C>0$ possibly varying from line to line.
In particular, when~$\gamma_0$ (and so~$h_0$) is small, we have that~$-L\phi\le 1/2 \le \phi$ in~$B_{3/4}$.

Since also~$\phi\ge\bar w$ outside~$B_{3/4}$,
using the maximum principle we have that $\bar w \le \phi$ in $B_{3/4}$.
Consequently, $\bar w \le 1-h_0= 2^{-\gamma_0}$ in $B_{1/2}$.
This completes the proof of the statement in {\em Step 1}. 
\smallskip

\noindent{\em Step 2.} Now we prove the following statement. {\em 
Let $\tilde w$ be such that $L\tilde w \le -\tilde w$ in $B_{\tilde R}$, where $\tilde R\ge1$. Suppose that $0\le \tilde w\le 2$ in all of $\R^n$, then, for any~$\tilde\rho\in\left[\frac{1}{2},\,\tilde R\right)$, we have
\[ 0\le \tilde w \le C\,\left(\frac{\tilde\rho}{\tilde R}\right)^{\gamma_0}\quad \mbox{in }B_{\tilde\rho},\]
for some~$C$, $\gamma_0>0$.
}\smallskip

The proof of this claim
is an iteration of {\em Step 1}. Namely, we take~$N\in\N$ such
that~$2^N\le\tilde R<2^{N+1}$. For any~$i\in\N$, $i\in[1,\,N+1]$,
we set
\begin{equation}\label{PAL0} \bar w_i(x) = 2^{(i-1)\gamma_0-1}\,\tilde w(2^{N-i+1}x).\end{equation}
Notice that, by construction,
\begin{equation}\label{PAL2}
{\mbox{$
L\bar w_i\le -2^{(N-i+1)s}\,\bar w_i\le -\bar w_i$ in~$B_{2^{i-1}}\supset B_1$}}
\end{equation}
and, if~$i\in\N$, $i\in[1,\,N]$,
\begin{equation}\label{PAL3} \bar w_{i+1}(x)=2^{\gamma_0}\,\bar w_i(x/2)
.\end{equation}
We claim that
\begin{equation}\label{PAL}
{\mbox{for any~$0\le j\le i-1$, we have that $\bar w_i \le 2^{(j-1)\gamma_0}$ in $B_{2^{j-1}}$.}}
\end{equation}
The proof of \eqref{PAL} is by induction. First, we observe that, for any~$j\ge0$,
in~$B_{2^j}$ we have that
$$ \bar w_1\le 2^{-1}\,\sup_{\R^n}\tilde w\le 1\le 2^j.$$
{F}rom this and~\eqref{PAL2}, we can use {\em Step 1} with~$\bar w:=\bar w_1$ and find that
$ \bar w_1 \le 2^{-\gamma_0}$ in~$B_{1/2}$.
This is \eqref{PAL} when~$i=1$.

Now, we suppose that~\eqref{PAL} holds true for the index~$i\in[1,\,N]$, and we prove it for
the index~$i+1$. To this aim, we claim that, for any~$j\ge0$,
\begin{equation}\label{PAL4}
{\mbox{$\bar w_{i+1}\le 2^{\gamma_0 j}$
in $B_{2^j}$.}}
\end{equation}
To check this, we distinguish two cases. If~$j\ge i$, then we recall~\eqref{PAL0} and we see that
$$ \sup_{B_{2^j}} \bar w_{i+1} \le 2^{i\gamma_0-1}\,\sup_{\R^n} \tilde w\le
2^{i\gamma_0}\le 2^{j\gamma_0},$$
as desired. If instead~$j\le i-1$, then we exploit~\eqref{PAL} with index~$i$
together with~\eqref{PAL3} and we obtain
$$ \sup_{B_{2^j}} \bar w_{i+1}=2^{\gamma_0}\,
\sup_{B_{2^{j-1}}} \bar w_i \le 2^{\gamma_0}\cdot
2^{(j-1)\gamma_0}= 2^{j\gamma_0}.$$ 
This proves~\eqref{PAL4}. 

So, by~\eqref{PAL2} and~\eqref{PAL4}, we can use {\em Step 1} with~$\bar w:=\bar w_{i+1}$
and conclude that~$\bar w_{i+1}\le 2^{-\gamma_0}$ in~$B_{1/2}$.
This inequality and~\eqref{PAL4} imply that
\begin{equation*}
{\mbox{for any~$0\le j\le i$, we have that $\bar w_{i+1} \le 2^{(j-1)\gamma_0}$ in $B_{2^{j-1}}$,}}
\end{equation*}
that is~\eqref{PAL} for the index~$i+1$, as desired.
This completes the inductive proof of~\eqref{PAL}.

Hence, using the notation~$m:=i-j$, we deduce from~\eqref{PAL} that
\begin{equation}\label{0a01q2} \sup_{B_{2^{N-m}}} \tilde w\le 2^{1-m\gamma_0},\end{equation}
for any~$m\in\Z$ with~$m\le N+1$.

Now we take~$M\in\Z$ such that~$2^{-M-1}\le2^{-N}\tilde\rho<2^{-M}$.
Notice that
$$ \frac12\le \tilde\rho\le 2^{N-M},$$
hence~$M\le N+1$.
Then, we can apply~\eqref{0a01q2} with~$m:= M$ and we obtain that
$$ \sup_{B_{\tilde\rho}} \tilde w\le 
\sup_{B_{2^{N-M}}} \tilde w=
\le
2^{1-M\gamma_0} = \frac{2^{1+2\gamma_0} \cdot 2^{(N-M-1)\gamma_0} }{2^{(N+1)\gamma_0}}\le 
\frac{2^{1+2\gamma_0} \cdot \tilde\rho^{\gamma_0} }{\tilde R^{\gamma_0}}.$$
This establishes the claim in {\em Step 2}.\smallskip

\noindent{\em Step 3.} Now we complete the proof of Lemma~\ref{lemgamma0} scaling the statement proven in  {\em Step 2}. 
To this aim,
we take~$w$ as in the statement of Lemma~\ref{lemgamma0} and~$p\in B_1$.
We define~$\tilde R:= (R-1)k^{1/s}$ and
$$ \tilde w(x):= w\left( p+\frac{x}{k^{1/s}}\right).$$
Notice that~$\tilde R\ge k^{1/s}\ge1$. Furthermore, for any~$x\in B_{\tilde R}$ we have that
$$ \left| p+\frac{x}{k^{1/s}}\right|\le
|p|+\frac{|x|}{k^{1/s}}\le 1+\frac{\tilde R}{k^{1/s}}=R,$$
and therefore, for any~$x\in B_{\tilde R}$,
$$ L\tilde w(x)= \frac1{k}\, Lw\left( p+\frac{x}{k^{1/s}}\right)\le -w\left( p+\frac{x}{k^{1/s}}\right)=-\tilde w(x).$$
So, we can use {\em Step~2} with~$\tilde \rho:=1/2$ and obtain that
$$ w(p)=\tilde w(0)\le\sup_{B_{1/2}} \tilde w\le \frac{C}{(2\tilde R)^{\gamma_0}}=
\frac{C}{(2(R-1)k^{1/s})^{\gamma_0}}\le
\frac{C}{(Rk^{1/s})^{\gamma_0}}
,$$ which is the desired result.
\end{proof}

As a consequence of the previous preliminary result, we have:

\begin{lemma}\label{lemdecay}
Let $R\ge2$ and~$\eps\in(0,1]$. Let~$u:\R^n\to[-1,1]$ be a solution of $Lu = \eps^{-s} f(u)$ in $\R^n$. Then, 
if~$\eps$ is sufficiently small,
\[ u(x)\ge 1- C \,\left(\frac\eps{R}\right)^{\gamma_0} \quad \mbox{whenever}\quad B_R(x)\subset \{u\ge 1-\kappa\}\]
and 
\[ u(x)\le -1+ C \,\left(\frac\eps{R}\right)^{\gamma_0} \quad \mbox{whenever}\quad B_R(x)\subset \{u\le -1+\kappa\},\]
for some~$C$, $\gamma_0>0$.

In particular, for $n=1$, the profile $\phi_0$ satisfies 
\begin{equation}\label{DEC}
\big|\phi_0-(-1)| \le C_{f} |x|^{-\gamma_0} \  \mbox{in  }(-\infty,-1]
\qquad \mbox{and}\qquad \big|\phi_0-1| \le C_{f} |x|^{-\gamma_0} \  \mbox{in  }[1,+\infty).
\end{equation}
\end{lemma}

\begin{proof}
Using assumption \eqref{assumpf} we have
\[ -f(u) =f(1)-f(u)\le -c_\kappa (1-u)\quad\mbox{for }  u \ge 1-\kappa\]
and therefore
\[
L(1-u) =-Lu = -\eps^{-s}f(u) \le -\eps^{-s} c_\kappa(1-u)\quad\mbox{in } \{u \ge 1-\kappa\}.
\]
Thus, from Lemma \ref{lemgamma0} with $w:=1-u$ and $k:=\eps^{-s}c_\kappa$ we obtain the desired
decay estimates. 
\end{proof}

\section{Improvement of oscillation for level sets of solutions}\label{SECT4}

The goal of this section is
to establish 
the following improvement of oscillation result for level sets,
which is one of the cornerstones of this paper. 
This result is crucial since it gives compactness of sequences of vertical rescaling of the level sets.

For fixed $\alpha\in(0,s)$, 
$m_0\in \N$ and $a>0$, let us introduce
\begin{equation}\label{kdia} k_a :=  
\left\lfloor  \frac{\log a}{\log(2^{-\alpha})}\right\rfloor-m_0 ,\quad \mbox{which belongs to $\N$ for $a$ small}.
\end{equation}
Notice that $k_a\uparrow +\infty$ as $a\downarrow0$,
and
\begin{equation}\label{a2ka}
\frac{1}{2} 2^{-\alpha m_0} 2^{-\alpha k_a}\le a \le 2^{-\alpha m_0} 2^{-\alpha k_a}.
\end{equation}

\begin{theorem} \label{harnack}
Assume that $L$ satisfies $\eqref{assumpL}$ and that $f$ satisfies \eqref{assumpf} and \eqref{existslayer}.
Then, given $\alpha\in(0,s)$ there exist $p_0\in(2,\infty)$,  $a_0\in(0,1/4)$, and $\eta_0 \in(0,1)$, depending only on $\alpha$, $m_0$, and on the universal constants, such that the following statement holds.

\vspace{6pt}

Let $a\in(0,a_0)$ and $\eps \in (0, a^{p_0})$. 
Let $u: \R^n \rightarrow (-1,1)$ be a solution of $
Lu=\epsilon^{-s}  f(u)$ in $B'_{2^{k_a}}\times(-2^{k_a},2^{k_a})$ such that
\[ 
\{x_n \le -a 2^{j(1+\alpha)}\}\, \subset\, \{u\le -1+\kappa\} \,\subset \,\{u\le 1-\kappa\}  \,\subset \,\{x_n\le a 2^{j(1+\alpha)}\} \quad \mbox{in }B'_{2^j}\times (-2^{k_a}, 2^{k_a}) , 
\]
for $j= \{0,1,2,\dots, k_a\}$.

Then, either 
\[ 
\left\{x_n \le - a (1-\eta_0) \right\} \,\subset \,\{u\le -1+\kappa\} \quad \mbox{in }B'_{1/2}\times (-2^{k_a}, 2^{k_a}) \]
or
\[ \{u\le 1-\kappa\}  \,\subset\, \left\{x_n \le a (1-\eta_0) \right\} \quad \mbox{in }B'_{1/2}\times (-2^{k_a}, 2^{k_a}).
\]
\end{theorem}

We will deduce Theorem \ref{harnack} from the following result:

\begin{proposition}\label{propimpliesharnack}
Assume that $L$ satisfies $\eqref{assumpL}$ and that $f$ satisfies \eqref{assumpf} and \eqref{existslayer}.
Then, given $\alpha\in(0,s)$ there exist $p_0\in(2,\infty)$,  $a_0\in(0,1/4)$, and $\eta_0 \in(0,1)$, depending only on $\alpha$, $m_0$, and on the universal constants, such that the following statement holds.

\vspace{6pt}

Let $a\in(0,a_0)$ and $\eps \in (0, a^{p_0})$. 
Let $u: \R^n \rightarrow (-1,1)$ be a solution of~$
Lu=\epsilon^{-s}  f(u)$
in $B'_{2^{k_a}}\times(-2^{k_a},2^{k_a})$ such that
\begin{equation}\label{01L:A}
\{u\le 1-\kappa\}  \,\subset \,
\left\{x_n \le  a 2^{j(1+\alpha )}\right\}
\qquad{\mbox{in }}\;B'_{2^j}\times(-2^{k_a},2^{k_a})
\end{equation}
for  $j= \{0,1,2,\dots, k_a\}$, and 
\begin{equation}\label{signchoice}
 \int_{B_2} u\,dx\ge0.
\end{equation}
Then, we have that
\begin{equation}\label{conclusionPropHar}
 \{u\le 1-\kappa\}  \,\subset\, \left\{ x_n \le a (1-\eta_0) \right\} \quad \mbox{in }B'_{1/2}\times(-2^{k_a},2^{k_a}).
\end{equation}
\end{proposition}

For its use in the proof of Proposition~\ref{propimpliesharnack}, we recall the following maximum principle: 

\begin{lemma}\label{maxpr}
There exists~$\theta>0$, depending only on~$n$, $s$, 
$\lambda$ and~$\Lambda$, such that the following statement holds true. 

Let~$w\in C^2(B_4)$ satisfy 
$$ \begin{cases}
Lw\ge -\theta \quad {\mbox{in }}B_4\cap\{w\le0\},\\
\,\\
\displaystyle\int_{\R^n} w_-(y) (1+|y|)^{-n-s}\,dy \le\theta,\\ \,\\
\displaystyle\int_{B_4}w_+(y)\,dy\ge 1.
\end{cases} $$ 
Then~$w>0$ in~$B_2$. 
\end{lemma}

\begin{proof} See Lemma 6.2 in~\cite{caffa}. 
\end{proof}

In order to prove Proposition~\ref{propimpliesharnack} (and so
Theorem \ref{harnack}),
we also need the following observation:

\begin{lemma}\label{approxsol}
Let $\phi := \phi_0(\,\cdot\,/\eps)$ 
and $\phi^b:= \phi\circ d_b$, where $d_b$ 
is defined in \eqref{defd} (see also \eqref{formulad}).

Then, 
\[ \big| L\phi^b - \eps^{-s} f(\phi^b) \big| \le C (b+ \eps^{\gamma_0})  \quad \mbox{in }B_4,\]
where $C>0$ is a universal constant
and~$\gamma_0>0$ is the constant given by Lemma~\ref{lemdecay}.
\end{lemma}

\begin{proof}
By~\eqref{DEC}, we have that~\eqref{imageassumption}
is satisfied with~$\delta:=C \eps^{\gamma_0}$.
Hence, using Lemma \ref{lemdeform} (scaled to $B_4$ and with~$f_\phi:=\eps^{-s} f$),
we obtain that $\big| L\phi^b - \eps^{-s} f(\phi^b) \big|
\le C(b+\delta)$. The desired result now plainly follows.
\end{proof}

With this, we are in the position of proving Proposition \ref{propimpliesharnack}.

\begin{proof}[Proof of Proposition \ref{propimpliesharnack}]
In all the proof we denote
\[ C_{r} := B'_r\times(-2^{k_a}, 2^{k_a}).\]
Fix $z'\in B'_{1/2}$ and let 
\begin{equation}\label{translate}
 \bar u(x) := u(x'-z', x_n).
\end{equation}
By assumptions, we have
\begin{equation}\label{control1+a}
\{ \bar u\le 1-\kappa\} \subset
 \left\{x_n \le a + \frac 1 2\, b\, \,\xi(x') 
\right\} \quad \mbox{in }
C_{ 2^{k_a}}
\end{equation}
for
\begin{equation}\label{deb} 
b := Ca,\end{equation}
where $C>0$ depends only on $\alpha$ and~$\xi$ was defined in~\eqref{xi}.

Throughout the proof, we use the notations 
\begin{equation}\label{osidphi-0p}
\phi(t):= \phi_0\left(\frac{t}{\eps}\right)\quad\mbox{and} \quad\phi^b(x):=
\phi\circ d_b(x).
\end{equation}

The idea of the proof is to consider the infimum $h_*$ among all the $h\ge 0$ such that 
\begin{equation}\label{stop-or-go}
\min_{x\in\overline{B_1}}\left (  \bar u(x)-\phi^{b}(x- he_n)\right) \ge0.
\end{equation}
We will indeed observe that such~$h_*$ is well defined. Then,
we will show that
\begin{equation}\label{as01ls}
h_* < a(1-\eta)
\end{equation}
for a suitable and universal~$\eta\in(0,1)$.
The proof of~\eqref{as01ls} will be done by contradiction (namely, we will
show that the
inequality $h_*\ge a-\eta a$ leads to a contradiction). Then,
from the inequality in~\eqref{as01ls}, the claim in
Proposition~\ref{propimpliesharnack} will follow in a straightforward way.
\vspace{5pt}

\noindent{\em Step 1}. Let us show first 
that \emph{if $h\ge a+3$ then~\eqref{stop-or-go} 
holds true}.

First, we claim that 
\begin{equation}\label{decayphi}
\phi^{b}(x- he_n)  \le -1 + 
\frac{C \eps^{\gamma_0} }{\big(x_n-h-b\xi(x')\big)_-^{\gamma_0}}  \quad \mbox{for all }x\in C_{2^{k_a-1}}
\end{equation}
and
\begin{equation}\label{decayu}
\bar u (x)\ge 1 - \frac{C\eps^{\gamma_0}}{ \left(x_n-a-\frac 1 2\, b\,\xi(x')\right)_+^{\gamma_0}} \quad \mbox{for all }x\in C_{2^{k_a-1}}.
\end{equation}

To prove \eqref{decayphi} and \eqref{decayu}, it is important to observe that, by \eqref{a2ka}, 
\begin{equation}\label{Lip}
 |\nabla (b\xi)(z')|\le Ca (1+|z'|^2)^{\frac{\alpha-1}{2}}|z'|\le  C a 2^{k_a} \le C2^{-m_0} \quad \mbox{ for all }z'\in B'_{2^{k_a}}.
\end{equation}
Now, to show \eqref{decayphi}, we use the decay properties 
of~$\phi_0$ in Lemma~\ref{lemdecay}, which imply that,
for all $h\ge0$,
\begin{equation}\label{alm:01}
\phi^b(x-he_n) =\phi_0\left(\frac{d_b(x-he_n)}{\eps}\right) \le -1 + \frac{C \eps^{\gamma_0} }{ \big( d_b(x-he_n) \big)_-^{\gamma_0}}.
\end{equation}
Also, as a consequence of \eqref{Lip}, 
we see that, for all $y\in B'_{2^{k_a}}\times \R$, 
\begin{equation}\label{alm:02}
\big(d_b(y)\big)_-  \ge c\,\big(y_n-b\xi(y')\big)_-, 
\end{equation}
for some~$c>0$ depending only on~$\rho$ and~$\rho'$
(for more details see Lemma~8
in~\cite{torino}).

Now, making use of~\eqref{alm:01} and~\eqref{alm:02} 
(with $x\in B_{2^{k_a}}$
and $y:=x-he_n$),
we deduce \eqref{decayphi}.

Let us now prove \eqref{decayu}. 
To do it, given $x\in C_{2^{k_a-1}}$, define $R=R(x)$ to be the the largest radius for which 
\[
B_R(x)\subset C_{2^{k_a}}\cap\left\{ y_n>a+ \frac 1 2 \,b\,\xi(y')\right\}.
\]

By ~\eqref{control1+a}, we know that $u(y)\ge1-\kappa$
for any~$y\in B_{2^{k_a}}$ with~$y_n>a+ \frac 1 2 \,b\,\xi(y')$ and by assumption $u$ solves $Lu=\eps^{-s} u$ in $C_{2^{k_a}}$. Hence, 
using Lemma \ref{lemdecay} we obtain 
\begin{equation} \label{0idjIA}
u(x)\ge 1-\frac{C\eps^{\gamma_0}}{R^{\gamma_0}}.\end{equation}

Now we observe that, by \eqref{Lip}, for any~$x\in C_{2^{k_a}/2}$
with~$x_n>a+ \frac 1 2 \,b\,\xi(x')$ we have
$$ R(x) \ge  c\, \left(x_n-a-\frac 1 2\, b\,\xi(x')\right)_+,$$
as long as~$c>0$ is sufficiently small. Hence, \eqref{decayu} follows.

Now we remark that 
\begin{equation}\label{grghswew8657}
\big(x_n-a-  {\textstyle \frac 1 2}\,b \xi(x')\big)-\big(x_n-h- b\, \xi(x')\big) = h-a +\frac {b}{2} \, \xi(x').
\end{equation}
Hence, since we are now assuming that $h-a\ge 3> 2$,
we deduce from~\eqref{grghswew8657} that
$$ \big(x_n-a-  {\textstyle \frac 1 2}\,b \xi(x')\big)-\big(x_n-h- b\, \xi(x')\big)
\ge 1+\frac {b}{2} \, \xi(x')\ge 1.$$
Consequently,
\begin{eqnarray}
&& \label{kaos:1}
\mbox{either}\quad
\big(x_n-a-\frac 12 \,b \,\xi(x')\big)\ge 1\quad\\
&& \label{kaos:2}
\mbox{or}\quad \big(x_n-h-b\xi(x')\big) \le -1.\end{eqnarray}
Now we claim that
\begin{equation}\label{aggiunto}
\bar u(x)- \phi^{b}(x- he_n) \ge -C \eps^{\gamma_0}  \quad 
\mbox{for any }x\in C_{2^{k_a-1}}.\end{equation}
For this, we distinguish two cases, according to~\eqref{kaos:1}
and~\eqref{kaos:2}. If~\eqref{kaos:1} is satisfied, then 
we exploit~\eqref{decayu} and the fact that~$\phi^{b} \le 1$ to find that
$$ \bar u(x)-
\phi^{b}(x- he_n)  \ge \bar u(x)-1 \ge
- \frac{C\eps^{\gamma_0}}{ \left(x_n-a-\frac 1 2\, b\,\xi(x')\right)_+^{\gamma_0}} 
\ge -C\eps^{\gamma_0},$$
up to renaming~$C>0$, which gives~\eqref{aggiunto} in this case.

If instead
the inequality in~\eqref{kaos:2} holds true, we use~\eqref{decayphi} and the fact that $\bar u\ge-1$ to see that
$$ \bar u(x)-
\phi^{b}(x- he_n)  \ge \bar u(x)+1 -
\frac{C \eps^{\gamma_0} }{\big(x_n-h-b\xi(x')\big)_-^{\gamma_0}}
\ge -\frac{C \eps^{\gamma_0} }{\big(x_n-h-b\xi(x')\big)_-^{\gamma_0}}
\ge -C \eps^{\gamma_0}
,$$
up to renaming constants,
and this completes the proof of~\eqref{aggiunto}.

Furthermore, since $\xi$
is a nonnegative function with~$\xi(0)=0$, 
the affine function~$\ell(x):= x_n/\tilde c$, 
with~$\tilde c = h_L(e_n)>0$,
is admissible in~\eqref{defd}. As a consequence, 
we obtain that~$d_b(x)\le x_n/\tilde c$.
Accordingly, from the monotonicity of~$\phi$,
we have that 
\begin{equation}\label{KA0yrewf8idaa}
{\mbox{$\phi_b(x)= \phi( d_b(x)) \le \phi( x_n/\tilde c)$ for all $x\in \R^n$.}}\end{equation}
Now, since in this case~$h\ge a+3\ge 3$, we observe that, for any~$x\in B_2$,
$$ \frac{x_n-h}\eps \le \frac{2-3}{\eps}=-\frac1\eps$$
and so, if~$\eps$ is large enough, 
$$ \sup_{B_2} \phi_0 \left(\frac{x_n-h}{\tilde c\eps} \right) \le -\frac{1}{2}.$$
Therefore,
recalling the assumption~\eqref{signchoice} and~\eqref{KA0yrewf8idaa},
\begin{equation}\label{conseq-int}
\begin{split}
&\int_{B_2} \bar u(x)- \phi^{b}(x-he_n) \,dx 
\ge   \int_{B_2} \bar u(x)- \phi (\tilde c(x_n -h)) \,dx \\
&\qquad\qquad= \int_{B_2} \bar u(x)- \phi_0 \left(\tilde c\, \frac{x_n-h}{\tilde c\eps} \right) \,dx
\ge 0- \int_{B_2}  \phi_0 \left(\frac{x_n-h}{\tilde c\eps} \right) \,dx 
\ge c,\end{split}
\end{equation}
where $c>0$ is 
a universal constant.

We consider now the function~$w(x) := \bar u(x)- \phi^{b}(x-he_n)$. Let us show that  
\begin{equation}\label{eqdifference}
L w \ge -C(b+ \eps^{\gamma_0}) \quad \mbox{in } \{w\le0\}\cap B_4
.\end{equation}
Indeed, let 
\[
\Omega : = \{w\le0\} \cap \big( \{u\ge 1-\kappa \}\cup \{\phi^b (\cdot-he_n)\le -1+\kappa\}\big).
\]
To start with, we will show that
\begin{equation}\label{iosdjvk52364asd7}
\big( \{w\le0\}\cap B_4\big)\setminus \Omega=\varnothing.
\end{equation}
Indeed, suppose, by contradiction, that there exists a point~$
y\in \big( \{w\le0\}\cap B_4\big)\setminus \Omega$. Then, 
\begin{equation}
\bar u(y)<1-\kappa \;\;{\mbox{ and }} \; \;
\phi^{b}(y-he_n)>-1+\kappa.
\label{2pos}\end{equation}
Thus, by \eqref{control1+a}, we see that
$$ 0\ge y_n-a-\frac 12 b\xi(y') =
y_n-h+h-a-\frac 12 b\xi(y')\ge y_n-h+3-\frac 12 b\xi(y').$$
Therefore
$$ y_n-h-b\xi(y')=
y_n-h+3-\frac 12 b\xi(y')-3-\frac 12 b\xi(y')\le
0-3-\frac 12 b\xi(y')<0.$$
Hence, we can use~\eqref{decayphi}, which gives that
$$\phi^b (y_n-he_n)\le -1+C\eps^{-\gamma_0},$$ 
up to renaming~$C>0$.
Thus, for $\eps$ small,
we deduce that~$\phi^{b}(y-he_n)\le-1+\kappa$,
which gives that the second inequality in~\eqref{2pos} cannot occur.
This contradiction establishes~\eqref{iosdjvk52364asd7}.

Hence, in view of~\eqref{iosdjvk52364asd7},
to complete the proof of~\eqref{eqdifference},
we only need to show that~\eqref{eqdifference} holds true in~$\Omega\cap B_4$.
To this aim, we
take $y\in \Omega\cap B_4$. Then,  $w(y)\le0$ and so~$\bar u(y)\le\phi^b(y-he_n)$. Therefore, 
using Lemma~\ref{approxsol},
\begin{equation}\label{qs0ioja2kx}
\begin{split}
L w(y) =L\bar u(y)-L\phi^b(y-he_n) &\ge \eps^{-s}f\big(\bar u(y)\big)- 
\eps^{-s}f\big(\phi^b(y-he_n)\big) - C\,(b+ \eps^{\gamma_0})
\\
&\ge \eps^{-s} f'(\xi)\, w(y) - C\,(b+ \eps^{\gamma_0}),
\end{split}
\end{equation}
where $C>0$
and~$\xi = \xi(y)$ belongs to the real interval $\big[\bar u(y),\, \phi^b (y\cdot-he_n) \big]$.

We also recall that by \eqref{assumpf} we have that
$f'\le 0$ in $[-1, -1+\kappa ]\cup[1-\kappa, 1]$. 
Moreover, by the definition of $\Omega$, we have that either 
$1-\kappa \le \bar u(y)<\phi^b (y-he_n)\le 1$ or
$-1 \le \bar u(y)<\phi^b (y-he_n)\le -1+\kappa$. In any case, 
we have that $f'(\xi)\le 0$ and so~\eqref{eqdifference}
follows
from~\eqref{qs0ioja2kx}.

Now, putting together~\eqref{eqdifference}, \eqref{aggiunto}
and~\eqref{conseq-int},
we have proven that $w$ satisfies 
$$ \begin{cases}
Lw\ge -C(b+\eps^{\gamma_0}) \quad {\mbox{in }}B_4\cap\{w<0\},\\
w\ge-C\eps^{\gamma_0} \quad {\mbox{in }} C_{2^{k_a-1}},\\
w\ge -2 \quad {\mbox{in }} \R^n\setminus C_{2^{k_a-1}},\\
\displaystyle\int_{B_2} w(y)\,dy\ge c.
\end{cases} $$

Note that 
\[
\int_{\R^n} \frac{w^-(y)}{(1+|y|)^{n+s}}\,dy \le C \eps^{\gamma_0} + \int_{|y|\ge 2^{k_a-1}} \frac{2dy}{|y|^{n+s}} \le C \eps^{\gamma_0} + C2^{-sk_a}. 
\]
Then, choosing~$a_0$ small enough (that corresponds to~$k_a$ large in view of~\eqref{kdia}),
we fall under the assumptions of
Lemma \ref{maxpr}, which yields that $w>0$ in $B_2$.
This plainly implies the desired statement for~{\em Step 1}.

\vspace{5pt}

{\em Step 2}. Let  
 \[h_*:=  \inf \ \big\{\ h\ge 0 \ : \ \mbox{\eqref{stop-or-go} holds}\ \big\}. \]
Notice that the infimum is taken over a nonempty set, 
thanks to {\em Step 1}, and indeed
$h_*\le a+3 < +\infty$.
We next show that
\begin{equation}\label{maingoal}
h_*< a-\eta a \quad\mbox{as long as }\eta>0\mbox{ 
is sufficiently small}.
\end{equation}
The proof of~\eqref{maingoal}
will be by contradiction, namely we will show
that the two conditions~$h_*\ge a-\eta a$ and~$\eta$ 
small enough 
lead to a contradiction (for an appropriately small~$a_0$). 

To this aim, we define
\[\phi_* (x):= \phi^b(x-h_*e_n).\]
We observe that, by the definition of $h_*$, we have that $u-\phi_* \ge 0$ in $B_1$.

Under this assumption, we will prove that
\begin{equation}\label{iokdvfpodsgposdguqqq}
\bar u-\phi_* > 0\quad \mbox{in } B_2,
\end{equation}
which contradicts the definitions of~$h_*$ and~$\phi_*$.

Indeed, using the contradictory assumption 
that~$h_*\ge a-\eta a$, we have
\[
\big(x_n-a-{\textstyle \frac 1 2 } \, b\,\xi(x')\big)-\big(x_n-h_*-b\, \xi(x')\big) 
= h_*-a + \frac b 2\,  \xi(x') \ge  \frac b 2  \, \xi(x') -\eta a.
\]
Then, if $\eta$ is small enough we have, for all $x \in C_{2^{k_a-1}}\setminus B_1$, 
$$ \big(x_n-a-{\textstyle \frac 1 2 } \, b\,\xi(x')\big)-\big(x_n-h_*-b\, \xi(x')\big) 
\ge  \frac b 2  \, \xi(1/2) -\eta a\ge \frac{b}{8}$$
where we have used that $b=Ca$ (recall~\eqref{deb}).

Therefore, for all $x \in C_{2^{k_a-1}}\setminus B_1$, 
\[ {\mbox{either}}\quad
\big(x_n-a-{\textstyle \frac 12}\,b\,  g(x')\big) \ge  
\frac{b}{16}
\quad{\mbox{or}}\quad \big(x_n-b\, g(x')\big) \le -\frac{b}{16}.\] 
Thus, similarly as in {\em Step 1}, 
using either~\eqref{decayphi} and the fact 
that~$\bar u\ge-1$, or~\eqref{decayu} 
and~$\phi_{*} \le 1$, we obtain that 
\[  \bar u- \phi_* \ge -C (\eps/b)^{\gamma_0} \quad \mbox {in } C_{2^{k_a-1}},\]
for some~$C>0$.

Next, similarly as in {\em Step 1},  the function~$w:= u-\phi_*$ satisfies 
\begin{equation}\label{09disovy8reugifdbjbf} \begin{cases}
Lw\ge -C(b+\eps^{\gamma_0} )\quad {\mbox{in }}B_4\cap\{w\le0\},\\
w\ge-C(\eps/b)^{\gamma_0} \quad {\mbox{in }} B_{2^{k_a-1}}\setminus B_4,\\
w\ge -2 \quad {\mbox{in }} \R^n\setminus B_{2^{k_a-1}},\\
\displaystyle\int_{B_2} w(y)\,dy\ge c\, h_*\ge ca,
\end{cases} \end{equation}
up to renaming~$c>0$. 

Notice now that, recalling~\eqref{deb},
\[
\begin{split}
\frac 1 a \int_{\R^n} \frac{w^-(y)}{(1+|y|)^{n+s}}\,dy
& \le \frac{C}{a} \left(  \frac{\eps}{b}\right)^{\gamma_0} + \int_{|y|\ge 2^{k_a-1}} \frac{2dy}{|y|^{n+s}} \le 
\frac{C}{a}   \left(\frac{\eps}{a}\right)^{\gamma_0} + \frac{C}{a}2^{-sk_a}
\\
&\le \frac{C}{a}   \left(a^{p_0-1}\right)^{\gamma_0} + C2^{-(s-\alpha)k_a} \frac{2^{-\alpha k_a}}{a}
\\
&\le  C\big(a^{(p_0-1)\gamma_0-1} + C_{m_0} 2^{-(s-\alpha)k_a}\big) \  \rightarrow 0 
 \qquad \mbox{as }a\downarrow 0,
\end{split}
\]
where~$C_{m_0}>0$ depends on~$m_0$.
Similarly,
\begin{equation*}\label{epsb}
\frac{C}{a}\left(b+ \eps^{\gamma_0} +\left(\frac{\eps}{b}\right)^{\gamma_0}\right) \le Ca^{(p_0-1)\gamma_0-1}\  \rightarrow 0 \qquad \mbox{as }a\downarrow 0.
\end{equation*}
Then, choosing~$a_0$ small enough, we can apply  
Lemma~\ref{maxpr} to show that~$w> 0$ in~$B_2$, 
thus proving~\eqref{iokdvfpodsgposdguqqq},

Now, by the definition of~$h_*$, we know that there exists a point $x_*\in \overline{B_1}$ such that 
$ w(x_*) = \bar u(x_*)-\phi_*(x_*) =0$.
This is in contradiction with~\eqref{iokdvfpodsgposdguqqq}. Therefore, we have proved \eqref{maingoal} and completed the proof of {\em Step 2}. 
\vspace{5pt}

{\em Step 3.} We now complete the 
proof of Proposition~\ref{propimpliesharnack}.
For this, we recall the definition of~$\bar u$ in~\eqref{translate}
and we prove that
\begin{equation}\label{ZX0fue9 guic}
\{ \bar u\le 1-\kappa\}\subset 
\left\{x_n \le a\,\left(1-\frac\eta2
\right)  \right\}
\quad \mbox{on } \{0\}\times (-1,1).
\end{equation}
Indeed, 
by {\em Step 2}, we know that
$$\bar u(x)-\phi^b(x -a(1-\eta)e_n)\ge 0.$$
Moreover (see e.g. Lemma~7
in~\cite{torino}), we 
have that, on $\{x'=0\}\times (-1,1)$,
$$d_b(x -a(1-\eta)e_n)\ge\frac{x_n-
a(1-\eta)}{\tilde c}$$
for some~$\tilde c>0$,
and so
$$ \phi^b(x -a(1-\eta)e_n) \ge
\phi_0\left(\frac{x_n-
a(1-\eta)}{\tilde c\eps}\right)$$ 
on $\{x'=0\}\times (-1,1)$. Therefore,
we have that
\begin{equation*}\begin{split}&
 \{x_n\in(-1,1)\ : \ \bar u (0, x_n) \le 1-\kappa\} \subset  
\left\{ x_n\in(-1,1)\ : \
\phi_0\left(\frac{x_n-
a(1-\eta)}{\tilde c\eps}\right) \le1-\kappa\right\}\\&\qquad
\quad\subset \left\{
\frac{x_n-
a(1-\eta)}{\tilde c\eps}\in
\left(-\infty, l_\kappa\right)\right\}
\subset \left\{
\frac{x_n}{a}<\frac{\tilde c\eps l_k}{a}+
(1-\eta)
\right\}
\subset \left\{
\frac{x_n}{a}<1-\frac{\eta}{2}
\right\}, \end{split}
\end{equation*}
where $l_\kappa$ has been introduced in~\eqref{lkappa},
and the last inclusion holds since~$\eps/a$ is as 
small as desired.
This estimate establishes~\eqref{ZX0fue9 guic}, as desired.

Now, from~\eqref{translate} and~\eqref{ZX0fue9 guic}, we obtain that
\begin{equation}\label{78:013w84:PRE}
\{x_n\in(-1,1)\ : \  u (x', x_n) \le 1-\kappa\} 
\subset  \left\{x_n \le a\,\left(1-\frac\eta2
\right)  \right\}, 
\end{equation}
where $x'\in B'_{1/2}$ is arbitrary.

Now, to complete the proof of
Proposition \ref{propimpliesharnack},
let~$x=(x',x_n)\in\{u\le1-\kappa\}$, with~$|x'|<1/2$
and~$|x_n|<2^{k_a}$. Then, using~\eqref{01L:A}
with~$j=0$, we obtain that
\begin{equation}\label{78:013w84}
x_n\le a<1.\end{equation}
Now, if~$x_n\le0$, then~\eqref{conclusionPropHar}
is obviously true, so we may assume that~$x_n>0$.
Thanks to this and~\eqref{78:013w84}, we are in position
of using~\eqref{78:013w84:PRE},
which in turn implies~\eqref{conclusionPropHar},
as desired.\end{proof}

With this, we are now in the position of completing the proof of Theorem~\ref{harnack}.

\begin{proof}[Proof of Theorem~\ref{harnack}] If~\eqref{signchoice} holds true, \label{PP:p001}
the claim follows from Proposition~\ref{propimpliesharnack}.
If instead the opposite inequality in~\eqref{signchoice} holds,
we look at~$\tilde u:=-u$, which satisfies
$$ L\tilde u=-\epsilon^{-s}f(-\tilde u)=:\epsilon^{-s}\tilde f(\tilde u).$$
Since $\tilde f$ satisfies the same structural conditions as $f$ in~\eqref{assumpf}
and~\eqref{existslayer}, and now~$\tilde u$ satisfies~\eqref{signchoice},
we can apply Proposition~\ref{propimpliesharnack} to~$\tilde u$ and obtain the
desired result.
\end{proof}

Rescaling and iterating  Theorem \ref{harnack} we obtain the following result:

\begin{corollary} \label{cor:globalCalpha}
There exist constants $a_0>0$, $p_0>2$, $\sigma>0$
and~$C>0$, depending only on $\alpha$, $m_0$, and on universal constants, with $\sigma$ satisfying $\alpha(1+\sigma)<s$, such the the following statement holds. 

Let $a\in(0,a_0)$ and $\eps \in (0, a^{p_0})$. 
Let $k_a$ be given by \eqref{kdia}.
Assume that $u_a: \R^n \rightarrow (-1,1)$ is a
solution of $Lu=\epsilon^{-s}  f(u)$
in $B'_{2^{k_a}}\times (-2^{k_a}, 2^{k_a})$ such that
\begin{equation}\label{trapping}
\{x_n\le -a 2^{j(1+\alpha)}\}\, \subset\, \{u_a\le -1+\kappa\} \,\subset \,\{u_a\le 1-\kappa\}  \,\subset \,\{x_n \le a 2^{j(1+\alpha)}\} \quad \mbox{in }B'_{2^j}\times (-2^{k_a}, 2^{k_a})
\end{equation}
for $0\le j\le k_a$.

Then, there exist two functions $g_a=  g_a(x')$ and $ g^a=  g^a(x')$ belonging to $C^\sigma(B'_{2^{k_a-1}})$ and satisfying $g_a\le g^a$  such that, for all $R\in [1, 2^{k_a-1}]$,
we have
\begin{equation}\label{normcontrol}
\|g_a\|_{L^\infty(B_R)} +  R^\sigma[g_a]_{C^\sigma(B_R)} \le CR^{1+\alpha(1+\sigma)}, \quad 
\|g^a\|_{L^\infty(B_R)} +  R^\sigma[g^a]_{C^\sigma(B_R)} \le CR^{1+\alpha(1+\sigma)},
\end{equation}
\begin{equation}\label{separationcontrol}
\| g_a - g^a\|_{L^\infty(B_R)} \le C R^{1 +\alpha(1+\sigma)} a^{1+\sigma}   ,
\end{equation}
and
\[
\{x_n \le a g_a(x')\}\, \subset\, \{u_a\le -1+\kappa\} \,\subset \,\{u_a\le 1-\kappa\}  \,\subset \, \{x_n \le a  g^a(x')\}
\quad \mbox{in }B'_{2^{k_a-1}}\times(-2^{k_a}, 2^{k_a}). 
\]

In particular, the two functions $g_a$ and $g^a$ converge 
locally uniformly as~$a\to0$
to some H\"older continuous function $g$ satisfying the growth control $g(x')\le C(1+|x'|)^{1+\alpha(1+\sigma)}$. 
\end{corollary}
\begin{proof}
The proof of this result follows from iterating and rescaling the  Harnack inequality of  Theorem \ref{harnack}; see \cite{Savin, CRS} for similar arguments.
\medskip

{\em Step 1}. We first prove the following claim 
which states that the 
transition region
is trapped near the origin between two 
H\"older functions that are separated by a very small distance near the origin.

Throughout the proof we denote by~$C_r := B'_r\times (-2^{k_a}, 2^{k_a})$.

\hspace{6pt}
{\bf Claim.} {\it For some $(0,z_n)\in \{-1+\kappa \le u_a \le 1-\kappa \}$ we have 
\begin{equation}\label{csigmaat0}
\big\{x_n \le z_n- a C(|x'|^\sigma +r)\big\} \subset\{u_a\le -1+\kappa\} \subset \{u_a\le 1-\kappa\}  \subset  \big\{x_n \le z_n +a C(|x'|^\sigma+r) \big\}
\quad \mbox{in }C_{1}, 
\end{equation}
for 
 \[  r := 8\left(a_0\right)^{-\frac{1}{1-\sigma}} a^{\frac{1}{1-\sigma}-1} ,\] 
where $a_0>0$ is the small constant 
in Theorem \ref{harnack} and where $C>0$ and $\sigma\in (0,1)$ depend only on  $\alpha$, $m_0$, and on universal constants.}
\hspace{6pt}

Let us prove that for every integer~$l\ge0$,  satisfying
\begin{equation}\label{stop}
 a2^{(1-\sigma)l}  <a_0,
\end{equation}
we have that
\begin{equation}\label{indu}
\{  x_n\le c_l-a 2^{-\sigma l}\}\, \subset\, \{u_a\le -1+\kappa\} \,\subset \,\{u_a\le 1-\kappa\}  \,\subset \,\{x_n\le c_l + a 2^{-\sigma l}\} \quad \mbox{in }C_{2^{-l}}
\end{equation}
where $c_l\in \R$ satisfy
\begin{equation}\label{nested}
c_l-a2^{-\sigma l}\le c_{l+1}-
a2^{-\sigma (l+1)} \le c_{l+1}+a2^{-\sigma (l+1)} 
\le c_{l} +a2^{-\sigma l}.
\end{equation}
The proof is by induction over the integer~$l$. Indeed,
it follows from~\eqref{trapping}
that \eqref{indu} holds true for~$
l=0$, with 
\begin{equation}\label{chico}
c_0=0\end{equation}

Assume now that \eqref{indu} holds true
for $0\le l \le l_0$, 
and let us prove that~\eqref{indu}
is also satisfied for $l=l_0+1$.
For this,
let
\[ U(x) := u_a\big(2^{- l_0}x', 2^{-l_0}x_n +  c_{l_0}\big).\]
We have 
\begin{equation}\label{epsprime}
 LU= \left(\frac{\eps}{2^{-l_0}} \right)^{-s} f(U) \quad \mbox{in }C_1 
.\end{equation}
To abbreviate the notation we define
\[ \mathcal A:=   \{U\le -1+\kappa\} \quad \mbox{and}\quad \mathcal B:= \{U\le 1-\kappa\}.\]
We claim that 
\begin{equation}\label{induGOAL}
\{  x_n\le  -a 2^{(1-\sigma)l_0} 2^{j(1+\alpha)} \}\, 
\subset \mathcal A\subset \mathcal B\subset 
\,\{x_n\le   a 2^{(1-\sigma)l_0} 2^{j(1+\alpha)} \} \quad \mbox{in }C_{2^{j}},
\end{equation}
for $j=0,\dots, k_a$.
As a matter of fact, to prove~\eqref{induGOAL}, we
first show that it holds for~$j=0$, then for~$j=1,\dots,l_0$
and then we complete the argument by showing that~\eqref{induGOAL}
holds also for~$j=l_0+1,\dots,k_a$.

To this aim, we observe that,
since \eqref{indu} holds for $0\le l \le l_0$, we have 
\begin{equation}\label{induU}
\{  x_n\le 2^{l_0}(c_l-c_{l_0})-a 2^{l_0-\sigma l}\}\,
\subset \mathcal A\subset \mathcal B\subset 
\,\{x_n\le 2^{l_0}(c_l-c_{l_0}) + a 2^{l_0-\sigma l}\} 
\quad \mbox{in }C_{2^{l_0-l}},
\end{equation}
for any~$0\le l \le l_0$.
This, when~$l=l_0$, gives~\eqref{induGOAL} for~$j=0$.

Hence, we focus now on the proof of~\eqref{induGOAL}
when~$j=1,\dots,l_0$. For this, 
we can suppose that
\begin{equation}\label{l01+}
l_0 \ge 1,\end{equation}
otherwise this case is void, and
we will use~\eqref{induU}
with~$l=0,\dots,l_0-1$.
We remark that the inequalities
in~\eqref{nested} imply that, for any~$0\le l\le l_0-1$,
\[ c_{l} - a2^{-\sigma l}\le c_{l_0} - a2^{-\sigma l_0}\le
c_{l_0} + a2^{-\sigma l_0}\le c_{l} + a2^{-\sigma l}.\]
Therefore
\begin{eqnarray*}&& c_l\le c_{l_0} - a2^{-\sigma l_0}+a2^{-\sigma l}
\le c_{l_0} +a2^{-\sigma l}\\
{\mbox{and }}&&
c_{l_0}\le c_{l} + a2^{-\sigma l}-a2^{-\sigma l_0}\le
c_{l} + a2^{-\sigma l}.\end{eqnarray*}
Accordingly, 
we have that, for $0\le l<l_0-1$,
$$ |c_l-c_{l_0}|\le a2^{-\sigma l}$$
and so
\[
2^{l_0}|c_l-c_{l_0}| + a 2^{l_0-\sigma l}\le 
2a 2^{l_0-\sigma l} = a 2^{(1-\sigma)l_0} 2^{\sigma (l_0-l)+1}.
\]
{F}rom this and~\eqref{induU}, using the notation~$j:=l_0-l$,
we see that, for any~$j=1,\dots,l_0$,
\begin{equation}\label{lolo00}
\{  x_n\le -a 2^{(1-\sigma)l_0} 2^{\sigma j+1}\}\,
\subset \mathcal A\subset \mathcal B\subset 
\,\{x_n\le 
a 2^{(1-\sigma)l_0} 2^{\sigma j+1}\} 
\quad \mbox{in }C_{2^{j}},
\end{equation}
We also observe that, for any~$j=1,\dots,l_0$,
taking~$\sigma\le\alpha$, we have that
$$ (\sigma j +1)-(1+\alpha)j\le 
(\alpha j +1)-(1+\alpha)j = 1-j\le0$$
and thus
$$ 2^{\sigma j +1} 
\le 2^{(1+\alpha)j}.$$
So, we insert this into~\eqref{lolo00}
and we complete the proof of~\eqref{induGOAL}
for $j=1,\dots l_0$.

To complete the proof of~\eqref{induGOAL},
we have now to take into account the case~$j=l_0+1,\dots,
k_a$. For this,
we recall assumption \eqref{trapping} (used here with
the index~$i$)
and we obtain that
\begin{equation}\label{induU5}
\{  x_n\le  -2^{l_0}c_{l_0}-a 2^{l_0+i(1+\alpha)}\}\, 
\subset \mathcal A\subset\mathcal B\subset 
\,\{x_n\le   -2^{l_0}c_{l_0}+ a 2^{l_0+i(1+\alpha)}\} 
\quad \mbox{in }C_{2^{l_0+i}}
\end{equation}
for $i=0,\dots, k_a$ (in our setting,
we will then take~$j=l_0+i$,
with~$i=1,\dots,k_a-l_0$). Now, we point out that
\begin{eqnarray*}&&
c_{l_0}+a2^{-\sigma l_0}\le c_{0}+a2^{-\sigma \cdot 0}= a\\
{\mbox{and }}&& -a=c_{0}-a2^{-\sigma \cdot 0}\le c_{l_0}-a2^{-\sigma l_0}
,\end{eqnarray*}
thanks to~\eqref{nested} and~\eqref{chico}.
Consequently, we have that~$|c_{l_0}|\le a$ and so
\begin{equation}\label{0938}
2^{l_0}|c_{l_0}|+ a 2^{l_0+i(1+\alpha)}\le
a 2^{l_0}(1+ 2^{i(1+\alpha)})\le
a 2^{l_0+1+i(1+\alpha)}
.\end{equation}
We also observe that, taking~$\sigma\le\alpha$ and using~\eqref{l01+},
\begin{eqnarray*}&& l_0+1+i(1+\alpha) =1+
(\sigma-1-\alpha)l_0+(1-\sigma)l_0+(i+l_0)(1+\alpha)
\\&&\qquad\le1
-l_0+(1-\sigma)l_0+(i+l_0)(1+\alpha)\le
(1-\sigma)l_0+(i+l_0)(1+\alpha).
\end{eqnarray*}
This and~\eqref{0938} give that
$$ 2^{l_0}|c_{l_0}|+ a 2^{l_0+i(1+\alpha)}\le
a2^{(1-\sigma)l_0+(i+l_0)(1+\alpha)}.$$
Plugging this into~\eqref{induU5} with~$i=j-l_0$,
we obtain \eqref{induGOAL}
for $j= l_0+1,\dots, k_a$. 

These considerations complete the proof of~\eqref{induGOAL}.
Next, in view of \eqref{induGOAL}, we may apply Theorem~\ref{harnack} with $u$ replaced by $U$, with $a$ replaced by 
\[
\bar a := a 2^{(1-\sigma)l_0},
\]
and with $\eps$ replaced by 
\[
\bar\eps := \frac{\eps}{2^{-l_0}}. 
\]
Note that, since we assume that $\eps<a^{p_0}$,
the condition $\bar \eps< \bar a^{p_0}$ holds whenever
\[
 \frac{a^{p_0}}{2^{-l_0}} < \big(a 2^{(1-\sigma)l_0}\big)^{p_0}.
\]
This is equivalent to
\[ 
1 <  2^{((1-\sigma)p_0-1)l_0}
\]
which is always satisfied when $p_0>2$ and $\sigma$ is taken small. 
 
We recall however that, in order to 
apply Theorem~\ref{harnack}, 
we must have that~$\bar a$ is less than the
small universal constant~$a_0$. 
This is the reason why we need condition~\eqref{stop} 
to continue the iteration.

Thanks to these observations and~\eqref{induGOAL},
we can thus apply Theorem \ref{harnack}. In this way,
we have
proved that  \eqref{indu} holds whenever  \eqref{stop} holds, which immediately implies the statement of 
the claim. 
\medskip

{\em Step 2}. To complete the proof 
of Corollary~\ref{cor:globalCalpha}, let us fix a nonnegative integer $l\le k_a-1$ and $z'\in B'_{2^l}$.
Here, we define
\[ U(x) := u\big(z'+2^lx', 2^l x_n\big).\] 
Then, rescaling \eqref{trapping} we find
\begin{equation}\label{trappingresc:0}
\{x_n\le - 2^{-l}a 2^{(l+i+1)(1+\alpha)} \}\, \subset\, \{U\le -1+\kappa\} \,\subset \,\{ U\le 1-\kappa\}  \,\subset \,\{x_n \le  2^{-l} 2^{(l+i+1)(1+\alpha)}\} 
\end{equation}
in $B'_{2^{i}}\times (-2^{k_a-l}, 2^{k_a-l})$,  for $0 \le i\le k_a-l-1$.

Let us denote
\[ \bar a := 2^{-l} 2^{(l+1)(1+\alpha)} a = 2^{(l+1)\alpha+1} a.\]
Observe that, recalling the definition of~$k_a$ 
in~\eqref{kdia}, we have  
\[
k_{\bar a} < k_a-l-1.
\]
Thus, \eqref{trappingresc:0} implies that
\begin{equation}\label{trappingresc}
\{x_n\le -\bar a 2^{i(1+\alpha)}\}\, 
\subset\, \{U\le -1+\kappa\} \,\subset 
\,\{ U\le 1-\kappa\}  \,\subset 
\,\{x_n \le  \bar a 2^{i(1+\alpha)}\} 
\end{equation}
in~$B'_{2^i}\times (-2^{k_a-l}, 2^{k_a-l})$.
We note also that $U$ solves $LU =\bar \eps^{-s} f(U)$ for 
\[\bar \eps := 2^l \eps < 2^l a^{p_0} \le  
\frac{2^l}{(2^{\alpha l})^{p_0}} \bar a^{p_0}\]
and hence the inequality $\bar \eps< \bar a^{p_0}$ is satisfied provided that we choose $p_0$ large enough.

Thus, the claim in {\em Step 1} yields that, 
for a suitable~$\bar z_n\in \R$,
\begin{equation}\label{star123}
\big\{\bar x_n \le \bar z_n- \bar a C(|\bar x'|^\sigma+\bar r)\big\} \subset\{U\le -1+\kappa\} \subset \{U\le 1-\kappa\}  \subset  \big\{\bar x_n \le z_n +\bar a C(|\bar x'|^\sigma +\bar r) \big\}
\end{equation}
in $B'_{1}\times  (-2^{k_a-l}, 2^{k_a-l})$, 
for $\bar r = C (\bar a)^{\frac{1}{1-\sigma}-1}$.

After rescaling, and setting 
$x =2^l \bar x$, $z_n =2^l \bar z_n$ and 
$r:= 2^l \bar r$, we obtain
\begin{eqnarray*}&&
\left\{\frac{x_n}{2^l} \le \frac{z_n}{2^l} -  C2^{l\alpha}a \left(\frac{ |x'|^\sigma}{2^{l\sigma}}+\frac{r}{2^l}\right)\right\} \subset\{U\le -1+\kappa\} \\
&& \qquad\subset \{U\le 1-\kappa\}  \subset  \left\{\frac{x_n}{2^l} \le \frac{z_n}{2^l} + C2^{l\alpha} a\left(\frac{ |x'|^\sigma}{2^{l\sigma}}+\frac{r}{2^l}\right)  \right\}
\end{eqnarray*}
in $B'_{2^l}(z')\times  (-2^{k_a}, 2^{k_a})$,  for
\begin{equation}\label{r geq}
r = 2^l \bar r =  C2^l\bar a^{\left(\frac{1}{1-\sigma}-1\right)} 
=  C(2^l)^{1+ \alpha\left (\frac{1}{1-\sigma}-1\right)} 
a^{\left(\frac{1}{1-\sigma}-1\right)} 
\le  C(2^l)^{1+ \alpha(1+\sigma)} a^{\sigma} 
.\end{equation}
Now, given $z'\in B'_{2^{k_a-1}}$, 
let us denote~$R_{z'}:= 2^l$, 
where~$l:=\min \{l' \ :\ 2^{l'} \ge |z'|\}$.
In view of~\eqref{r geq}, we define also
\[ r_{z'} := C\, (R_{z'})^{1+\alpha(1+\sigma)} \,a^{\sigma}\]
and the function~$\Psi_{z'} : \R^{n-1}\rightarrow 
[0, +\infty]$, given by
\[
\Psi_{z'}(x') :=  
\begin{cases}
CR_{z'}^{1+\alpha} \left(R_{z'}^{\alpha\sigma} 
\frac{|x'-z'|^\sigma}{R_{z'}^\sigma} 
+\frac{r_{z'}}{R_{z'}}\right ) \quad& 
\mbox{for }|x'|\le R_{z'}, \\ 
+\infty & \mbox{for }|x'|> R_{z'}.
\end{cases}
\]
Hence, from~\eqref{star123}, we have that
\[
\left\{x_n \le z_n-  a\Psi_{z'}(x')\right\} \subset\{U\le -1+\kappa\} \subset \{U\le 1-\kappa\}  \subset  \left\{\frac{x_n}{2^l} \le z_n+ -  a\Psi{z'}(x')\right) 
\]
in $B'_{2^{k_a-1}}\times(-2^{k_a}, 2^{k_a})$.

Furthermore, we notice that  
\[
\Psi_{z'}(z') \le  CR_{z'}^{\alpha} R_{z'}^{1+\alpha(1+\sigma)} a^{\sigma}\]
and
\[
\|\Psi_{z'}\|_{L^{\infty}}(B_{R_{z'}}(z')) + R_{z'}^{\sigma} [\Psi_{z'}]_{C^{\sigma}}(B_{R_{z'}}(z')) = CR_{z'}^{1+\alpha(1+\sigma)}
.\]
We then define 
\[
g^a (x') := \min_{z' \in \overline{B'_{2^{k_a-1}}}}  \big( z_n(z') +\Psi_{z'}(x')\big) \quad \mbox{and} \quad  g_a (x') := \max_{z' \in \overline{B'_{2^{k_a-1}}}}   \big( z_n(z') -\Psi_{z'}(x')\big)
.\]
It is now straightforward to 
verify that these two functions satisfy the 
requirements in the statement of Corollary~\ref{cor:globalCalpha},
as desired.
\end{proof}

We state a further consequence of Corollary \ref{cor:globalCalpha} and Lemma \ref{lemdecay} for its use in the next section.

\begin{corollary} \label{cor:globalCalpha2}
With the same assumptions as in 
Corollary \ref{cor:globalCalpha}, the 
following statement holds true.

Given $\theta\in(-1,1)$, we have that
\[
\{x_n \le a g(x') - Ca^{1+\sigma}(1+|x|)^{
1+\alpha(1+\sigma)} - C (1+|x|)^{\alpha\sigma} d   \} 
\subset   \{u_a\le \theta\}
\]
and 
\[
 \{u_a\le \theta\} \subset \{x_n \le ag(x') + Ca^{1+\sigma}(1+|x|)^{1+\alpha(1+\sigma)} + C (1+|x|)^{\alpha\sigma} d   \}
\]
in $C_{2^{k_a-1}}$, for all $d>0$ satisfying
\[
\left(\frac{\eps}{d}\right)^{\gamma_0} \le 1-|\theta|. 
\]
\end{corollary}

\begin{proof}
This is a direct consequence of 
Corollary~\ref{cor:globalCalpha} 
and the decay estimates of Lemma~\ref{lemdecay}.
\end{proof}

\section{Viscosity equation for the limit of vertical rescalings}\label{SECT5}

In this section we will prove that the limiting graph~$g$ 
given by Corollary \ref{cor:globalCalpha} satisfies the equation 
\begin{equation}\label{g eq p}
\bar L g = 0 \quad \mbox{in }\R^{n-1}
\end{equation}
where 
\begin{equation}\label{def:LBARRA}
\bar Lh(x'):=\int_{\R^{n-1}} \big( h(x') +\nabla h(x')\cdot(y'-x') -h(y') \big) \,
\mathcal K(x'-y',0) \,dy', \qquad  x'\in\R^{n-1},
\end{equation}
and
\begin{equation*}
\mathcal K(y):=\frac{\mu\left( y/|y|\right)}{|y|^{n+s}}.
\end{equation*}
We introduce $\mathcal K$ both to simplify the
notation and because the results of this part are also valid for more general kernels.
The definition of~$\bar Lh(x')$ is valid for functions~$h$ 
which are~$C^2$ in a neighborhood of~$x'$ and 
satisfying 
$$\int_{\R^{n-1}} |h(x')| \, (1+|x'|)^{-n-s}\,dx' <+\infty.$$
We also point out that~\eqref{g eq p} is a 
linear and translation invariant equation.

The strategy that we have in mind is the following:
once we have proved that $g$ is an entire solution 
of~\eqref{g eq p},
satisfying the growth control $g(x')\le C(1+|x'|)^{1+\alpha(1+\sigma)}$ 
(as given by Corollary \ref{cor:globalCalpha}),
we will deduce that $g$ is affine. 
This will be an immediate consequence of 
the interior regularity estimates for the 
equation~\eqref{g eq p}.

This set of ideas is indeed the content of the following result:

\begin{proposition}\label{prop:limiteqn}
The limit function $g:\R^{n-1}\to \R$ given by 
Corollary \ref{cor:globalCalpha} satisfies~\eqref{g eq p}
in the viscosity sense.
As a consequence, $g$ is affine.
\end{proposition}

In all this section we assume that $u_a$ is a solution of~$Lu_a= \eps^{-s} f(u)$ in $B_{2^{k_a}}$, where $\eps \in (0, a^{p_0})$ with $p_0$ large enough.
We denote by  $g$ the limiting graph as~$a\to0$
of the vertical rescalings of the level set, see
Corollary \ref{cor:globalCalpha}. We recall that this graph 
satisfies the growth control 
\begin{equation}\label{growthcontrol}
\big| g(x') \big| \le C(1+|x'|)^{1+\alpha(1+\sigma)}. 
\end{equation}
Moreover, as a consequence
of Corollary \ref{cor:globalCalpha2} we may assume that,
for any given~$\theta\in(-1,1)$, 
\begin{equation}\label{strongcontrolbelow}
\{x_n \le ag(x') - C(a^{1+\sigma}+d)(1+|x|)^{1+\alpha(1+\sigma)}  \} \subset   \{u_a\le \theta\}
\end{equation}
and
\begin{equation}\label{strongcontrolabove}
\{u_a\le \theta\} \subset \{x_n \le ag(x') + C(a^{1+\sigma}+d)(1+|x|)^{1+\alpha(1+\sigma)}  \}
\end{equation}
for all $d>0$ satisfying
\begin{equation}\label{whichd11}
\left(\frac{\eps}{d}\right)^{\gamma_0} \le 1-|\theta|. 
\end{equation}

In all the section and in the rest of the paper we will fix constant $\alpha, \sigma>0$ satisfying
\[ 
\alpha(1+\sigma)<s \qquad \mbox{and}\qquad \alpha<
\sigma
\]
For concreteness we may take, here and in the rest of the paper, 
\[\alpha =\frac{s}{4}\qquad \mbox{and}\qquad  
\sigma =1.\]

To prove that~$g$ is a viscosity solution of~\eqref{g eq p},
we will argue by contradiction. Indeed,
we will assume that~$g$ is touched by above by a 
convex paraboloid at~$x_0$ 
and that the operator computed at a test function~$h$ that is built (from $g$) 
by replacing~$g$ with the paraboloid in a tiny neighborhood of $x_0$ gives the wrong sign. 
Using this contradictory assumption, we will be able 
to build a supersolution of~$Lu =\eps^{-s}f(u)$ 
touching~$u_a$ from above at some interior point near~$x_0$. 
This will give the desired contradiction.
 
In all the section,  we assume that $Q$ is a fixed convex 
quadratic polynomial and, up to a rigid motion,
we can take the touching point~$x_0$ to be the origin. We
also let $d_a$ be the anisotropic signed distance function 
to $\{x_n \ge aQ(x')\}$, i.e. we use the setting 
in~\eqref{dk defin}, with $K:=K_a:=\{x_n \ge aQ(x')\}$.
More explicitly
\begin{equation}\label{da}
d_a(x)\, :=\, \inf \big\{\ell(x)\ :\ \ell \mbox{ affine, }h_L(\nabla \ell)=1, \mbox{ and }\ell \ge 0  \mbox{ in }K_a\,  \big\}.
\end{equation}
Then,
we will consider the following functions:
\begin{equation}\label{ua}
\tilde u_a(x) := \phi_0\left(\frac{d_a(x)}{\eps}\right) \chi_{{\QQQ}_{\delta}} + u_a(x',ax_n) \chi_{\R^n\setminus {\QQQ}_\delta}
\end{equation}
and
\begin{equation}\label{va}
v_a(x) := \phi_0\left(\frac{d_a(x)}{\eps}\right) 
\chi_{{\QQQ}_\delta} + {\rm sign}(x_n-ag(x')) 
\chi_{\R^n \setminus {\QQQ}_\delta},
\end{equation}
where $\delta>0$, 
\begin{equation}\label{8hrQdelta}
{\QQQ}_\delta := B_\delta'\times (-\delta,\delta),
\end{equation}
and~$\phi_0$ is the $1$D profile in~\eqref{existslayer}.
In a sense, $u_a$ and~$v_a$ have ``very flat level sets''
and we will compute the action of the operator~$L$ on such functions.

By explicit computations and error estimates, 
we will prove that not only~$L\tilde u_a - 
\eps^{-s}f(\tilde u_a) \rightarrow 0$ and~$Lv_a - 
\eps^{-s}f(v_a) \rightarrow 0$ as~$a\to 0$ 
in a neighborhood of~$0$, but we also
provide
the behavior of the next order in an expansion in the variable $a$. 
Namely, for~$a$ small enough, we will show that
\[ \frac 1 a \big( L\tilde u_a- \eps^{-s} f(\tilde u_a) \big)\approx  \frac 1 a \big( Lv_a- \eps^{-s} f(v_a) \big) \approx \bar -Lh(0)\]  
in neighborhood of $0$ in  $\R^n$ (we recall that~$h$
is the test function built from the touching paraboloid before~\eqref{da}).

To prove this, we will use our previous idea of 
``subtracting the  tangent $1$D  profile'' 
\begin{equation}\label{tgprof}
\tilde\phi(x) =  \phi_0(\tilde d/\eps), 
\end{equation}
where $\tilde d$ will be the signed anisotropic 
distance function to some appropriate tangent 
plane to the zero level set of $u_a$.
 
More precisely, in order to compute $L v_a- \eps^{-s} f(v_a)$
at a point~$z\in B_{\delta/4}$, we introduce the ``tangent profile'' at $z$ defined as \eqref{tgprof} with
\begin{equation}\label{tilded}
\tilde d(x) := \frac{\omega}{h_L(\omega)}\cdot (x-z) +t_0, \
\quad \mbox{where }t_0 = d_a(z)
\end{equation}
and $\omega \in S^{n-1}$ is the unit normal vector to  $\{d_a=t_0\}$ pointing towards $\{d_a>t_0\}$. 

Using the layer cake decomposition in Lemma~\ref{layer cake}, we will compute the  difference $Lv_a- \eps^{-s} f(v_a)$ as the integral
\begin{equation}\label{LAYER}
\begin{split}
 Lv_a(z) - \eps^{-s} f(v_a(z)) &= Lv_a(z)- L\tilde\phi(z)
 \\
&=
\int_{-1}^1 d\theta \, \int_{\R^n} \big(\chi_{S_\theta}(y)-\chi_{T_\theta}(y) \big)
\mathcal K(z-y) \,dy 
\end{split}
\end{equation}
where
\begin{equation}\label{defStheta}
S_\theta := \big\{ v_a\le \theta\le \tilde \phi\big\} \quad \mbox{and}\quad T_{\theta} :=\big\{\tilde\phi  \le \theta\le  v_a\big\}.
\end{equation}

However, in this section we will obtain more information by introducing the vertical rescaling (or change of variables)
\[ (y', y_n) =(\bar y',a\bar y_n)\]
which allows us to compute
\[
\frac 1 a \big( Lv_a(z) -f(v_a(z)) \big) = \int_{-1}^1 d\theta \, \int_{\R^n} \big(\chi_{\bar S_\theta}(\bar y)-\chi_{\bar T_\theta}(\bar y) \big) \mathcal K( z'-\bar y', z_n-a\bar y_n) \,d\bar y
\]
where
\begin{equation}\label{defbarStheta}
\bar S_\theta := \big\{ (\bar x',\bar x_n) \ :\  
(\bar x', a\bar x_n)  \in  S_\theta \big\}  \quad \mbox{and}\quad \bar T_\theta := \big\{ (\bar x',\bar x_n) \ :\  (\bar x', a\bar x_n)  \in  T_\theta \big\} 
.\end{equation}
We will see that for all the level sets outside a
set of ``small'' measure $2a^2$, namely for 
\[ \theta\in \left(-1+a^2, 1-a^2 \right),\]
we have
\begin{equation}\label{htheta} 
\begin{split}
\bar S_\theta \ &=\  \big\{\, \bar y= (\bar y',\bar y_n) \ :\ \ h_\theta(\bar y')\le  \bar y_n \le   h_\theta(\bar z') + \nabla h_\theta(\bar z')\cdot ( \bar y'-\bar z') \, \big\}  
\\ {\mbox{and }}\quad
\bar T_\theta \ &=\  \big\{\, \bar y= (\bar y',\bar y_n) \ :\ \ h_\theta(\bar z') + \nabla h_\theta(\bar z')\cdot ( \bar y'-\bar z') \le  \bar y_n \le  h_\theta(\bar y') \, \big\}, 
\end{split}
\end{equation}
where, given $\beta\in(0,1)$, we have, for some $\eta>0$, 
\begin{equation}\label{htheta2}
 \|h_\theta-h\|_{C^{1,\beta}(B_\delta')} \le  Ca^{\eta}\quad \mbox{and}\quad h_\theta=h \quad\mbox{in }\R^n\setminus B_\delta'.
\end{equation}
This will imply that when $|z'|$, $|z_n|$
and $a$ are all converging to $0$, we have
\begin{equation}\label{thechain}
\begin{split}&
\frac 1 a \big( L\tilde u_a(z) - \eps^{-s} f(\tilde u_a(z)) \big) \\
\stackrel{1}{\approx} \;&\frac 1 a \big( Lv_a(z) - \eps^{-s} f(v_a(z)) \big) 
\\
\stackrel{2}{=} \;&\int_{-1}^1 d\theta  \int_{\R^n} \big(\chi_{\bar S_\theta}(\bar y)-\chi_{\bar T_\theta}(\bar y) \big)\mathcal K( z'-\bar y', z_n-a\bar y_n) \,d\bar y 
\\
 \stackrel{3}{\approx} \;&\int_{-1+a^2}^{1-a^2} d\theta  \int_{\R^n} \big(\chi_{\bar S_\theta}(\bar y)-\chi_{\bar T_\theta}(\bar y) \big)\mathcal K( z'-\bar y', z_n-a\bar y_n) \,d\bar y 
\\
\stackrel{4}{\approx} \;&\int_{-1+a^2}^{1-a^2} d\theta \int_{\R^n} \big(\chi_{\bar S_\theta}(\bar y)-\chi_{\bar T_\theta}(\bar y) \big)\mathcal K( z'-\bar y', 0) \,d\bar y 
\\
\stackrel{5}{=}\;&\int_{-1+a^2}^{1-a^2}
-\bar L  h_\theta(z') \,d\theta
\\
\stackrel{6}{\approx} \;& -\bar L  h(0).
\end{split}
\end{equation}

In the next six lemmas, corresponding to the numbers appearing in \eqref{thechain}, we  prove the claimed equalities and  we control the errors in the previous chain of approximations.
\begin{lemma}[Approximation 1]\label{l1}
We have
\[ \lim_{a\to0}
\sup_{z\in B_{\delta/4}} \left| \frac 1 a \big( L\tilde u_a(z) 
-f(\tilde u_a(z)) \big) - \frac 1 a \big( Lv_a(z) -f(v_a(z)) \big) 
\right| \ = \ 0
.\]
\end{lemma}

\begin{proof}
We observe that $\tilde u_a =v_a$ in ${\QQQ}_\delta$. Then, using
the layer cake formula in~\eqref{RP:2} of Lemma~\ref{layer cake}, 
\begin{equation}\label{AAAA}
|L\tilde u_a(z) -Lv_a(z)| \le Ca^2 +\int_{-1+a^2}^{1-a^2} 
d\theta \int_{\R^n\setminus {\QQQ}_\delta } 
\chi_{\{ \tilde u_a\le \theta\le v_a\}\cup \{ v_a\le \theta
\le \tilde u_a\}}   (y) \;|y-z|^{-n-s}\, dy
.\end{equation}
We also remark
that, by the definition of~$v_a$, we have that, 
for all~$\theta\in (-1,1)$, 
\begin{equation}\label{BAb71}
\{ v_a\ge \theta\} =\{x_n \ge ag(x')\}  \quad 
\mbox{in }\R^n \setminus {\QQQ}_\delta.\end{equation}
Hence, if~$\theta\in(-1+a^2, 1-a^2)$, 
we use~\eqref{strongcontrolbelow}, 
\eqref{strongcontrolabove} and~\eqref{whichd11} 
and we find that
\begin{equation}\label{when3}
\begin{split}
\{ \tilde u_a\le& \theta\le v_a\}\cup\{ v_a\le \theta\le \tilde u_a\} \\
&\subset  \big\{ ag(x') -C(a^{1+\sigma} +d) (1+|x|)^{1+\alpha(1+\sigma)}\le x_n\le ag(x') + C(a^{1+\sigma} +d) (1+|x|)^{1+\alpha(1+\sigma)} \big\} 
\end{split}
\end{equation}
in $B_{2^{k_a-1}}\setminus {\QQQ}_\delta$, whenever 
\begin{equation}\label{when}
(\eps/d)^{
\gamma_0}\le a^2 .\end{equation} For $p_0$ chosen large enough
(recall that we assume $\eps<a^{p_0}$), 
we may take
\begin{equation}\label{when2}
d:=a^{1+\sigma}\end{equation} and satisfy~\eqref{when}. 
Hence, with the setting in~\eqref{when2}, we get from~\eqref{when3} that
\begin{equation*}
\begin{split}&
\{ \tilde u_a\le  \theta\le v_a \} \cup \{ v_a\le \theta\le \tilde u_a\} 
 \\ &\qquad\qquad\subset
  \big\{ |x_n-ag(x') | \le Ca^{1+\sigma} 
(1+|x|)^{ 1+\alpha(1+\sigma) } \big\} \quad \mbox{ in }B_{ 2^{k_a-1} }.
\end{split}\end{equation*}
It then follows that, for all $\theta\in(-1+a^2, 1-a^2)$,
\begin{equation}\label{BBBB}
\begin{split}&
\int_{\R^n\setminus {\QQQ}_\delta } 
\chi_{\{ \tilde u_a\le \theta\le v_a\}\cup 
\{ v_a\le \theta\le \tilde u_a\}}  (y)  \;|y-z|^{-n-s}\, dy 
\\ \le\;& \int_{\R^n\setminus B_{2^{k_a-1} }}  
|y-z|^{-n-s}\, dy + C_\delta \int_1^{2^{k_a-1}}  
a^{1+\sigma} \;\frac{r^{1+\alpha(1+\sigma)+n-2}}{r^{n+s}}dr  
\\
\le\;&  C_\delta (a^{s/\alpha}+ a^{1+\sigma}),
\end{split}
\end{equation}
where we have used that $\sigma$ is chosen small 
so that $\alpha(1+\sigma)<s$ (recall the setting of
Corollary~\ref{cor:globalCalpha}).
The desired result then follows immediately from \eqref{AAAA} and \eqref{BBBB}.
\end{proof}

\begin{lemma}[Equality 2]\label{l2}
Let $z\in B_{\delta/4}$. Then
\[
\frac 1 a \big( Lv_a(z) -f(v_a(z)) \big)  =  \int_{-1}^1 d\theta  \int_{\R^n} \big(\chi_{\bar S_\theta}(\bar y)-\chi_{\bar T_\theta}(\bar y) \big)\mathcal K( z'-\bar y', z_n-a\bar y_n) \,d\bar y 
\]
where $\bar S_\theta$ and $\bar T_\theta$ are defined in \eqref{defbarStheta}.
\end{lemma}

\begin{proof}
From the layer cake formula in~\eqref{RP:2}
of Lemma~\ref{layer cake}
and the idea 
of ``subtracting the
tangent $1$D  profile'' at~$z$
(exactly as in the proof of Lemma \ref{lemdeform}) we obtain that
\eqref{LAYER} and \eqref{defStheta} hold,  
where $\tilde \phi$ is defined by~\eqref{tgprof}
and~\eqref{tilded}.
Then, the result simply follows by
performing the change of variables $(y',y_n)=(\bar y', a\bar y_n)$.
\end{proof}

\begin{lemma}[Approximation 3] \label{l3}
Let $z\in B_{\delta/4}$. If $a$ is small enough, then for all $\theta\in(-1,1)$ with $ |\theta|\ge 1-a^2$ we have
\[
\left|\int_{\R^n} \big(\chi_{\bar S_\theta}(\bar y)-\chi_{\bar T_\theta}(\bar y) \big)\mathcal K( z'-\bar y', z_n-a\bar y_n) \,d\bar y \right| \le \frac C a,
\]
for some~$C>0$.
\end{lemma}

\begin{proof}
To prove this result, it is convenient
to look at the statement with the integrals written 
with respect to the original variables~$(y',y_n)=
(\bar y', a\bar y_n)$. In this setting, we have to show that 
\begin{equation}\label{ST:pq}
I_1 := 
\left|\int_{\R^n} \big(\chi_{ S_\theta}(y)-\chi_{T_\theta}(y) 
\big)\mathcal K( z-y) \,dy \right| \le C.  
\end{equation}
To prove this, we actually do not need the condition $|\theta|\ge 1-a^2$,
although the result will be used only for these values of~$\theta$.  

Note that in~${\QQQ}_\delta$ we have 
that~$v_a =\phi_0(d/\eps)$ 
and~$\tilde \phi = \phi_0(\tilde d /\eps)$. 
Recalling the definition of~$T_\theta$ in~\eqref{defStheta} 
and the facts that, by construction, the level sets of~$d$ 
are convex, and the level sets of~$\tilde d$ 
are tangent hyperplanes to the level sets of~$d$, 
we obtain that 
\begin{equation}\label{ju0-UT}
T_\theta\cap {\QQQ}_\delta= \varnothing\end{equation}
for all $\theta$.

Now, to prove~\eqref{ST:pq},
we distinguish the
two cases $S_\theta\cap {\QQQ}_{\delta/2}= \varnothing$ and  $S_\theta\cap {\QQQ}_{\delta/2}\neq \varnothing$.

In the first case in which
\begin{equation}\label{yu:1ca0}
S_\theta\cap {\QQQ}_{\delta/2}= \varnothing,\end{equation} 
we claim that
\begin{equation}\label{yu:1ca}
{\mbox{$|z-y|\ge \displaystyle\frac\delta4$ 
for all $y\in S_\theta\cup T_\theta$.}}\end{equation}
To check this, let~$y\in S_\theta\cup T_\theta$.
Then, by~\eqref{ju0-UT}
and~\eqref{yu:1ca0}, we have that~$y\not\in {\QQQ}_{\delta/2}$.
This, together with the fact that $z\in {\QQQ}_{\delta/4}$, 
proves~\eqref{yu:1ca}. 

Therefore, in 
light of~\eqref{yu:1ca},
we have that
\[
I_1  \le 
C_\delta \int_{\R^n}\frac{dy}{(\delta +|y|)^{n+s}} \le C .
\]
This proves~\eqref{ST:pq} in this case.

In the second case in which
\begin{equation*}
S_\theta\cap {\QQQ}_{\delta/2}\ne \varnothing,\end{equation*} 
we use the fact
that $\{v_a = \theta\}\cap {\QQQ}_\delta$ is the level set 
of the anisotropic distance function to the parabola~$
x_n=Q_a(x'):=aQ(x')$. Hence,  exactly as in 
Lemma \ref{lemcurvbounds}, we have 
that  $\{v_a = \theta\}\cap {\QQQ}_\delta$ 
is a convex $C^{1,1}$ graph with $C^{1,1}$ 
norm bounded by $Ca$ (and thus by~$C$).
Therefore, recalling also~\eqref{ju0-UT},
$$ \left|\int_{B_{\delta/4}(z)} \big(\chi_{ S_\theta}(y)-\chi_{T_\theta}(y) 
\big)\mathcal K( z-y) \,dy \right| \le 
\int_{B_{\delta/4}(z)\cap S_\theta} 
\mathcal K( z-y) \,dy \le C
.$$
Consequently, we conclude that
\[ I_1 \le C + \int_{\R^n\setminus B_{\delta/4}(z)} \frac{dy}{|z-y|^{n+s}} \le C,\]
up to renaming~$C>0$, and so~\eqref{ST:pq}
follows also in this second case, as desired.
\end{proof}

\begin{lemma}[Approximation 4] \label{l4}
For all $\theta\in(-1,1)$ with $ |\theta|\le 1-a^2$ we have
\[
\left|\int_{\R^n} \big(\chi_{\bar S_\theta}(\bar y)-\chi_{\bar T_\theta}(\bar y) \big)\mathcal K( z'-\bar y', z_n-a\bar y_n) \,d\bar y- \int_{\R^n} \big(\chi_{\bar S_\theta}(\bar y)-\chi_{\bar T_\theta}(\bar y) \big)\mathcal K( z'-\bar y', 0) \,d\bar y \right| \rightarrow 0  
\]
as $(|a|+ |z_n|)\rightarrow 0$ whenever $|z'|\le \delta/4$.
\end{lemma}
To prove Lemma \ref{l4}, we need the following pivotal result:

\begin{lemma}\label{thehtheta}
For all $\theta \in (-1+a^2,1-a^2)$ there exists a 
function $h_\theta:\R^{n-1}\rightarrow\R$  such 
that 
\begin{equation}\label{h ug g}
{\mbox{$h_\theta =h=g$ outside~$B'_\delta$,}}\end{equation}
$h_\theta\in C^{1,1}(B'_\delta)$ and~\eqref{htheta} holds true. Namely, 
\begin{equation}\label{g6al}
\begin{split}
\bar S_\theta \ &=\  \big\{\, \bar y= (\bar y',\bar y_n) \ :\ \ h_\theta(\bar y')\le  \bar y_n \le   h_\theta(\bar z') + \nabla h_\theta(\bar z')\cdot ( \bar y'-\bar z') \, \big\}  
\\ {\mbox{and }}\quad
\bar T_\theta \ &=\  \big\{\, \bar y= (\bar y',\bar y_n) \ :\ \ h_\theta(\bar z') + \nabla h_\theta(\bar z')\cdot ( \bar y'-\bar z') \le  \bar y_n \le  h_\theta(\bar y') \, \big\}.\end{split}
\end{equation}
Moreover, 
\begin{equation}\label{inte892}
 \|h_\theta-h\|_{L^\infty(B_\delta')} \le  Ca
\quad \mbox{and}\quad  
\|h_\theta-h\|_{C^{1,1}(B_\delta')} \le  C
\end{equation}
for some~$C>0$.
In particular, \eqref{htheta2} holds true
for $\eta = \frac{1-\beta}{2}$.\end{lemma}

\begin{proof} If~$\theta$ is as in the statement
of Lemma~\ref{thehtheta},
we take~$t_{\theta}:=\epsilon\phi_0^{-1}(
\theta)$.
Then, using~\eqref{DEC}, we have that
$$ a^2\le1-|\theta|=1-\left|\phi_0\left(\frac{t_\theta
}\epsilon\right)\right| 
\le \frac{C}{1+
\left(\frac{|t_\theta|}{\eps}\right)^{\gamma_0}}.$$ 
Hence (assuming~$\eps<a^{p_0}$ and~$p_0$ conveniently large),
we find that
\begin{equation}\label{tth} 
|t_\theta|\le \frac{C\eps}{a^{2/\gamma_0}}\le a^2.\end{equation}
Then, by the
definition of $v_a$, we have 
\[
\{v_a = \theta  \} = \{d_a =t_\theta\} \quad \mbox{in }{\QQQ}_\delta
.\]
Now, since  $\{d_a=0\}= \{x_n = aQ(x')\}$, by
exactly the same argument of Lemma \ref{lemcurvbounds}, we obtain that 
\[
\{d_a = t_\theta\} = \{x_n = G_\theta(x')\}
\]
for some $G_\theta$ satisfying 
\[ |D^2G_\theta|\le Ca \quad \mbox{in }B'_1.\]
Notice also that, by \eqref{tth}, the graph of~$G_\theta$
in~$B'_\delta$ lies in a $Ca^2$-neighborhood of
the graph of~$aQ$ (that is~$ah$, recall the construction
of the touching test function before~\eqref{da}).

We now recall that
the tangent profile at $z$, that we denoted by~$\tilde \phi$, is built in 
such a way that 
\[ \{\tilde \phi = \theta\} = \{\tilde d = t_\theta\}\]
is the tangent plane to $\{x_n =ag(x')\}$ at the point $z =(z',z_n)$.

These observations and~\eqref{BAb71} imply that 
\[
\begin{split}
S_\theta \ &=\  \big\{\,y= (y', y_n) \ :\ \ \tilde h_\theta(y')\le  y_n \le   \tilde h_\theta(\bar z') + \nabla \tilde h_\theta(z')\cdot ( y'-z') \, \big\}  
\\{\mbox{and }}\quad
T_\theta \ &=\  \big\{\,  y= (y',y_n) \ :\ \ \tilde h_\theta( z') + \nabla \tilde h_\theta(z')\cdot ( y'- z') \le   y_n \le  \tilde h_\theta(y') \, \big\}, 
\end{split}
\]
for a suitable function~$\tilde h_\theta$,
with
\begin{equation}\label{hjKA}
\sup_{y'\in B'_\delta}
|D^2\tilde h_\theta(y')|\le Ca\end{equation} and $\tilde h_\theta =ag$ outside $B'_\delta$.
In addition,
\begin{equation}\label{d-da}
{\mbox{the graph of~$\tilde h_\theta$
in~$B'_\delta$ lies in a $Ca^2$-neighborhood of
the graph of~$ah$.}}\end{equation}
Now, the desired result in~\eqref{g6al} follows from the change of 
variables $(y',y_n)=(\bar y', a\bar y_n)$, by taking 
$$ h_\theta := \tilde h_\theta /a.$$
To check~\eqref{inte892}, we observe that the estimate in~$C^{1,1}(B_\delta')$
follows from the bound in~\eqref{hjKA}
and the fact that $h$ is a given paraboloid in~$B'_\delta$.
Also, the uniform bound in~\eqref{inte892}
is a consequence of~\eqref{d-da}.

These observations establish~\eqref{inte892}.
We also remark that~\eqref{htheta2}
follows from~\eqref{inte892} by interpolation.
\end{proof}

\begin{proof}[Proof of Lemma \ref{l4}]
We claim that the map
\begin{equation}\label{INFO}
\R^n\ni\bar y=(\bar y',\bar y_n)
\mapsto {\mathcal{J}}(\bar y):=\frac{
\chi_{\bar S_\theta}(\bar y)+
\chi_{\bar T_\theta}(\bar y)}{|z'-\bar y'|^{n+s}} \;
{\mbox{ belongs to }} \;L^1(\R^n).
\end{equation}
For this, we use Lemma \ref{thehtheta} to see that
\begin{equation}\label{INFO:p1}
\begin{split}
& \int_{ B'_{\delta/4}(z)\times  \left(-\infty,\infty\right) }
{\mathcal{J}}(\bar y)\,d\bar y\\&\qquad\le
C\int_\R\,d\bar y_n\,\int_{S^{n-2}} d\omega \int_0^{\delta} \,dr\;
\frac{r^{n-2}\,
\big( \chi_{\bar S_\theta}(z'+r\omega,\bar y_n)+
\chi_{\bar T_\theta}(z'+r\omega,\bar y_n)\big)
}{r^{n+s}} \\
&\qquad\le
C \int_0^{\delta}
\frac{r^{n-2}\, r^2
}{r^{n+s}}\,dr \le C\delta^{1-s}\le C,
\end{split}\end{equation}
up to renaming~$C>0$.
On the other hand,
recalling~\eqref{growthcontrol} and~\eqref{h ug g},
we deduce from~\eqref{g6al} that~$\bar S_\theta$
and~$\bar T_\theta$ are controlled at infinity
by a function with growth~$C|\bar y'|^{1+\alpha}$.
Consequently,
\begin{equation*}\begin{split}&
\int_{ \R^n\setminus\left(
B'_{\delta/4}(z)\times  \left(-\infty,\infty\right) \right)}
{\mathcal{J}}(\bar y)\,d\bar y\\ \le\;&
C\int_\R\,d\bar y_n\,\int_{S^{n-2}} d\omega 
\int_{\delta/4}^{+\infty} \,dr\;
\frac{r^{n-2}\,
\big( \chi_{\bar S_\theta}(z'+r\omega,\bar y_n)+
\chi_{\bar T_\theta}(z'+r\omega,\bar y_n)\big)
}{r^{n+s}} \\ \le\;& C 
\int_{\delta/4}^{+\infty} 
\frac{r^{n-2}\,r^{1+\alpha}
}{r^{n+s}}\,dr\le C\delta^{\alpha-s}\le C .
\end{split}\end{equation*}
This and~\eqref{INFO:p1}
imply~\eqref{INFO}, as desired.

Then,
using~\eqref{INFO} and the fact
that $K( z'-\bar y', z_n-a\bar y_n) \rightarrow 
K( z'-\bar y', 0) $ almost everywhere 
in $\R^n$ as $(|a|+ |z_n|)\rightarrow 0$,
we see that the result in Lemma \ref{l4}
follows by the dominated convergence theorem.
\end{proof}

\begin{lemma}[Equality 5] \label{l5}
For all $\theta\in(-1+a^2,1-a^2)$ we have
\[
\int_{\R^n} \big(\chi_{\bar S_\theta}(\bar y)-\chi_{\bar T_\theta}(\bar y) \big)
\mathcal K( z'-\bar y', 0) \,d\bar y  = -\bar L  h_\theta(z') 
\]
where $h_\theta\in C^{1,1}(B_\delta')$ is given in Lemma \ref{thehtheta}.
\end{lemma}
\begin{proof} {F}rom~\eqref{g6al}, we see that
$$ \int_{\R^n} \big(\chi_{\bar S_\theta}(\bar y)-\chi_{\bar T_\theta}(\bar y) \big)
\mathcal K( z'-\bar y', 0) \,d\bar y 
=
\int_{\R^{n-1}} 
\big(h_\theta(\bar y') -\nabla h_\theta(z')(\bar y'-z')-h_\theta(z') \big)  \mathcal K( z'-\bar y', 0) \,d\bar y'
.$$
This and~\eqref{def:LBARRA} give the desired result.
\end{proof}

\begin{lemma}[Approximation 6] \label{l6}
For all $\theta\in(-1+a^2,1-a^2)$ we have
\[
\bigl| \bar L  h_\theta(z')  - \bar L  h(0) \bigr| \rightarrow 0
\]
as $(|a|+ |z'|)\rightarrow 0$.
\end{lemma}
\begin{proof}
It is standard using that \eqref{htheta2} holds, as given by
Lemma~\ref{thehtheta}.
\end{proof}

Let us give now an elementary result 
that will be useful in the proof of 
Proposition~\ref{prop:limiteqn}.

\begin{lemma}\label{lacosadisempre}
Given $r>0$, there exists~$\delta>0$, 
depending only on~$n$, $s$, ellipticity constants 
and~$r$, such that the following holds. 

Assume that $Lw\ge a>0$  in  $B_r\cap \{w\le 0\}$ and $w\ge -\delta a$ in all of $\R^n$. 

Then, $w> 0$ in $\overline{B_{r/2}}$.
\end{lemma}

\begin{proof}
The proof is standard, we give the details for the convenience 
of the reader. 
We consider the function~$\tilde w := 
w + \delta a(1-\eta(x/r))$, 
where~$\eta\in C^2_0(B_1)$ is a smooth radial cutoff with~$
\eta=1$ in~$B_{1/2}$. If, by contradiction, 
$w \le 0$ at some point in~$B_{r/2}$, then~$\tilde w$ 
attains an absolute minimum at some point~$x_0$ 
in~$B_r$.
Thus,  
$$ 0\ge L \tilde w(x_0)\ge Lw - C\delta  a r^{-s}  
\ge a - C\delta  a r^{-s} \ge a/2 >0,$$  
which gives a contradiction if $\delta$ is taken 
small enough.
\end{proof}

With this preliminary work,
we can finally complete the 
proof of Proposition \ref{prop:limiteqn}, by arguing as follows.
 
\begin{proof}[Proof of Proposition \ref{prop:limiteqn}]
Up to a translation,
we can test
the definition of viscosity solution for a
smooth function touching $g$ by above
at the point $x_0= 0$ (the argument to take care
of the touching by below is similar).

Let~$U'\subset \R^{n-1}$
be a neighborhood of the origin and~$\psi \in C^2(U')$.
Assume that~$\psi$
touches by above $g$ in $U'$ at $0$. 
Assume by contradiction that~$\tilde \psi := \psi \chi_{U'} +g\chi_{\R^n\setminus U'}$ satisfies~$\bar L \tilde\psi(0)>0$.

Then
(see, for instance, Section~3 in~\cite{CaffSil}),
we know that there exist~$\delta>0$ small and
two concave polynomials, denoted by~$Q$ and~$\tilde Q$,
satisfying 
\begin{equation}\label{QandQtilde}
Q(0)= \tilde Q(0)= g(0) \quad {\mbox{ and }}
\quad Q > \tilde Q \ge g \quad \mbox{in }B'_{\delta}\setminus\{0\}
\end{equation}
and such that, if we define~$
Q^{t} :=  Q+t$ and~$
h:= Q^{t}\chi_{B'_{\delta}} +g\chi_{\R^n\setminus B'_{\delta}}$,
it holds that
\[ 
\bar L h(0) >0, 
\]
for all
$t\in(-\delta^3, \delta^3)$.

Let us now consider the function~$\tilde u_{a,t}$
defined as in~\eqref{ua},
with~$d_a$ replaced by the distance from~$a Q^{t}$,
namely, 
\begin{equation}\label{e:qCX672}
\tilde u_{a,t}(x) := \phi_0\left(\frac{d_a(x)}{\eps}\right) 
\chi_{{\QQQ}_\delta} + u_a(x) \chi_{\R^n\setminus {\QQQ}_\delta}
\end{equation}
where now~$d_a$ is the anisotropic signed distance
function to~$\{x_n \ge aQ^{t}(x')\}$
and~${\QQQ}_\delta$ was defined in~\eqref{8hrQdelta}.

By~\eqref{thechain} 
(which has been proved in
Lemmas~\ref{l1}, \ref{l2}, \ref{l3}, \ref{l4}, \ref{l5}
and~\ref{l6}), we obtain that 
\begin{equation}\label{eqdif22}
 L\tilde u_{a,t}- \eps^{-s} f(\tilde u_{a,t})\le- ca 
\quad {\mbox{ in }}\; B_r, 
\end{equation}  
for some $r>0$ and~$c>0$, whenever~$a$ is 
small enough and~$t\in[-\delta^3,\delta^3]$.
By possibly reducing~$r>0$, we will suppose that
\begin{equation}\label{RED}
r\in(0,\delta).\end{equation} 
We
note that, in this setting, $r$ and~$c$ 
depend on~$\bar L h(0)$.

Next we show that, for~$t=\delta^3$ and~$a$ 
small enough, we have
\begin{equation}\label{difpositive}
u_a -\tilde u_{a,t}>0  \quad{\mbox{ in }}\;  \overline{B_{r/2}}.
\end{equation}
To prove this, 
we recall that, by Corollary~\ref{cor:globalCalpha2}
(used here with~$d:=a^2$), 
we have 
\begin{equation}\label{levelsua}
 \{x_n\le ag(x')-Ca^{1+\sigma} \}  
\subset \{ u_a \le \theta\} \subset 
\{x_n\le ag(x')+Ca^{1+\sigma} \}
\end{equation}
in $B_1'\times(-1,1)$, provided 
that~$(\eps/a^2)^{\gamma_0}  \le 1-|\theta|$. 
On the other hand,
by definition~$\tilde u_{a,t}= \phi_0(d_a/\eps)$ 
in $\QQQ_\delta$. Therefore,
\begin{equation}\label{levelstildeu}
 \{x_n\le aQ^{t} (x')
-Ca^{2} \}  \subset \{ \tilde u_{a,t} \le \theta \} 
\subset \{x_n\le aQ^{t}(x')+Ca^{2} \}
\end{equation}
 in $\QQQ_\delta$, also provided that $(\eps/a^2)^{
\gamma_0}  \le 1-|\theta|$ (with~$\gamma_0$ given by~\eqref{DEC}).

We remark that,
roughly speaking, \eqref{levelsua} says that the ``transition level sets'' of $u_a$ lie essentially on the surface $\{x_n = ag(x')\}$,
while \eqref{levelstildeu} says  that the ``transition level sets'' of $\tilde u_{a,t}$ lie essentially on the surface $\{x_n = a Q^{t} (x')\}$,
up to small errors 
of size $a^{1+\sigma}$.

Then, since $Q\ge g$ in~$B_\delta'$
by~\eqref{QandQtilde},  for $t=\delta^3$
(or any other fixed positive number), 
if we assume that $\eps\le a^{p_0}$ with $p_0$ 
large enough, we can use~\eqref{levelsua}
with~$\theta:=1-a^2$ and~\eqref{levelstildeu}
with~$\theta:=-1+a^2$,
take $a$ small enough and conclude that 
\begin{equation}\label{inclusionX:Y}
\{u_{a} \le 1-a^2\} \subset 
\{\tilde u_{a,t} \le -1+a^2\} \quad \mbox{in } \QQQ_\delta
.\end{equation}
In particular, by~\eqref{RED}, we obtain that
\begin{equation}\label{inclusionX:0}
\{u_{a} \le 1-\kappa\} \subset 
\{\tilde u_{a,t} \le -1+\kappa\} \quad \mbox{in } B_r.
\end{equation}
Now we observe that
\begin{equation}\label{inclusionX}
u_a -\tilde u_{a,t} > - a^2 \quad \mbox{in all of }\R^n.
\end{equation}
Indeed, if~$x\in \QQQ_\delta$, 
we distinguish two cases: either~$u_a(x)>1-a^2$
or~$u_a(x)\le1-a^2$. In the first case, we have that
$$ u_a (x)-\tilde u_{a,t} (x)>(1-a^2)-1=-a^2.$$
In the second case,
we can use~\eqref{inclusionX:Y} and obtain that~$
\tilde u_{a,t} (x)\le -1+a^2$ and, consequently
$$ u_a (x)-\tilde u_{a,t} (x)>-1-(-1+a^2)=-a^2.$$
These observations prove~\eqref{inclusionX}
when~$x\in \QQQ_\delta$.
If instead~$x\in\R^n\setminus \QQQ_\delta$,
we recall~\eqref{e:qCX672}
and we have that~$\tilde u_{a,t}(x) = u_a(x)$,
and this implies~\eqref{inclusionX} also in this case.

Now, we observe that
\begin{equation}\label{CH:do}
f(u_a) \ge f(\tilde u_{a,t})\;{\mbox{ in }}
B_r\cap \{ u_a -\tilde u_{a,t}\le 0\}.
\end{equation}
To check this
we take~$x\in B_r\cap \{ u_a -\tilde u_{a,t}\le 0\}$
and we distinguish two cases, either~$u_a(x)\le1-\kappa$
or~$u_a(x)>1-\kappa$.
In the first case, we exploit~\eqref{inclusionX:0}
and we obtain that~$\tilde u_{a,t}(x) \le -1+\kappa$
and thus
$$ u_a(x)\le\tilde u_{a,t}(x)\le-1+\kappa.$$
This and the monotonicity of~$f$ in~\eqref{assumpf}
imply~\eqref{CH:do} in this case.

If instead~$u_a(x)>1-\kappa$, we have
$$ 1-\kappa<u_a(x)\le\tilde u_{a,t}(x),$$
and once again 
the monotonicity of~$f$ in~\eqref{assumpf}
implies~\eqref{CH:do}, as desired.

Now, from \eqref{eqdif22} and \eqref{CH:do}
it follows that 
\[
L(u_a -\tilde u_{a,t}) \ge \eps^{-s} \big(
f(u_a) - f(\tilde u_{a,t}) \big)+ ca \ge ca  \quad \mbox{in } 
B_r\cap \{ u_a -\tilde u_{a,t}\le 0\}.
\]
Then, Lemma \ref{lacosadisempre}
applied to $w:=u_a -\tilde u_{a,t}$ gives
that~\eqref{difpositive} holds for $t=\delta^3$.

Also, using~\eqref{levelsua} with~$\theta:=0$, we have that
$$ (0,\dots,0,\,ag(0)-Ca^{1+\sigma})\in \{ u_a\le0\}\quad
{\mbox{ and }}\quad
(0,\dots,0,\,ag(0)+Ca^{1+\sigma})\in\{ u_a\ge0\}
.$$
Therefore there exists~$\tau\in[g(0)-Ca^\sigma,\,
g(0)+Ca^\sigma]$
such that the point~$p_a=(p_{a}',p_{a,n})
:=(0,\dots,0, a\tau)$
satisfies
\begin{equation}\label{noholds:PRE}
u_a(p_a)=0.
\end{equation}
We claim that, for every fixed~$t<0$, taking $a$ small enough
(possibly in dependence of~$t$), 
we have
\begin{equation}\label{noholds}
u_a -\tilde u_{a,t}\le0 \quad \mbox{at the point }p_a.
\end{equation}
To this end, we recall~\eqref{QandQtilde} and
we observe that
\begin{eqnarray*}&& p_{a,n}-aQ^t(p_a')-Ca^2=
a\tau-aQ^t(0)-Ca^2
\ge a\big( g(0)-Ca^\sigma \big)-aQ(0)-at-Ca^2
\\ &&\qquad\qquad= -Ca^{1+\sigma}-at-Ca^2>0,\end{eqnarray*}
since~$t<0$,
as long as~$a$ is small enough (possibly depending on~$t$).
{F}rom this
and~\eqref{levelstildeu} (applied here with~$\theta:=0$),
we conclude that
$$p_a\in
\{x_n> aQ^{t}(x')-Ca^{2} \}\subset
\{ \tilde u_{a,t} \ge 0 \} .$$
This and~\eqref{noholds:PRE}
give that
$$ u_a (p_a)-\tilde u_{a,t} (p_a)\le u_a (p_a)=0,$$
which proves~\eqref{noholds}.

Now we let $t_*=t_*(a)$ be the infimum of 
the $t\in \R$ such that \eqref{difpositive} holds.  Notice that,
by~\eqref{difpositive} and~\eqref{noholds},
we know that 
\begin{equation}\label{tstara}
\liminf_{a\to0} t_*(a)=0.\end{equation}
Next, by \eqref{QandQtilde} we have 
\begin{equation}\label{FOA1}
Q-g \ge c_0 >0 \quad \mbox{for any $x'$ outside } B'_{r/8} 
,\end{equation}
where $c_0$ depends only on $Q$ and $\tilde Q$.

Also, in view of~\eqref{tstara},
if  $a$  is small enough, 
we may assume that $t_*> -c_0/2$.
Thus, by~\eqref{FOA1}, we have that
\[
Q^{t_*}-g =Q+t_*-g\ge c_0/2 >0 \quad \mbox{for any
$x'$ outside } B'_{r/8} .
\]
Hence, using again \eqref{levelsua} and  \eqref{levelstildeu}, we obtain that 
\[
\{u_{a} \le 1-\kappa\} \subset \{\tilde u_{a,t} \le -1+\kappa\} \quad \mbox{in } \QQQ_\delta\setminus B_{r/2}. 
\]
Hence, as before,
using that $\tilde u_{a,t} = u_a$ outside of $\QQQ_\delta$,
we conclude that
\[ 
u_a -\tilde u_{a,t_*} > - a^2 \quad \mbox{in }\R^n \setminus B_{r/2} 
.\]
Using again \eqref{eqdif22}
and assumption \eqref{assumpf}, it follows that, for $a$ small enough, 
\begin{equation}\label{thelast}
L(u_a -\tilde u_{a,t_*}) \ge \eps^{-s} \big( 
f(u_a) - f(\tilde u_{a,t_*}) \big)+ ca \ge ca  \quad \mbox{in } (B_r\setminus B_{r/2})\cap \{ u_a -\tilde u_{a,t_*}\le 0\}
.\end{equation}
On the other hand, by the
definition of $t_*$, we have that~$u_a -\tilde u_{a,t_*}
\ge 0$ in $B_{r/2}$ and hence formula~\eqref{thelast} 
holds true by replacing $(B_r\setminus B_{r/2})$ with $B_r$
(since the contribution in~$B_{r/2}$ is void).

Then, Lemma \ref{lacosadisempre},
applied to $w:=u_a -\tilde u_{a,t_*}$, yields that 
$u_a -\tilde u_{a,t_*}>0$ in $\overline{B_{r/2}}$,
which is a contradiction with the definition of $t_*$.
\end{proof}

\section{Completion of the proof of Theorem~\ref{thm:improvement}}\label{SECT6}

Using the techniques developed till now, we are in the position
to prove Theorem~\ref{thm:improvement}.

We need an auxiliary result, a geometric observation. It
says that if in a sequence of dyadic balls a set is trapped  in a sequence of slabs 
with possibly varying orientations,
then it is also trapped in a sequence of parallel slabs.

\begin{lemma}\label{geomseries}
Let $\alpha\in(0,1)$.  Assume that, for some~$a\in(0,1)$ and $X\subset \R^n$, we have  
\begin{equation}\label{induc11}
\big\{ x\cdot \omega_j \le  - \,a\,2^{j(1+\alpha)} \big\} \subset X \subset \big\{ x\cdot \omega_j \le \,a\,2^{j(1+\alpha)} \big\}  \quad \mbox{in }B_{2^j}
\end{equation}
for all
\[j= \left\{0,1,2,\dots, j_a := \left\lfloor  \frac{\log a}{\log(2^{-\alpha})}\right\rfloor \right\}\] 
 where $\omega_j\in S^{n-1}$.

Then, for some $m_0\in\N$, with~$m_0\le j_a$, and $C>0$, depending only on $\alpha$, we have\footnote{
We stress that~$\omega_0$ in~\eqref{0io7xba8} is 
simply~$\omega_j$ with~$j:=0$.} that 
\begin{equation}\label{0io7xba8}
\big\{ x\cdot \omega_0 \le  -C\theta\, a\, 2^{j(1+\alpha)} \big\} \subset X \subset \big\{ x\cdot \omega_0 \le  C\theta\, a\, 2^{j(1+\alpha)} \big\} \quad \mbox{in }B'_{\theta 2^j}\times(-2^{k_a}, 2^{k_a}) 
\end{equation}
for every~$j\in\N$, with~$0\le j\le j_a-m_0$.
\end{lemma}

\begin{proof}
We have, for all~$j\in \{0,1,\dots,j_a\}$,
\[
\big\{ x\cdot \omega_{j+1} \le  -a 2^{(j+1)(1+\alpha)} \big\} \subset X \subset \big\{ x\cdot \omega_j \le  a 2^{(j+1)(1+\alpha)} \big\}  \quad \mbox{in }B_{2^j}.
\]
Thus, rescaling by a factor $2^{-j}$, we obtain that
\begin{equation}\label{9udiyhj73chj:1}
\big\{ x\cdot \omega_{j+1} \le  -a 2^{j\alpha + 1+\alpha} \big\} \subset  \big\{ x\cdot \omega_j \le  a2^{j\alpha} \big\}  \quad \mbox{in }B_{1}.
\end{equation}
Also, 
for all~$j\in \{0,1,\dots,j_a-1\}$,
we have that \begin{equation}\label{9udiyhj73chj:2}
a2^{(j+1)\alpha} \le a2^{j_a\alpha}\le  1.\end{equation}
Hence, 
\begin{equation}\label{9udiyhj73chj:2bis} \delta_j:=
a2^{-j\alpha}\le 2^{-j-1-\alpha} <1.\end{equation}

Notice that, with this notation, \eqref{9udiyhj73chj:1} implies
that
\begin{equation}\label{9udiyhj73chj:1BIS}
\big\{ x\cdot \omega_{j+1} \le  -4\delta_j \big\} 
\subset  \big\{ x\cdot \omega_j \le  \delta_j\big\}  \quad \mbox{in }B_{1}.
\end{equation}

Observe now that 
\begin{equation}\label{9udiyhj73chj:3}
|\omega_{j+1} -\omega_j| \le 32 \delta_j.\end{equation}

Now, from~\eqref{9udiyhj73chj:3},
summing a geometric series, 
we deduce that
\[ |\omega_{j} -\omega_0| \le \sum_{i=0}^{j-1}
|\omega_{i+1} -\omega_i| \le C \sum_{i=0}^{j-1}\delta_i
=
Ca \sum_{i=0}^{j-1} 2^{i\alpha} = \frac{Ca \,2^{j\alpha}}{2^\alpha-1}\le
Ca \,2^{j\alpha},\]
up to renaming~$C>0$.

{F}rom this, and up to renaming~$C$ once again, we obtain that
\begin{eqnarray*} &&
\big\{ x\cdot \omega_0 \le  -Ca\, 2^{j(1+\alpha)} \big\} 
\subset 
\big\{ x\cdot \omega_j \le  - a\, 2^{j(1+\alpha)} \big\} 
\\
{\mbox{and }}&&
\big\{ x\cdot \omega_j \le  a\, 2^{j(1+\alpha)} \big\} 
\subset 
\big\{ x\cdot \omega_0 \le  C a\, 2^{j(1+\alpha)} \big\} 
\quad \mbox{in }B_{ 2^j} ,\end{eqnarray*}
which implies the desired result (if~$m_0$ is sufficiently large).\end{proof}
 
Now we are in the position of completing
the proof of Theorem~\ref{thm:improvement}.

\begin{proof}[Proof of Theorem~\ref{thm:improvement}]
Let us denote $u=u_a$ to emphasize  the dependence of the statement on $a$. 
By Lemma \ref{geomseries} we have that, 
in a suitable coordinate system such that the axis $x_n$ is parallel to $\omega_0$,
\begin{equation*}
\{x_n\le -a 2^{j(1+\alpha_0)}\}\, \subset\, 
\{u_a\le -1+\kappa\} \,\subset \,\{u_a\le 1-\kappa\}  
\,\subset \,\{x_n \le a 2^{j(1+\alpha_0)}\} \quad 
\mbox{in }B'_{2^j}\times (-2^{k_a}, 2^{k_a})
\end{equation*}
for $0\le j\le k_a$, where $k_a=j_a-m_0$ and 
where $m_0=m_0(\alpha_0)$ is the constant of Lemma~\ref{geomseries}.

Then, by Corollaries \ref{cor:globalCalpha} 
and \ref{cor:globalCalpha2}, combined with
Proposition \ref{prop:limiteqn}, we find that 
\[
\{x_n \le a g(x') -Ca^{1+\sigma}\}\, \subset\, \{u_a\le -1+\kappa\} \,\subset \,\{u_a\le 1-\kappa\}  \,\subset \, \{x_n \le a g(x') +Ca^{1+\sigma}\}
\]
in~$B'_{1}\times(-2^{k_a}, 2^{k_a})$,
where $g$ is affine. The assumption $0\in \{-1+\kappa\le u_a\le 1-\kappa\}$ guarantees that 
\[ g(0)=0.\]

Then, if~$a$ is small enough, this implies that 
\[ 
\left\{\omega\cdot x \le - 
\frac{a}{2^{1+\alpha_0}} \right\} \,\subset \,\{u_a\le -1+\kappa\} \,\subset\, \{u_a\le 1-\kappa\}  \,\subset\, \left\{\omega\cdot x\le  \frac{a}{2^{1+\alpha_0}} \right\} \quad \mbox{in }B_{1 /2},
\]
for some $\omega\in S^{n-1}$, and thus
Theorem~\ref{thm:improvement} follows.
\end{proof}

\section{Proof of Theorem~\ref{C:1}}\label{SECT7}

Now we give the proof of Theorem~\ref{C:1},
by applying a suitable iteration of Theorem~\ref{thm:improvement}
at any scale and the sliding method. 
For this, we point out two useful rescaled iterations
of Theorem~\ref{thm:improvement}.
The first, in Corollary~\ref{0NL:VV}, is a ``preservation of flatness'' iteration up to scale 1,
while the second, in Corollary~\ref{NL:VV},
is a ``improvement of flatness'' iteration up to a mesoscale.

We first give the 
\begin{corollary}[``preservation of flatness'']\label{0NL:VV} 
Assume that $L$ satisfies $\eqref{assumpL}$ and that
$f$ satisfies \eqref{assumpf} and \eqref{existslayer}.
Then there exist universal constants~$\alpha_0\in (0,s/2)$, $
p_0\in(2,\infty)$ and~$a_0\in(0,1/4)$ such that the following statement holds.
\smallskip

Let $u: \R^n \rightarrow (-1,1)$ be a solution 
of $Lu=f(u)$ in $\R^n$, such 
that $0\in \{-1+\kappa \le u \le  1-\kappa\}$.
Let~$k\ge j\in\N$ and suppose that
\begin{equation}\label{97f38jssd4d}
j\ge \frac{p_0\,|\log a_0|}{\log2}.
\end{equation}
Assume that
\begin{equation}\label{KOR:S1}
\{\omega_i \cdot x\le -a_0 2^{i}\}\, \subset\,
\{u\le -1+\kappa\} \,\subset \,\{u\le 1-\kappa\}
\,\subset \,\{\omega_i\cdot x\le a_0 2^{i}\} \quad \mbox{in }
B_{2^{i}}, 
\end{equation}
for every~$i\ge k$, where  $\omega_i\in S^{n-1}$.

Then, for every~$i\in\N$, with~$j\le i\le k$,
it holds that
\begin{equation}\label{KOR:S2}
\{\omega_i \cdot x\le -a_0 2^{i}\}\, \subset\,
\{u\le -1+\kappa\} \,\subset \,\{u\le 1-\kappa\}
\,\subset \,\{\omega_i\cdot x\le a_0 2^{i}\} \quad \mbox{in }
B_{2^{i}}, 
\end{equation}
for some $\omega_i\in S^{n-1}$.
\end{corollary}

\begin{proof} We prove~\eqref{KOR:S2} for all indices~$i$
of the form~$i=k-\ell$, with~$\ell\in\{0,\dots, k-j\}$.
The argument is by induction over~$\ell$. Indeed, when~$\ell=0$,
then~\eqref{KOR:S2} is a consequence of~\eqref{KOR:S1}.
Hence, recursively, we assume that
the interface of~$u$ in~$B_{2^{k-q}}$
is contained in a slab of size~$a_0 2^{k-q}$,
with~$q\in\{0,\dots,\ell-1\}$,
and we prove that the same holds for~$q=\ell$.
To this aim, we set~$\tilde u(x):=u(2^{k-\ell+1}x)$
and~$\eps:=\frac{1}{2^{k-\ell+1}}$. Notice that~$L\tilde u=\eps^{-s}f(\tilde u)$
and
\begin{equation}\label{CH:81}
\frac{\eps}{a_0^{p_0}}=
\frac{1}{a_0^{p_0}\,2^{k-\ell+1}}\le
\frac{1}{a_0^{p_0}\,2^{j+1}}\le1,
\end{equation}
thanks to~\eqref{97f38jssd4d}.
In addition,
we claim that
\begin{equation}\label{CH:82}
{\mbox{for any~$i\in\N$,
the interface of~$\tilde u$ in~$B_{2^i}$
is trapped in a slab of size~$a_0\, 2^{i(1+\alpha_0)}$.}}
\end{equation}
For this, we distinguish the cases~$i\ge\ell$
and~$i\in\{0,\dots,\ell-1\}$.
First, suppose that~$i\ge\ell$. Then,
if~$x$ lies in the interface of~$\tilde u$
in~$B_{2^i}$, then~$y:=2^{k-\ell+1}x$ lies
in the interface of~$u$
in~$B_{2^{k-\ell+1+i}}$. Accordingly, by~\eqref{KOR:S1},
we know that~$y$ is trapped in a slab of size~$a_0\,2^{k-\ell+1+i}$.
As a consequence, $x$
is trapped in 
a slab of size~$a_0\,2^{i}\le a_0\, 2^{i(1+\alpha_0)}$.

This is~\eqref{CH:82} in this case, so we can now focus on the
case in which~$i\in\{0,\dots,\ell-1\}$.
For this, we take~$x$ in the interface of~$\tilde u$
in~$B_{2^i}$, and we observe that~$y:=2^{k-\ell+1}x$ lies
in the interface of~$u$
in~$B_{2^{k-\ell+1+i}}=B_{2^{k-(\ell-1-i)}}$.
Then, from the inductive assumption, we know that~$y$
is trapped in a slab of size~$a_0\,2^{k-(\ell-1-i)}=a_0\,2^{k-\ell+1+i}$.
Scaling back, it follows that~$x$
is trapped in a slab of size~$a_0\,2^{i}$,
which implies~\eqref{CH:82} also in this case.

So, in light of~\eqref{CH:81}
and~\eqref{CH:82},
we can apply Theorem~\ref{thm:improvement} to~$\tilde u$
and find that the interface of~$\tilde u$ in~$B_{1/2}$
is trapped in a slab of size~$\frac{a_0}{2^{1+\alpha_0}}$.

That is, scaling back,
the interface of~$u$ in~$B_{2^{k-\ell}}$
is trapped in a slab of size~$\frac{a_0\,2^{k-\ell+1}}{2^{1+\alpha_0}}
\le a_0 2^{k-\ell}$, which gives the desired step of the induction.
\end{proof}

We next give the 
\begin{corollary}[``improvement of flatness'']\label{NL:VV}
Assume that $L$ satisfies $\eqref{assumpL}$ and that
$f$ satisfies \eqref{assumpf} and \eqref{existslayer}.
Then there exist universal constants~$\alpha_0\in (0,s/2)$, $
p_0\in(2,\infty)$ and~$a_0\in(0,1/4)$ such that the following statement holds.
\smallskip

Let $u: \R^n \rightarrow (-1,1)$ be a solution 
of $Lu=f(u)$ in $\R^n$, such 
that $0\in \{-1+\kappa \le u \le  1-\kappa\}$.
Let~$k$, $l\in\N$ be such that
\begin{equation}\label{CO:ASS}
l\le \frac{k}{\alpha_0 p_0+1}+1 +\frac{p_0 \log a_0}{(\alpha_0 p_0+1)\log2}.
\end{equation}
Assume that
\begin{equation}\label{COR:S1}
\{\omega_j \cdot x\le -a_0 2^{j}\}\, \subset\,
\{u\le -1+\kappa\} \,\subset \,\{u\le 1-\kappa\}
\,\subset \,\{\omega_j\cdot x\le a_0 2^{j}\} \quad \mbox{in }B_{2^j}, 
\end{equation}
for every~$j\ge k$, where  $\omega_j\in S^{n-1}$.

Then, for every~$i\in\{0,\dots,l\}$, it holds that
\begin{equation}\label{COR:S2} 
\left\{\omega_i\cdot x \le - 
\frac{a_0\,2^{k-i}}{2^{\alpha_0\, i}} \right\} \,\subset \,\{u\le -1+\kappa\}
\,\subset\, \{u\le 1-\kappa\}  \,\subset\,
\left\{\omega_i\cdot x\le  \frac{a_0\,2^{k-i}}{2^{\alpha_0 i}} \right\} 
\quad \mbox{in }B_{2^{k-i}},
\end{equation}
for some $\omega_i\in S^{n-1}$.
\end{corollary}

\begin{proof} The proof is by induction over~$i$.
When~$i=0$, we have that~\eqref{COR:S2} follows
from~\eqref{COR:S1}
with~$j=k$.

Now, we assume that~\eqref{COR:S2} holds true for all~$i\in
\{0,\dots,i_0-1\}$, with~$1\le i_0\le l$, and we prove it for~$i_0$.
To this aim, we set
$$ \tilde u(x):= u(2^{k-i_0+1}x),\qquad
\tilde\eps:=\frac{1}{2^{k-i_0+1}},\qquad \tilde a:=
\frac{a_0}{2^{\alpha_0(i_0-1)}}.$$
Our goal is to use
Theorem~\ref{thm:improvement} in this setting
(namely, the triple~$(u,\eps,a)$ in the statement of
Theorem~\ref{thm:improvement} becomes here~$(\tilde u,\tilde\eps,\tilde a)$).
For this, we need to check that~$(\tilde u,\tilde\eps,\tilde a)$
satisfy the assumptions of Theorem~\ref{thm:improvement}.
First of all, we notice that~$\tilde a\le a_0$ and
\begin{equation}\label{CO:INTERFACE:0} \frac{\tilde\eps}{\tilde a^{p_0}}
=\frac{2^{\alpha_0p_0(i_0-1)}}{a_0^{p_0}\, 2^{k-i_0+1}}
=\frac{2^{(\alpha_0p_0+1) i_0}}{a_0^{p_0}\, 2^{\alpha_0p_0+k+1}}\le
\frac{2^{(\alpha_0p_0+1) l}}{a_0^{p_0}\, 2^{\alpha_0p_0+k+1}}\le 1,
\end{equation}
thanks to~\eqref{CO:ASS}.

Now we claim that, for any~$j\ge0$,
\begin{equation}\label{CO:INTERFACE}
{\mbox{the interface of~$\tilde u$ in~$B_{2^j}$ is trapped
in a slab of width~$\tilde a 2^{j(1+\alpha_0)} $.}}
\end{equation}
For this, we distinguish two cases, either~$j\ge i_0$ or~$j\in\{0,\dots,i_0-1\}$.
In the first case,
we take~$x\in B_{2^j}$ belonging to the interface of~$\tilde u$, and we observe
that~$y:=2^{k-i_0+1}x\in B_{
2^{j+k-i_0+1}}$ belongs to the interface of~$u$: then, we can use~\eqref{COR:S1}
and find that~$y$ is trapped in a slab of size
$$ a_0 2^{j+k-i_0+1}=
\tilde a 2^{\alpha_0(i_0-1)+j+k-i_0+1}.
$$
Scaling back, this says that~$x$ is trapped in a slab of size
$$ \tilde a 2^{\alpha_0(i_0-1)+j}\le 
\tilde a 2^{\alpha_0(j-1)+j}\le \tilde a 2^{j(1+\alpha_0)}.$$
This proves~\eqref{CO:INTERFACE} in this case, and now we focus
on the case in which~$j\in\{0,\dots,i_0-1\}$.
For this, let us take~$x\in B_{2^j}$ in the interface of~$\tilde u$. Then,
we have that~$y:=2^{k-i_0+1}x\in B_{
2^{j+k-i_0+1}}= B_{2^{k-(i_0- j-1)}}$ belongs to the interface of~$u$
and hence, in view of the inductive assumption, is trapped
in a slab of width
$$\frac{a_0 \, 2^{k-(i_0- j-1)}}{ 2^{\alpha_0(i_0- j-1)} }
= \tilde a2^{\alpha_0 j+k-i_0+1+ j} .
$$
Thus, scaling back, we find that~$x$ is trapped
in a slab of width~$\tilde a2^{\alpha_0 j+ j}$, which establishes~\eqref{CO:INTERFACE}.

In light of~\eqref{CO:INTERFACE:0}
and~\eqref{CO:INTERFACE}, we can apply 
Theorem~\ref{thm:improvement} (with~$(u,\eps,a)$
replaced here by~$(\tilde u,\tilde\eps,\tilde a)$): in this way, we conclude that
the interface of~$\tilde u$ in~$B_{1/2}$ is trapped in a slab of
width~$\frac{\tilde a}{2^{1+\alpha_0}}$.
That is, scaling back,
the interface of~$u$ in~$B_{2^{k-i_0}}$ is trapped in a slab of width
$$\frac{\tilde a\,2^{k-i_0+1}}{2^{1+\alpha_0}}= 
\frac{a_0}{2^{\alpha_0(i_0-1)}}\cdot\frac{2^{k-i_0+1}}{2^{1+\alpha_0}}
=a_0\,2^{k-i_0-\alpha_0 i_0},$$
which is~\eqref{COR:S2} for~$i_0$. This completes the inductive step.
\end{proof}

For the proof of Theorem~\ref{C:1},
it is also useful to have the following maximum principle:

\begin{lemma}\label{VA5RP}
Assume that~$w$ is continuous
and bounded from below,
and satisfies, in the viscosity sense,~$Lw\ge -cw$ in~$\{w<0\}$, for some~$c>0$.
Then~$w\ge0$ in~$\R^n$.
\end{lemma}

\begin{proof} Assume, by contradiction, that~$\{w<0\}\ne\varnothing$.
Then, up to a translation, we may assume that~$w(0)<0$.
Let also~$C_o\ge0$ be such that~$w\ge -C_o$ in~$\R^n$.
Fix~$\eta\in C^\infty(\R^n,\,[0,1])$ with~$\eta=0$ in~$B_{1/2}$
and~$\eta=1$ in~$\R^n\setminus B_1$. For any~$\delta>0$,
we define
$$ w_\delta(x):=w(x)+C_o\eta(\delta x).$$
Notice that
\begin{equation}\label{w d a}
\inf_{\R^n} w_\delta \le w(0)+C_o\eta(0)=w(0)<0.\end{equation}
Moreover, if~$x\in \R^n\setminus B_1$, then
$$ w_\delta(x)=w(x)+C_o\ge0.$$
This and~\eqref{w d a} imply that
$$ \inf_{\R^n} w_\delta=\min_{\overline{B_1}} w_\delta=
w_\delta(x_\delta),$$
for a suitable~$x_\delta\in{\overline{B_1}}$.

We remark that~$w_\delta(x_\delta)\le w_\delta(0)=w(0)<0$,
and so~$w(x_\delta)=w_\delta(x_\delta)-C_o\eta(\delta x_\delta)<0$.
Hence
$$ 0\ge Lw_\delta(x_\delta)=Lw(x_\delta)
+C_o\,L\big(\eta(\delta x_\delta)\big)\ge -c w(x_\delta)-C\delta^{s},$$
for some~$C>0$.
Consequently,
$$ \inf_{\R^n} w_\delta=
w(x_\delta)+ C_o\eta(\delta x_\delta)
\ge-\frac{C\delta^{s}}{c}+ C_o\eta(\delta x_\delta).$$
That is, for any~$x\in\R^n$,
$$ w(x)+C_o\eta(\delta x) \ge -\frac{C\delta^{s}}{c}+ C_o\eta(\delta x_\delta).$$
Taking limit in~$\delta$, we thus conclude that,
for any~$x\in\R^n$,
$$ w(x)=w(x)+C_o\eta(0) \ge0,$$
against our initial assumption.
\end{proof}

With this, we can now complete the 
proof of Theorem~\ref{C:1}, with the following argument:

\begin{proof}[Proof of Theorem~\ref{C:1}]
{\em Step 1.} We prove that
in an appropriate orthonormal coordinate system we have
\begin{equation}\label{trap4xn}
 \{ x_n \le z_n- C 2^{j(1-\delta)} \}\subset \{u\le -1+\kappa\} \subset \{u\le 1-\kappa\} \subset  \{ x_n \le z_n + C 2^{j(1-\delta)} \}\quad \mbox{in }B_{2^j}(z)
\end{equation}
for all $z\in \{-1+\kappa\le u\le 1-\kappa\}$
and~$j\in\N$, for a suitable~$\delta\in(0,1)$.

Let $a_0>0$ be the constant in 
Theorem~\ref{thm:improvement}.
First we claim that
there exists $k_0\ge 1$ universal 
such that,
for any~$z\in \{-1+\kappa\le u\le 1-\kappa\}$ 
and $k\ge k_0$, we have 
\begin{equation}\label{trap1}
 \left\{ \omega\cdot (x-z) \le - 
a_0\,2^k \right\}\subset 
\{u\le -1+\kappa\} \subset \{u\le 1-\kappa\}
\subset \left\{ \omega\cdot (x-z) \le  a_0\,2^k
\right\}\quad \mbox{in }B_{2^{k}}(z),
\end{equation}
where $
\omega\in S^{n-1}$ may depend on $z$ and $k$.

To prove~\eqref{trap1}, we use~\eqref{ASS-R}, 
to see that, if~$k$ is sufficiently large
(depending on~$a_0$), 
\begin{equation}\label{o0pia}
 \left\{ \omega\cdot x \le - a_02^{k-1} \right\}
\subset \{u\le -1+\kappa\} \subset \{u\le 1-\kappa\}
\subset \left\{ \omega\cdot x \le  a_02^{k-1} \right\}\quad 
\mbox{in } B_{2^{k+1}}
,\end{equation}
for some~$\omega\in S^{n-1}$ possibly depending on~$k$.
Then, if~$k$
is also large enough (depending on $z$) in such a
way that~$|z|\le k$, we can suppose that~$B_{2^{k}}(z)
\subset B_{ 2^{k+1} }$ and
$$ a_02^{k-1}+|z|\le a_02^{k-1}+k \le a_02^{k}.$$
These observations and~\eqref{o0pia} give that,
if~$k$ is sufficiently large, possibly
depending on~$a_0$
and~$z$, then
\begin{equation*}
 \left\{ \omega\cdot (x-z) \le - a_02^{k} \right\}
\subset \{u\le -1+\kappa\} \subset \{u\le 1-\kappa\}
\subset \left\{ \omega\cdot (x-z) \le  a_02^{k} \right\}\quad 
\mbox{in } B_{2^{k}}(z).
\end{equation*}
Hence, in light of Corollary~\ref{0NL:VV}
(centered here at the point~$z$), we can conclude that~\eqref{trap1}
holds true (we stress indeed that
condition~\eqref{97f38jssd4d} gives a universal lower threshold
for the validity of~\eqref{trap1}).

Our goal is now to use~\eqref{trap1} to prove~\eqref{trap4xn}.
For this, we need to pick up
the exponent~$\delta$ in~\eqref{trap4xn}
which will imply the 
``stabilization'' of the direction~$\omega$
from one scale to another.
To this aim, fixed~$j$ large enough,
we take
$$k:= \left\lfloor
\frac{\alpha_0p_0 +1}{\alpha_0 p_0}j+\frac{\log a_0}{\alpha_0\log 2} 
\right\rfloor$$
and~$l:=k-j$. 
We observe that
\begin{equation} \label{m piu}
l\ge
\frac{\alpha_0p_0 +1}{\alpha_0 p_0}j
+\frac{\log a_0}{\alpha_0\log 2} 
-1-j= \frac{j}{\alpha_0p_0}+\frac{\log a_0}{\alpha_0\log 2} 
-1.\end{equation}
In this setting, we have that
$$ l-\frac{k}{\alpha_0 p_0+1}=
\frac{\alpha_0 p_0 k}{\alpha_0 p_0+1}-j\le
\frac{\alpha_0 p_0 }{\alpha_0 p_0+1}\left(
\frac{\alpha_0p_0 +1}{\alpha_0 p_0}j+
\frac{\log a_0}{\alpha_0\log 2} \right)-j
=\frac{p_0\log a_0 }{(\alpha_0 p_0+1)\,\log2}
.$$
This says that~\eqref{CO:ASS} is satisfied.
Also, condition~\eqref{COR:S1} (here, centered at the point~$z$)
follows from~\eqref{trap1}.
Consequently, in view of~\eqref{COR:S2} (centered here at
the point~$z$), we conclude that
the interface of~$u$ in~$B_{2^j}=B_{2^{k-m}}$ is trapped
in a slab of size
$$ \frac{a_0 \, 2^{k-l} }{2^{\alpha_0 m}}=\frac{
a_0\,2^j}{ 2^{\alpha_0 l}} \le
\frac{
a_1\,2^j}{ 2^{\frac{j}{p_0}}}=a_1\,2^{j(1-\delta)},
$$
for some~$a_1>0$,
where~$\delta:=\frac{1}{p_0}$,
and~\eqref{m piu} has been exploited.

In formulas, this says that
\begin{equation}\label{trap4}\begin{split}&
 \{ \omega_{z,j} \cdot  (x-z) \le - a_12^{j(1-\delta)} \}
\subset \{u\le -1+\kappa\} \\&\qquad
\subset \{u\le 1-\kappa\} \subset  
\{ \omega_{z,j} \cdot  (x-z) \le 
a_1 2^{j(1-\delta)} \}\quad \mbox{in }B_{2^j}(z),
\end{split}\end{equation}
for any~$j\ge j_0$ large enough,
for suitable~$\omega_{z,j}\in S^{n-1}$.

Next we improve~\eqref{trap4}
by finding a direction which is
independent of~$j$ and~$z$. For this, we start to
get rid of the dependence of~$j$: namely,
we use~\eqref{trap4}
in two consecutive dyadic scales (say, $j$ and $j+1$) 
and we obtain, similarly as in the proof of 
Lemma \ref{geomseries}, that 
\[
|\omega_{z,j+1}-\omega_{z,j}|\le C 2^{-j\delta} 
.\]
This implies that 
\begin{equation}\label{om:stab:z}
\lim_{j\to+\infty}
\omega_{z,j}=
\omega_{z,\infty},\end{equation}
for each fixed $z$.

We will make this statement
more precise, by showing that the limit is independent of~$z$,
namely we claim that
\begin{equation}\label{om:stab}
\lim_{j\to+\infty}\omega_{z,j}=\omega_\infty,\end{equation}
for some $\omega_\infty\in S^{n-1}$. For this,
we observe that, for any $z$,
$\bar z \in \{-1+\kappa\le u\le 1-\kappa\}$, 
\[
 \{ \omega_{z,j} \cdot  (x-z) \le - a_12^{j(1-\delta)} \} \subset \{u\le -1+\kappa\} \subset \{u\le 1-\kappa\}\subset  \{ \omega_{\bar z,j} \cdot  (x-\bar z) \le  a_12^{j(1-\delta)} \} 
\]
in $B_{2^j}(z)\cap B_{2^j}(\bar z)$, thanks to~\eqref{trap4}.
This implies that 
\[ 
|\omega_{z,j}-\omega_{\bar z,j}| \rightarrow 0\quad \mbox{as }j{\to} \infty.
\]
{F}rom this and~\eqref{om:stab:z}, we deduce~\eqref{om:stab},
as desired.

Let us choose now an orthonormal coordinate system
in which $\omega_\infty=(0,0,\dots,0,1) $. Then, 
\eqref{trap4} and~\eqref{om:stab} imply
that \eqref{trap4xn} holds true
for all $j\ge j_0$ universal.
Also, for~$j< j_0$, \eqref{trap4xn} holds true
simply by choosing~$C$ large enough,
hence we have proved the desired claim in~\eqref{trap4xn} for all~$j\in\N$.
\smallskip

In addition, for our purposes, it is interesting to observe that, as
as consequence of~\eqref{trap4xn}, we have
\begin{equation}\label{pre-sliding}
\{x_n \le G(x')-C\}\subset \{u\le -1+\kappa\}
\subset \{u\le 1-\kappa\}
\subset \{x_n \le G(x')+C\}
\end{equation}
in all of $\R^n$, for some~$G\in {\rm Lip}(\R^{n-1})$ 
with Lipschitz seminorm universally bounded and such that
\begin{equation}\label{pre-sliding:2} 
|G(x')-G(y')|\le \bar C\,(|x'-y'|^{1-\delta}+1),\end{equation}
for a suitable~$\bar C>0$.\medskip

{\em Step 2.} We now use \eqref{pre-sliding}
and a sliding method (which is somehow related to the one in~\cite{soave})
to conclude that $u$ has $1$D symmetry.
Indeed,  given $(e'_o,0) \in S^{n-1}\cap\{x_n=0\}$ and $\epsilon>0$ we consider 
\[u^t(x ):= u(x-et)\]
where
\begin{equation}\label{thee}
  e=(e',e_n) := \frac{ (e'_o,\epsilon)  }{\sqrt{1+\epsilon^2}} .
\end{equation}
Our goal is to prove that 
\begin{equation}\label{goalsliding}
u^t\le u \quad \mbox{in all of }\R^n\quad\mbox{and for all }t>0.
\end{equation}
From the fact that $e'_o$ and $\epsilon$ are arbitrary it will follow immediately that $u=u(x_n)$ is a
$1$D function.

%

To prove~\eqref{goalsliding}, we first observe that,
if we take $t$ large enough (depending on $\epsilon$), 
we have that 
\begin{equation}\label{TRAPPA} 
\{u\le 1-\kappa\}\subset \{u^t\le -1+\kappa\} .\end{equation}
To check this, let~$x\in\{u\le 1-\kappa\}$. Then, by~\eqref{pre-sliding}, 
we know that~$x_n\le G(x')+C$. Hence, in view of~\eqref{pre-sliding:2},
we have that
\begin{eqnarray*}&& (x-et)_n-G((x-et)') +C
= x_n - \frac{\epsilon t}{\sqrt{1+\epsilon^2}}
-G\left(x'-\frac{ e'_o\,t  }{\sqrt{1+\epsilon^2}}\right)
+C\\ &&\qquad\le
G(x')- \frac{\epsilon t}{\sqrt{1+\epsilon^2}}
-G\left(x'-\frac{ e'_o\,t  }{\sqrt{1+\epsilon^2}}\right)
+2C\\&&\qquad
\le \bar C \left[
\left(\frac{ t  }{\sqrt{1+\epsilon^2}}\right)^{1-\delta}+1\right]
- \frac{\epsilon t}{\sqrt{1+\epsilon^2}}
+2C
\le
0 ,\end{eqnarray*}
as long as~$t$ is large enough (possibly in dependence of~$\epsilon$).
Hence, by~\eqref{pre-sliding},
$$ u^t(x)=u(x-et)\le-1+\kappa,$$
that proves~\eqref{TRAPPA}.

Now we define~$I_-:=(-1,-1+\kappa]$ and~$I_+:=[1-\kappa,1)$
and we observe that, for large~$t$,
\begin{equation}\label{TR:2}
{\mbox{if $x\in\R^n$,
and~$u^t(x)\ge u(x)$, then either~$u^t(x)$, $u(x)\in I_-$
or $u^t(x)$, $u(x)\in I_+$.}}
\end{equation}
To prove it, let~$x$ be such that
\begin{equation}\label{POOS:0}
u^t(x)\ge u(x).\end{equation}
We distinguish two cases,
\begin{eqnarray}
\label{POOS:5}
&&{\mbox{either $u(x)\in I_+$,}}\\
\label{POOS:6}
&&{\mbox{or $u(x)\in(-1,1)\setminus I_+$.}}
\end{eqnarray}
If~\eqref{POOS:5} holds, then~\eqref{POOS:0} gives that~$u^t(x)\in I_+$,
and we are done. If instead~\eqref{POOS:6} holds, then~\eqref{TRAPPA}
gives that~$u^t(x)\in I_-$. This and~\eqref{POOS:0}
imply that~$u(x)\in I_-$, and this concludes the proof of~\eqref{TR:2}.

Now we claim that
 \begin{equation}\label{conclusion-sliding:PRIMO}
  u^t\le u\quad \mbox{ for all $t$ large enough 
(possibly in dependence of~$\epsilon$)}. 
 \end{equation}
To prove this, let~$w:=u-u^t$. We claim that
\begin{equation}\label{EQUAZIONE}
Lw \ge -c_\kappa w \quad{\mbox{ in }}\quad\{w\le0\}.
\end{equation}
Indeed, from~\eqref{TR:2} and the monotonicity of~$f$ in~$I_-\cup I_+$
given in~\eqref{assumpf},
we have that, if~$x\in\{w\le0\}=\{u^t\ge u\}$,
$$ -Lw(x)=Lu^t(x)-Lu(x)=f(u^t(x))-f(u(x))
=\int^{u^t(x)}_{u(x)} f'(\tau)\,d\tau\le -c_\kappa\,(u^t(x)-u(x))=c_\kappa w(x),$$
thus establishing~\eqref{EQUAZIONE}.

Then, from~\eqref{EQUAZIONE} and Lemma~\ref{VA5RP},
we deduce that~$w\ge0$.
This concludes the proof of~\eqref{conclusion-sliding:PRIMO}.
 
Now, to complete the proof of~\eqref{goalsliding},
we perform a sliding method to check that~$u^t\le u$
also when~$t$ decreases, up to~$t=0$. To this aim, 
we first check the touching points inside the
tubular neighborhood described by the function~$G$
in~\eqref{pre-sliding}. Namely, we let~$G$ and~$C$
be as in~\eqref{pre-sliding},
we let~$t_0>0$ be a fixed, suitably large,~$t$ for which~\eqref{TR:2}
holds true, and 
we define
\begin{equation}\label{89:12-1} 
C':= C+t_0\,\| \nabla G\|_{L^\infty(\R^{n-1})}.\end{equation}
Let also
\begin{equation}\label{89:12-536475} 
{\mathcal{G}}:=\{x=(x',x_n)\in\R^n
{\mbox{ s.t. }} |x_n-G(x')|\le C'\}.\end{equation} and
the set~${\mathcal{G}}$ is somehow the
cornerstone of the sliding strategy that we follow here,
since
\begin{equation}\label{CORN}
{\mbox{if~$t>0$ and~$u^t\le u$ in~${\mathcal{G}}$,
then~$u^t\le u$ in the whole of~$\R^n$.}}
\end{equation}
Notice that, from the discussion
before~\eqref{89:12-1}, we already know that~$u^t\le u$ in the whole of~$\R^n$
for~$t\ge t_0$, so, to establish~\eqref{CORN},
we can focus on the case~$t\in[0,t_0)$. To this objective,
we claim that~\eqref{TR:2} holds true also in this setting
(we stress that the original
statement in~\eqref{TR:2} was proved only for large~$t$).
To prove it, let~$x$ be such that
\begin{equation}\label{POOS:0BIS}
u^t(x)> u(x).\end{equation}
We distinguish two cases, namely
\begin{eqnarray}
\label{POOS:5BIS}
&&{\mbox{either $u(x)\in I_+$,}}\\
\label{POOS:6BIS}
&&{\mbox{or $u(x)\in(-1,1)\setminus I_+$.}}
\end{eqnarray}
If~\eqref{POOS:5BIS} is satisfied, then~\eqref{POOS:0BIS}
implies that~$u^t(x)$ also lies in~$I_+$, which gives~\eqref{TR:2}.
So, we can focus on the case in which~\eqref{POOS:6BIS}
holds true. Then, from
the assumption in~\eqref{CORN}, we know that~$u^t\le u$
in~${\mathcal{G}}$. This and~\eqref{POOS:0BIS} imply that~$x$
lies outside~${\mathcal{G}}$. This and~\eqref{POOS:6BIS}
give that~$x$ lies below~${\mathcal{G}}$, that is, recalling~\eqref{89:12-536475},
$$ x_n \le G(x')-C'.$$
Hence, in light of~\eqref{89:12-1},
\begin{eqnarray*}&& (x-e t)_n - G((x-et)')
\le x_n -G(x')+t\,\| \nabla G\|_{L^\infty(\R^{n-1})}\\&&\qquad
\le x_n -G(x')+t_0\,\| \nabla G\|_{L^\infty(\R^{n-1})}=
x_n -G(x')+C'-C\le -C.\end{eqnarray*}
This and~\eqref{pre-sliding} imply that~$x-te\in \{ u\le -1+\kappa\}$.
That is~$u^t(x)\in I_-$.
This proves that~\eqref{TR:2}
holds true also in this setting.
{F}rom this and the assumption in~\eqref{CORN},
it follows that~$u^t\le u$, by arguing exactly as
in the proof of~\eqref{conclusion-sliding:PRIMO}.
This completes the proof of~\eqref{CORN}.

Now, in view of~\eqref{CORN}, to complete the proof of~\eqref{goalsliding},
it is enough to show that
\begin{equation}\label{CORN2}
{\mbox{for any~$t>0$, it holds that~$u^t\le u$ in~${\mathcal{G}}$.}}
\end{equation}
To this aim, we let
$$ \bar t:=
\inf\{ t\ge 0 {\mbox{ s.t. }} u^t\le u {\mbox{ in }}{\mathcal{G}}\}.$$
Notice that~$\bar t\le t_0$, thanks to
the discussion
before~\eqref{89:12-1}.
We claim that, in fact,
\begin{equation}\label{bar t 0}
\bar t=0.
\end{equation} To this aim, we
assume, by contradiction, that $\bar t>0$. 
Then, we have that $u^{\bar t}\le u$ in~${\mathcal{G}}$, and
there exists a sequence of
points \begin{equation}\label{78:01:12:a0}
x_j\in{\mathcal{G}}\end{equation} such that~$u(x_j)-u^{\bar t}(x_j)\le 1/j$.
So, we set~$u_j(x):=u(x+x_j)$, $u^{\bar t}_j(x):=
u^{\bar t}(x+x_j)$ and~$w_j(x):=u_j^{\bar t}(x)-u_j(x)$,
and we see that~$w_j(0)\ge- 1/j$, 
$w_j(x)\le0$ for any~$x\in\R^n$ with~$x+x_j\in{\mathcal{G}}$, and
$$ Lw_j(x)= f(u^{\bar t}_j(x))-f(u_j(x))
\qquad{\mbox{in }}\;\R^n.$$
That is, from the Theorem of Ascoli,
passing to the limit as~$j\to+\infty$, we find that
there exist~$\bar u$, $\bar u^{\bar t}$ and~$\bar w$ (which are the locally
uniform limits
of~$u_j$, $u^{\bar t}_j$ and~$w_j$, respectively)
and~$\bar{\mathcal{G}}$ (which is a tubular neighborhood
obtained as the limit of~${\mathcal{G}}-x_j$)
such that~$\bar w(0)=0$ and
$$ \bar u(x-\bar t e)-\bar u(x)=
\bar u^{\bar t}(x)-\bar u(x)=
\bar w(x)\le0$$ for any~$x\in\bar{\mathcal{G}}$.
Consequently, we infer that
\begin{equation}\label{yh782hf}
{\mbox{$\bar w(x)\le0$ for any~$x\in
\R^n$,}}\end{equation} thanks to~\eqref{CORN}
(applied here to~$\bar u$, which solves the equation~$L\bar u=f(\bar u)$).

Notice that
$$ L\bar w=f(\bar u^{\bar t})-f(\bar u)\qquad{\mbox{ in }}\;\R^n$$
and so
$$ L\bar w(0)=f(\bar u^{\bar t}(0))-f(\bar u(0))=0
.$$
This and~\eqref{yh782hf}
imply that~$\bar w$ vanishes identically in~$\R^n$.
As a consequence, for any~$x\in\R^n$,
\begin{equation}\label{0oqppq:Y0} \bar u(x)=\bar u^{\bar t}(x)=\lim_{j\to+\infty}
u^{\bar t}(x+x_j)=\lim_{j\to+\infty} u(x+x_j-e\bar t )
=\lim_{j\to+\infty} u_j(x-e\bar t )=\bar u(x-e\bar t ),\end{equation}
which means that~$\bar u$ is periodic (of period~$\bar t$ in direction~$e$).
Also, from~\eqref{pre-sliding}
and~\eqref{78:01:12:a0}, moving in the vertical
direction, we know that there exists~$\tilde x_j$
that is at distance at most~$2C'$ from~$x_j$ and such that~$u(\tilde x_j)=0$.
So we write~$\tilde x_j=x_j+\hat x_j$, with~$|\hat x_j|\le2C'$,
and we find, up to a subsequence, that~$\hat x_j$ converges to
some~$\hat x$ and
\begin{equation}\label{0oqppq:X0}
0=\lim_{j\to+\infty} u(\tilde x_j)=\lim_{j\to+\infty}
u(x_j+\hat x_j)=\lim_{j\to+\infty}u_j(\hat x_j)=\bar u(\hat x).\end{equation}
We also claim that 
\begin{equation}\label{0oqppq:2}
\{\bar u=0\}\subset \{ x_n \ge -C_o\,(|x'|^{1-\delta}+1)\},
\end{equation}
for some~$C_o>0$,
where~$\delta\in(0,1)$
is as in~\eqref{pre-sliding:2}. To check this,
we use the notation~$x_j=(x'_j,x_{j,n})\in\R^{n-1}\times\R$,
we set~$G_j(x'):=G(x'+x'_j)-x_{j,n}$ and we see
that if~$p\in\{\bar u=0\}$, then, for~$j$ large enough,
we have that~$p\in \{ |u_j|<1-\kappa\}$,
that is~$p+x_j\in \{ |u|<1-\kappa\}
\subset \{x_n\ge G(x')-C\}$, thanks to~\eqref{pre-sliding}.
This gives that~$p_n+x_{j,n}\ge G(p'+x_j')-C$.
Since~$x_j\in{\mathcal{G}}$, we have that~$x_{j,n}-G(x_j')\le C'$.
Hence, recalling~\eqref{pre-sliding:2}, we find that
$$ p_n\ge G(p'+x_j')-x_{j,n}-C\ge 
G(p'+x_j')-G(x_j')-C-C'\ge -\bar C(|p'|^{1-\delta}+1)-
C-C'.$$
This completes the proof of~\eqref{0oqppq:2}.

Now, from~\eqref{0oqppq:Y0}
and~\eqref{0oqppq:X0}, we know that~$\hat x-\ell e\bar t\in\{\bar u=0\}$
for any~$\ell\in\N$. This and~\eqref{0oqppq:2}
imply that~$\hat x-\ell e\bar t\in\{ x_n \ge -C_o\,(|x'|^{1-\delta}+1)\}$,
for any~$\ell\in\N$. That is, recalling~\eqref{thee},
\begin{eqnarray*}
0&\le&\lim_{\ell\to+\infty}
(\hat x-\ell e\bar t)_n + C_o\,\big(|(\hat x-\ell e\bar t)'|^{
1-\delta}+1\big)\\&=&\lim_{\ell\to+\infty}
\hat x_n -\frac{\ell\eps \bar t}{\sqrt{1+\eps^2}}
+C_o\,\left(
\left|\hat x'-\frac{\ell e'_o\,\bar t}{\sqrt{1+\eps^2}}\right|^{
1-\delta}+1
\right)\\ &=&-\infty.
\end{eqnarray*}
This is a contradiction and so~\eqref{bar t 0} is proved.
Notice that~\eqref{bar t 0} implies~\eqref{CORN2},
which in turn implies~\eqref{goalsliding}, thanks to~\eqref{CORN}
 
Finally, from \eqref{goalsliding} we obtain that 
$D_e u \ge 0$ in all of $\R^n$ for all $e$ of the 
form \eqref{thee} where $\epsilon>0$ is arbitrary.

Accordingly, we have that~$D_{(e'_o,0)} u\ge0$
for any~$e'_o\in S^{n-1}\cap\{x_n=0\}$.
Hence, exchanging~$e'_o$ with~$-e'_o$, we obtain that~$
D_{(e'_o,0)} u$ vanishes identically.
It thus follows that $u(x)= u(x_n)$, that is $u$ has $1$D symmetry.
 \end{proof}

\section{Proof of Theorems~\ref{C:2}, \ref{C:3bis}, \ref{C:3} and~\ref{C:3tris}}\label{SECT8}

As a first step towards the proof of Theorems~\ref{C:2}, \ref{C:3bis}, \ref{C:3} and~\ref{C:3tris},
we recall that the limit interface of the minimizers
is a nonlocal minimizing surface.

In the rest of the section, we say that $u$ is a minimizing solution of~$(-\Delta)^{s/2} u=u-u^3$ in~$\R^n$ if $u$ minimizes the energy $\mathcal E$ ---see \eqref{energyfunctional}--- for every bounded domain $\Omega\subset \R^n$

Also, we say that~$E\subset\R^n$
is a $s$-perimeter minimizer in~$\R^n$ if its
characteristic function
is a minimizer for the
functional in~\eqref{KIN} among characteristic functions, that is if~${\mathcal{E}^{\rm Dir}}(\chi_E,B)<+\infty$ and
$$ {\mathcal{E}^{\rm Dir}}(\chi_E,B)\le {\mathcal{E}^{\rm Dir}}(\chi_F,B),$$
for any ball~$B\subset\R^n$ and any~$F\subset\R^n$ such that~$F\setminus B=E\setminus B$.

These nonlocal minimal surfaces have been introduced 
in~\cite{CRS}
and widely studied in the recent literature.
In this setting, we have 
\begin{lemma}[Corollary~1.7 in~\cite{SaV-DE}]\label{TEN}
Let~$u$ be a minimizing solution of~$(-\Delta)^{s/2} u=u-u^3$ in~$\R^n$ with $|u|<1$.
For any $\varepsilon>0$, let~$u_\varepsilon(x):=u(x/\varepsilon)$. Then
there exists a nontrivial set ($E\neq \varnothing, \R^n$) $E\subset\R^n$ which is a minimizer of the $s$ perimeter in~$\R^n$ and, up to a subsequence,
$u_\varepsilon\to \chi_E-\chi_{\R^n\setminus E}$ a.e. in~$\R^n$.
Also, $\{u_\varepsilon \le1-\kappa\}$ and $\{u_\varepsilon \le -1+\kappa\}$  converge locally uniformly to~$\R^n\setminus E$ (in the sense of the Hausdorff distance).
\end{lemma}

By a standard foliation argument, one also sees that monotone
solutions with limits $\pm1 $ are minimizing:
\begin{lemma}[see e.g. Lemma 9.1
in~\cite{VSS}]\label{MIN:MO}
Let~$u$ be a solution of~$(-\Delta)^{s/2} u=u-u^3$
in~$\R^n$.
Suppose that
$$ \frac{\partial u}{\partial x_n}(x)>0\quad{\mbox{ for any }}x\in\R^n$$
and
$$ \lim_{x_n\to\pm \infty} u(x',x_n)=\pm1.$$
Then, $u$ is a minimizing solution.\end{lemma}

We also need a lemma on flatness of nonlocal minimizing surfaces that are known to be contained in a halfspace.
\begin{lemma}\label{lemhalfspace}
Assume that $E$ is a minimizer of the $s$-perimeter that is contained in some halfspace. 
Then, either $E= \varnothing$ or $E$ is a parallel halfspace. 
\end{lemma}
\begin{proof}
Assume that $E $
is contained in  $\{e\cdot x>0\}$, for some direction~$e\in S^{n-1}$. 
After a translation we may assume that $E$ is {\em not} contained in $\{e\cdot x>t\}$ for any $t>0$
Hence, there exists
a sequence of points $x_k \in \partial E$  such that $t_k := e\cdot x_{k}\downarrow 0$.

If the sequence $x_k$ was bounded, say contained in $B_1$ after some dilation, then we may touch $\partial E$ with
huge balls $B_{R_k}((-R_k+t_k') e )\subset \R^n\setminus E$ with $R_k =
\frac {1}{100}(t_k)^{-1/2}$ and~$t'_k$ infinitesimal,
at some point $y_k \in B_2$. Using the viscosity equation for $\partial E$ (see \cite{CRS}), we obtain that
\[
\int_{\R^n} (\chi_E -\chi_{\R^n\setminus E})(x) |x-y_k|^{-n-s}\, dx  
=0.\]
Then we find that, for any fixed~$R_o>0$,
it holds that~$|(\{e\cdot x>0\}\setminus E) \cap B_{R_o}(y_k) | \downarrow 0$ as $k\to \infty$.
This implies  in the limit that $E$ is a halfspace. 

If the sequence $x_k$  (or a subsequence of
it) is divergent we may always rescale $E$ and consider $E_k :=|x_k|^{-1}E$. Note that now $E_k\subset \{e\cdot x>0\}$ and by construction there exists $\tilde x_k := x_k / |x_k| \in \partial B_1$ such that $\tilde t_k := e\cdot \tilde x_{k} = t_k/ |x_k| \downarrow 0$.
Hence, repeating the previous argument of touching $\partial E_k$ with huge balls $B_{R_k}((-R_k+t_k') e )\subset \R^n\setminus E$ with
$R_k= \frac {1}{100}(t_k)^{-1/2}$ at some point $y_k \in B_2$ and using the viscosity equation we obtain $|(\{e\cdot x>0\}\setminus E_k) \cap B_{R_o} (y_k)| \downarrow 0$ as $k\to \infty$.

We have therefore proven that  the blow downs $E_k =|x_k|^{-1}E$
converges to a halfspace.
Then, the improvement of flatness theorem from \cite{CRS}
(see e.g. Lemma~3.1 in~\cite{FIG})
implies that $E$ must be a halfspace.
\end{proof}

With these preliminary results,
we can now complete the proofs of Theorems~\ref{C:2},
\ref{C:3bis}, \ref{C:3} and~\ref{C:3tris}.

\begin{proof}[Proof of Theorems~\ref{C:2} and~\ref{C:3}]
This proof is rather standard and it is not substantially different for the one of the local case (see \cite{Savin}).  
{F}rom Lemma~\ref{TEN},
we know that the level sets of~$u_\varepsilon$
approach locally uniformly~$\partial E$, and~$E$ is $s$-minimal in~$\R^n$.
Then we use either~\cite{SaV-R2}
(in case we are in~$\R^2$ and we want to prove Theorem~\ref{C:2})
or~\cite{CafVal}
(in case we are in~$\R^n$ with~$n\le7$, $s$ is close to~$1$
and we want to prove Theorem~\ref{C:3}) and we see that~$\partial E$ is a hyperplane.

Hence, we are in the setting of Theorem~\ref{C:1}, which implies that~$u$ is~$1$D.
\end{proof}

\begin{proof}[Proof of Theorems~\ref{C:3bis} and~\ref{C:3tris}]
This proof is rather standard and it is not substantially different for the one of the local case (see \cite{Savin}). 
By Lemma~\ref{MIN:MO}
we know that~$u$ is a minimizing solution and 
the level set of~$u_\varepsilon$ approach an $s$-minimal
set~$E$ satisfying $E \subset E-te_n$ for all $t>0$. 
Let us prove that $E$ is a halfspace. 

To do this, we consider the two limit sets $$
E_{+\infty} := \bigcup_{t\in \R} (E-te_n)\;{\mbox{ and }}\;
E_{-\infty} := \bigcap_{t\in \R} (E-te_n)$$ 
which by compactness of $s$-minimizing sets (see \cite{CRS}) are also minimizers.
Note that $E_{-\infty}  \subset E \subset E_{+\infty}$. Let us prove now that $E_{+\infty} = \R^n$ and $E_{-\infty} = \varnothing$. 

Indeed, if one of the two sets, say,  $E_{+\infty}$ is nontrivial then it is a $s$-minimizer that is by construction invariant
under translations in the direction $e_n$.
Thus, its trace in $\R^{n-1}$ is a $s$-minimizer in one dimension less.
Now, when $n=3$ we use the classification of entire minimizers in $\R^2$ of \cite{SaV-R2} to conclude that $E_{+\infty}$ must be a halfspace. 
Similarly, if $n\le8$
and~$s$ is close to~$1$, then the asymptotic results from \cite{CafVal} give that entire minimizers in  $\R^{n-1}$ must be halfspaces. 

We have thus shown that $E_{+\infty}$ is a halfspace (if it is nontrivial) and $E \subset E_{+\infty}$. But then Lemma \ref{lemhalfspace} gives that $E$ must be also a halfspace. 

Similarly, if~$E_{-\infty}$ is nontrivial,
then we conclude that $E$ is a halfspace exactly in the same way.

Thus, it only remains to consider the case in which both $E_{+\infty}$ and $E_{-\infty}$ are trivial. Since $E$ is nontrivial and $E_{-\infty}  \subset E \subset E_{+\infty}$, it follows that~$E_{+\infty}= \R^n$ and  $E_{-\infty}= \emptyset$. 

This implies that $\partial E$ is an entire minimal graph in the direction $x_n$. 
Then, when $n=3$ and we want to prove
Theorem~\ref{C:3bis}, we make use of Corollary 1.3 in~\cite{FIG}.
Similarly, when~$n\le8$, $s$ is close to~$1$ and
we want to prove
Theorem~\ref{C:3tris}, we make use of Theorem~1.2 in~\cite{FIG}
combined with~\cite{CafVal}. In any case, we conclude that~$E$ is a halfspace.

Finally, once we have proven that the $E$
is a halfspace, it follows from Theorem~\ref{C:1}  that $u$ must be $1$D.
\end{proof}

\vfill

\end{document}